\documentclass{article}

\usepackage[final,main]{neurips_2025}

\usepackage[utf8]{inputenc} 
\usepackage[T1]{fontenc}    
\usepackage[hypertexnames=false]{hyperref}
\usepackage{url}            
\usepackage{booktabs}       
\usepackage{amsfonts}       
\usepackage{nicefrac}       
\usepackage{microtype}      
\usepackage{xcolor}         
\usepackage{subcaption}

\usepackage{graphicx}
\usepackage{apptools}
\usepackage[flushleft]{threeparttable}
\usepackage{array,booktabs,makecell}
\usepackage{multirow}
\usepackage{nicefrac}
\usepackage{amsthm}
\usepackage{amsmath,amsfonts,bm}
\usepackage{wrapfig}
\usepackage{caption}
\usepackage{siunitx}
\usepackage{thm-restate}
\usepackage{nccmath}
\usepackage{empheq}
\usepackage{bbm}
\usepackage{tabularx}

\usepackage{algorithmic}
\usepackage{algorithm}

\usepackage{color}
\usepackage{colortbl}
\definecolor{bgcolor}{rgb}{0.76,0.88,0.50}
\definecolor{bgcolor0}{rgb}{0.93,0.99,1}
\definecolor{bgcolor1}{rgb}{0.8,1,1}
\definecolor{bgcolor2}{rgb}{0.8,1,0.8}
\definecolor{bgcolor3}{rgb}{0.50,0.90,0.50}
\usepackage{tcolorbox}
\usepackage{pifont}

\usepackage[scaled=0.86]{helvet}

\newcommand{\norm}[1]{\left\| #1 \right\|}

\newcommand{\abs}[1]{\left| #1 \right|}
\newcommand{\flr}[1]{\left\lfloor #1\right\rfloor} 

\newcommand{\R}{\mathbb{R}} 
\newcommand{\N}{\mathbb{N}} 

\newcommand{\Exp}[1]{{\mathbb{E}}\left[#1\right]}
\newcommand{\ExpSub}[2]{{\mathbb{E}}_{#1}\left[#2\right]}
\newcommand{\ExpCond}[2]{{\mathbb{E}}\left[\left.#1\right\vert#2\right]}

\newcommand{\Prob}[1]{\mathbb{P}\left(#1\right)} 
\newcommand{\ProbCond}[2]{\mathbb{P}\left(#1\middle\vert#2\right)}

\newcommand{\cC}{\mathcal{C}}

\newcommand{\cO}{\mathcal{O}}

\newcommand{\cZ}{\mathcal{Z}}

\theoremstyle{plain}
\newtheorem{theorem}{Theorem}[section]
\newtheorem*{theorem*}{Theorem}

\newtheorem{lemma}[theorem]{Lemma}
\newtheorem{corollary}[theorem]{Corollary}
\theoremstyle{definition}
\newtheorem{definition}[theorem]{Definition}
\newtheorem{assumption}[theorem]{Assumption}

\theoremstyle{remark}

\newcommand{\eqdef}{:=}

\makeatletter
\newcommand{\vast}{\bBigg@{4}}

\usepackage{tcolorbox}
\newtcolorbox{mytheobox}{
  boxrule=0pt,      
  left=1mm,
  right=1mm,
  top=0mm,
  bottom=0mm
}
\usepackage{pifont}
\usepackage{xcolor}
\definecolor{mydarkgreen}{RGB}{39,130,67}
\definecolor{mydarkorange}{RGB}{236,147,14}
\definecolor{mydarkred}{RGB}{236,147,14}
\definecolor{blue}{RGB}{0,0,255}

\DeclareSymbolFont{extraup}{U}{zavm}{m}{n}
\DeclareMathSymbol{\varheart}{\mathalpha}{extraup}{86}
\DeclareMathSymbol{\vardiamond}{\mathalpha}{extraup}{87}

\makeatletter
\newenvironment{protocol}[1][htb]{%
    \renewcommand{\ALG@name}{Protocol}
   \begin{algorithm}[#1]%
  }{\end{algorithm}}
\makeatother



\DeclareSymbolFont{extraup}{U}{zavm}{m}{n}
\DeclareMathSymbol{\varheart}{\mathalpha}{extraup}{86}
\DeclareMathSymbol{\vardiamond}{\mathalpha}{extraup}{87}

\hypersetup{
  colorlinks   = true, 
  urlcolor     = blue, 
  linkcolor    = {red!75!black}, 
  citecolor   = {blue!50!black} 
}

\renewcommand{\eqref}[1]{(\ref{#1})}

\allowdisplaybreaks

\DeclareSymbolFont{extraup}{U}{zavm}{m}{n}
\DeclareMathSymbol{\varheart}{\mathalpha}{extraup}{86}
\DeclareMathSymbol{\vardiamond}{\mathalpha}{extraup}{87}

\title{Optimality in Decentralized Optimization \\ under Bandwidth Constraints}

\author{
    Alexander Tyurin \\ AXXX, Moscow, Russia \\ Applied AI Institute, Moscow, Russia
}

\begin{document}

\maketitle

\begin{abstract}
  We consider a realistic decentralized setup with \emph{bandwidth-constrained communication} and derive \emph{optimal time complexities} for non-convex stochastic parallel and asynchronous optimization (up to logarithmic factors). We develop the corresponding methods, Grace SGD and Leon SGD, for both homogeneous and heterogeneous settings. Unlike previous work, our optimal bounds are characterized in terms of \emph{min-cut/max-flow} quantities and rely on tools from \emph{Gomory--Hu trees} and \emph{Steiner Tree Packing} problems, providing tighter and more practical complexities.
\end{abstract}

\section{Introduction}
\label{sec:intro}
We consider a decentralized distributed optimization setup with $n$ workers, such as GPUs, CPUs, servers, or mobile devices \citep{kairouz2021advances}, that aim to solve a common optimization problem by computing stochastic gradients and sharing this information with each other 
through a communication network. 
We study a smooth nonconvex minimization problem defined as
\begin{align}
\label{eq:main_problem}
\textstyle \min \limits_{x \in \R^d} f(x),
\end{align}
where $f \,:\, \R^d \rightarrow \R$ and $d$ is the dimension of $f$. In this work, we assume that $d$ is large, which is the case in large language models and modern machine learning training \citep{GPT3,grattafiori2024llama}. We begin with the \emph{homogeneous} (i.i.d.) setting, where all workers compute stochastic gradients sampled from the same distribution. But we discuss our implications and also study \emph{heterogeneous} (non-i.i.d.) settings in Sections~\ref{sec:heter} and \ref{sec:heter_sim}.
\begin{assumption}
  \label{ass:lipschitz_constant}
$f$ is differentiable \& $L$--smooth, i.e., $\norm{\nabla f(x) - \nabla f(y)} \leq L \norm{x - y}$, $\forall x, y \in \R^d.$
\end{assumption}
\begin{assumption}
  \label{ass:lower_bound}
  There exist $f^* \in \R$ such that $f(x) \geq f^*$ for all $x \in \R^d$. We define $\Delta \eqdef f(x^0) - f^*,$ where $x^0$ is a starting point.
\end{assumption}
\begin{assumption}[Homogeneous setting]
  \label{ass:stochastic_variance_bounded}
  For all $i \in [n],$ worker $i$ can only calculate $\nabla f(x;\xi)$ and ${\rm \mathbb{E}}_{\xi}[\nabla f(x;\xi)] = \nabla f(x)$ and 
    ${\rm \mathbb{E}}_{\xi}[\|\nabla f(x;\xi) - \nabla f(x)\|^2] \leq \sigma^2$ for all $x \in \R^d,$ where $\sigma^2 \geq 0.$
\end{assumption}
In the nonconvex optimization, the goal is to find an $\varepsilon$--stationary point, a (random) point $\bar{x} \in \R^d$ such that ${\rm \mathbb{E}}[\|\nabla f(\bar{x})\|^2] \leq \varepsilon$ \citep{nemirovskij1983problem}. It is well-known that the optimal \emph{oracle complexity} in this setting is $\Theta(\nicefrac{L \Delta}{\varepsilon} + \nicefrac{\sigma^2 L \Delta}{\varepsilon^2})$ \citep{arjevani2022lower} achieved by the classical SGD method \citep{lan2020first}.

\textbf{Computation times.} 
To present our time complexities and new algorithms, and to compare them with previous results in the distributed setup, we consider the following computation model:
\begin{tcolorbox}[colback=gray!5,colframe=gray,title=Computation Model]
\label{box:computation_model}
We assume that worker $i$ requires $h_i$ seconds\footnotemark to compute a stochastic gradient for all $i \in [n]$.
\end{tcolorbox}
\footnotetext{{We can even assume that the computation times are not fixed to $h_i$ and can vary in the interval $[\bar{c}_{l} h_i, \bar{c}_r h_i],$ where $\bar{c}_{l}$ and $\bar{c}_{r}$ are some constants (e.g., $\bar{c}_{l} = 0.9, \bar{c}_{r} = 1.1$), but this would not change the proved asymptotics.}}
This is a standard and natural assumption in modern optimization \citep{mishchenko2022asynchronous,tyurin2023optimal}, as it allows us to compare parallel and asynchronous methods.

\textbf{Communication times.} 
Our main goal is to derive the optimal time complexity in the setup where communication times cannot be ignored. 
To the best of our knowledge, this work considers a new communication assumption motivated by practical bottlenecks. For example, consider two identical GPUs with the same computation time $h$, connected by a communication link. The main bottleneck is the bandwidth $b$, i.e., the number of bits (or coordinates) per second they can transmit to each other.

\begin{tcolorbox}[colback=gray!5,colframe=gray,title=Graph-Bandwidth Communication Model]
\label{box:comm_model}
Motivated by the previous example, in general, we assume that there exists a directed connected weighted graph $G = (V, E, b)$ with $n = |V|$ vertices and $|E|$ edges. Each vertex $i \in V$ represents a worker, and each ordered edge $(i, j) \in E$ with weight $b_{ij} > 0$ represents a communication link between workers $i$ and $j$. The weight/bandwidth $b_{ij}$ represents the number of coordinates/bits per second that worker $i$ can send to worker $j$ via edge $(i,j).$ We assume that if $(i, j) \in E$, then $(j, i) \in E$ and $b_{ij} = b_{ji}$, meaning that the communication is bidirectional: workers can send and receive messages from each other at the same time. It is also convenient to define an \emph{undirected} version of $G$, i.e., the undirected graph $\bar{G}=(V,\bar E,b)$ where $\{i,j\}\in\bar E$ with weight $b_{ij}$ if and only if $(i,j)\in E$ and $(j,i) \in E$ with weight $b_{ij}$.
\end{tcolorbox}
Without loss of generality (w.l.o.g.), we assume the graph is connected; if not, one should consider each connected component separately. We discuss this model with latencies in Section~\ref{sec:latency}.
\begin{assumption}
\label{ass:time}
The optimization environment satisfies \hyperref[box:computation_model]{\textcolor{gray}{Computation Model}} and \hyperref[box:comm_model]{\textcolor{gray}{Graph-Bandwidth Communication Model}}. All other operations, such as in-node aggregation, vector splitting, and other local computations (except for stochastic gradient computations), are assumed to take negligible time.
\end{assumption}

\emph{We emphasize that our setup allows freedom in how communication is performed and is more flexible than, for instance, the gossip protocol \citep{boyd2006randomized}.} Workers can overlap communication and computation and are allowed to send and receive through all edges asynchronously. 
Examples: \\
\textbf{Split vector:} In Figure~\ref{fig:system_example}, worker $5$ could send a vector of size $d$ to worker $2$ directly in $\nicefrac{d}{b_{25}}$ seconds, but it is faster to split the vector into two parts and route one part through worker $1$, yielding time $\max\{\nicefrac{d}{{\bf 2} b_{25}}, \nicefrac{d}{{\bf 2} \min\{b_{15}, b_{12}\}}\}$, which is two times faster.\\
\textbf{Online in-network aggregation:} Assume that all workers want to \emph{reduce} their vectors at worker $4$. A naive approach that sends the vectors separately is bottlenecked by $b_{45}$, requiring $\nicefrac{{\bf 4}d}{b_{45}}$ seconds. A better strategy aggregates the vectors at worker $5$ and forwards coordinates immediately as they arrive, yielding time $\nicefrac{d}{b_{45}}.$ See another example in Section~\ref{sec:sync}.\\
\textbf{Interleaving:} Worker $4$ needs to send two vectors, one to worker $1$ and one to worker $2$. Instead of sending them sequentially, it may \emph{interleave} transmissions: alternating coordinates of the two vectors. Worker $5$ then forwards each coordinate to the appropriate destination.












\begin{figure}[t]
\centering
\scalebox{0.85}{
\begin{tikzpicture}[
    scale=1,
    every node/.style={font=\small},
    worker/.style={
        circle,
        draw=black,
        thick,
        minimum size=0.4cm,
        fill=blue!8
    },
    link/.style={
        ->,
        thick,
        bend left=6
    },
    tree/.style={
        thick
    }
]

\begin{scope}

\node[worker] (1) at (0,0.6) {$h_1$};
\node[worker] (2) at (2.8,1.3) {$h_2$};
\node[worker] (3) at (5.6,0.6) {$h_3$};
\node[worker] (4) at (4.6,-0.6) {$h_4$};
\node[worker] (5) at (1.4,-0.6) {$h_5$};

\draw[link] (1) to node[midway, above] {$b_{12} = 2$} (2);
\draw[link] (2) to (1);

\draw[link] (2) to node[midway, above] {$b_{23} = 2$} (3);
\draw[link] (3) to (2);

\draw[link] (4) to node[midway, above, yshift=5pt] {$b_{45} = 1$} (5);
\draw[link] (5) to (4);

\draw[link] (5) to node[midway, above, xshift=-15pt,yshift=-10pt] {$b_{15} = 1$} (1);
\draw[link] (1) to (5);

\draw[link] (2) to node[midway, above, xshift=20pt,yshift=-5pt] {$b_{25} = 1$} (5);
\draw[link] (5) to (2);


\end{scope}

\begin{scope}[xshift=7.2cm]

\node[worker] (1t) at (0,0.6) {$h_1$};
\node[worker] (2t) at (2.8,1.3) {$h_2$};
\node[worker] (3t) at (5.6,0.6) {$h_3$};
\node[worker] (4t) at (4.6,-0.6) {$h_4$};
\node[worker] (5t) at (1.4,-0.6) {$h_5$};

\draw[tree] (1t) -- node[midway, above,xshift=-3pt] {$w_{12} = 3$} (2t);
\draw[tree] (2t) -- node[midway, above,xshift=3pt] {$w_{23} = 2$} (3t);
\draw[tree] (2t) -- node[midway, left] {$w_{25} = 2$} (5t);
\draw[tree] (5t) -- node[midway, above] {$w_{45} = 1$} (4t);


\end{scope}

\end{tikzpicture}
}

\caption{
Left: communication graph $G$. Node $i$ has computation time $h_i$; edges denote bidirectional links with bandwidth $b_{ij}=b_{ji}$. Right: a \emph{Gomory--Hu tree} $T$ of the undirected version of $G$ (Definition~\ref{def:gh}), \emph{the central tool for the design of optimal decentralized optimization methods.}
}
\label{fig:system_example}
\end{figure}
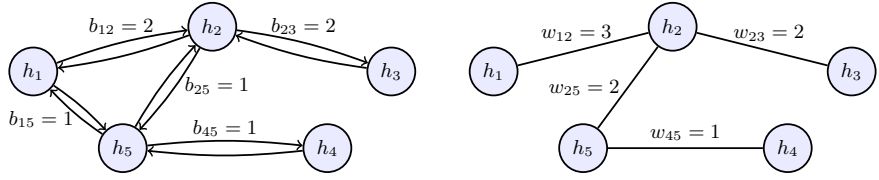


\subsection{Known methods}
\label{sec:known_methods}
We now describe some known methods for decentralized distributed optimization.
Related work under previous setups are discussed in Section~\ref{sec:relatedwork}.\\
\textbf{Synchronous SGD.} One of the most standard ways to solve the decentralized distributed optimization problem is to use Synchronous SGD:
$x^{k+1} = x^k - \frac{\gamma }{n} \sum_{i=1}^{n} \nabla f(x^k;\xi^k_i),$
where $\xi_i^k$ is a random variable sampled from the distribution $\mathcal{D}_{\xi}$ at worker $i$. In this method, all workers compute the gradients in parallel, and then the algorithm averages the gradients and updates the iterate. It is well-known that this method finds an $\varepsilon$--stationary point after $\cO(\nicefrac{L \Delta}{\varepsilon} + \nicefrac{\sigma^2 L \Delta}{n \varepsilon^2})$ \emph{iterations} \citep{lan2020first}. At the same time, under Assumption~\ref{ass:time}, the \emph{time complexity} of this method is
\begin{align}
\label{eq:mFiTJ}
\textstyle \cO\left(\left(\frac{d}{b_{\min}} + h_{\max}\right)\left(\frac{L \Delta}{\varepsilon} + \frac{\sigma^2 L \Delta}{n \varepsilon^2}\right)\right) 
\end{align}
seconds, where $h_{\max} \eqdef \max_{i \in [n]} h_i$ comes from the fact that every worker computes one stochastic gradient, waiting for the slowest one. We define $b_{\min} \eqdef \min_{(i,j) \in E} b_{ij}.$ The term $\nicefrac{d}{b_{\min}}$ is a time to collect all stochastic gradients at a one worker, 
update the iterate, 
and broadcast $x^{k+1}$ to all workers, using the \emph{online in-network aggregation} discussed in Section~\ref{sec:intro} (see details in Section~\ref{sec:sync}) \\
\textbf{Hero SGD.} Another strategy is to simply run SGD on the fastest worker. In this case, the time complexity is
$\cO\big(h_{\min} \left(\nicefrac{L \Delta}{\varepsilon} + \nicefrac{\sigma^2 L \Delta}{\varepsilon^2}\right)\big),$ where $h_{\min} \eqdef \min_{i \in [n]} h_i.$
As expected, this time complexity does not scale with $n$; nevertheless, it can be better than \eqref{eq:mFiTJ} if $\nicefrac{d}{b_{\min}}$ or $h_{\max}$ are large. \\
\textbf{Gossip methods.} Another important family of decentralized methods is gossip-based methods \citep{boyd2006randomized}, where each worker communicates vectors with its neighbors. These methods are typically analyzed by deriving \emph{iteration convergence rates} that depend on the spectral gap of a mixing matrix. 
As for Synchronous SGD and Hero SGD, it is possible to derive the time complexities of these methods; however, as we will show later, we present a \emph{lower bound} (with a matching algorithm) that virtually none of these methods can break. \\
\textbf{Compressed communication.} Instead of sending the full vectors, it is possible to use lossy compression techniques \citep{Seide2014,alistarh2017qsgd}, where the idea is to use different sparsification and quantization methods to reduce the number of coordinates/bits sent through the communication links \citep{beznosikov2020biased}. 

  \textbf{Main problem:} There are many ways, from different centralized and decentralized optimization fields, to construct methods with time complexities that depend on the parameters in Assumptions~\ref{ass:lipschitz_constant}, \ref{ass:lower_bound}, \ref{ass:stochastic_variance_bounded}, and \ref{ass:time}. This naturally raises the central question of the paper: \emph{what is the fastest possible method in this setting, and can we establish a matching lower bound?}
\subsection{Contributions}
$\spadesuit$ Our first main contribution is a new method, Grace SGD (Algorithm~\ref{alg:main}), that achieves an \emph{optimal time complexity}, up to logarithmic factors. To design this method and obtain the discovered complexity, we developed two important new subroutines, Algorithms~\ref{alg:preprocess} and \ref{alg:allreduce}, which use tools from the \emph{Gomory-Hu} trees \citep{gomory1961multi} and the \emph{Steiner Tree Packing} problem \citep{lau2004approximate} fields, and may be of independent interest. We now present the theorem explained in Section~\ref{sec:grace}.
\begin{mytheobox}
\begin{restatable}[Upper Bound in the Homogeneous Setting]{theorem}{THEOREMSGDHOMOG}
    \label{thm:sgd_homog}
    Assume that Assumptions~\ref{ass:lipschitz_constant}, \ref{ass:lower_bound}, and \ref{ass:stochastic_variance_bounded} hold. Then Grace SGD (Algorithm~\ref{alg:main}) finds an $\varepsilon$-stationary point of \eqref{eq:main_problem} after $K = \left\lceil\nicefrac{4 L \Delta}{\varepsilon}\right\rceil$ iterations, and the time complexity under Assumption~\ref{ass:time} is
    \begin{align}
      \label{eq:GvDIFIzv}
      \textstyle \cO\left(\min\limits_{k \in [n]} \left(\frac{d}{\bar{w}_{k}} + \min\limits_{p \in [k]}\min\limits_{m \in [\abs{S_{k,p}}]} \left[\left(\frac{1}{m} \sum\limits_{i=1}^{m} \frac{1}{h_{\pi_{i}(S_{k,p})}}\right)^{-1} \left(1 + \frac{\sigma^2}{m \varepsilon}\right)\right] \right) \frac{L \Delta}{\varepsilon}\right)
    \end{align}
    up to a universal constant factor, where $\{S_{k,p}\}_{k \in [n], p \in [k]}$ and $\{\bar{w}_{k}\}_{k \in [n]}$ are the sequences generated by Algorithm~\ref{alg:preprocess}, and 
    $\pi(S)$ is a permutation that sorts $\{h_i\}_{i \in S}:$ $h_{\pi_1(S)} \leq \dots \leq h_{\pi_{\abs{S}}(S)}.$
\end{restatable}
\begin{corollary}[Equal computation times]
  \label{cor:main}
  In view of Theorem~\ref{thm:sgd_homog}, if $h_i = h$ for all $i \in [n],$ then the time complexity of Grace SGD is
  \begin{align}
      \label{eq:WOjQOMzoROeoVwJxv}
      \textstyle \cO\left(\min\limits_{k \in [n]} \left(\frac{d}{\bar{w}_{k}} + \min\limits_{p \in [k]} \frac{h \sigma^2}{\abs{S_{k,p}} \varepsilon}\right) \frac{L \Delta}{\varepsilon} + \frac{h L \Delta}{\varepsilon}\right).
  \end{align}
\end{corollary}
\end{mytheobox}
The complexities depend on new quantities, $\bar{w}_{k}$ and $S_{k,p}$:
$\bar{w}_k$ is the minimum $s$--$t$ cut value in $G$ for some $s,t \in V$, corresponding to an edge of the Gomory--Hu tree $T$,
and $S_{k,p}$ is a subset of $V$ whose communication bottleneck is at least $\bar{w}_{k}$.
Intuitively, $\bar{w}_k$ reflects communication within $S_{k,p}$, and Grace SGD chooses $k$ to balance $\bar{w}_k$ and the number of workers $\max_{p \in [k]} \abs{S_{k,p}}$: 
the more workers compute stochastic gradients, the longer synchronization takes, and \eqref{eq:GvDIFIzv} and \eqref{eq:WOjQOMzoROeoVwJxv} capture the balance.

$\clubsuit$ While Theorem~\ref{thm:sgd_homog} may appear non-intuitive and semi-explicit, it is in fact \emph{fundamental and optimal up to logarithmic factors}. Despite the flexibility and generality of the setup, it is possible to prove a matching lower bound using new techniques for a broad family of methods (zero-respecting), including standard methods such as SGD, Asynchronous SGD \citep{recht2011hogwild}, Local SGD \citep{zinkevich2010parallelized}, and Adam \citep{kingma2014adam}, as well as gossip methods and compressed methods via random sparsification. See Theorem~\ref{thm:main} in Section~\ref{sec:lowerbound}.

$\vardiamond$ In Section~\ref{sec:heter}, we also consider the \emph{heterogeneous} setting, where workers have access to different distributions, and prove a new \emph{optimal time complexity}, achieved by a new method called Leon SGD: 
\begin{algorithm}[t]
\caption{Grace SGD}
\label{alg:main}
\begin{algorithmic}[1]
\REQUIRE Graph $G=(V,E,b),$ times $\{h_i\}_{i \in [n]},$ ratio $\nicefrac{\sigma^2}{\varepsilon},$ smoothness constant $L,$ point $x^0$
\STATE Find the fastest subset of workers $S^*$ using Algorithm~\ref{alg:preprocess} \hfill (\textbf{crucial new step})
\STATE Workers $S^*$ locally initialize starting point with $x^0$ and step size $\gamma = \nicefrac{1}{2 L}$
\FOR{$k = 0, 1, \ldots$}
\STATE Workers in $S^*$ compute $\sum_{j=1}^{B_i} \nabla f(x^k;\xi^k_{ij})$ in parallel, increasing local batch sizes $B_i$ over time, until $\sum_{i \in S^*} B_i = \max\{\lceil\nicefrac{\sigma^2}{\varepsilon}\rceil, 1\}$
\STATE Run optimal-bandwidth AllReduce (e.g., Algorithm~\ref{alg:allreduce}) so workers \hfill (\textbf{crucial new step}) \\
in $S^*$ obtain $\sum_{i \in S^*} \sum_{j=1}^{B_i} \nabla f(x^k;\xi^k_{ij})$
\STATE Every worker in $S^*$ updates the iterate: $x^{k+1} = x^k - \frac{\gamma}{\sum_{i \in S^*} B_i} \sum_{i \in S^*} \sum_{j=1}^{B_i} \nabla f(x^k;\xi^k_{ij})$
\ENDFOR
\end{algorithmic}
\end{algorithm}
\begin{mytheobox}
\begin{restatable}[Upper Bound in the Heterogeneous Setting]{theorem}{THEOREMSGDHETER}
    \label{thm:sgd_heter}
    Assume that Assumptions~\ref{ass:lipschitz_constant}, \ref{ass:lower_bound}, and \ref{ass:stochastic_variance_bounded_heter} hold. Then Leon SGD (Algorithm~\ref{alg:main_leon}) finds an $\varepsilon$-stationary point of \eqref{eq:main_heter} after $K = \Theta\left(\nicefrac{L \Delta}{\varepsilon}\right)$ iterations, and the time complexity under Assumption~\ref{ass:time} is
    \begin{align}
      \label{eq:GvDIFIzvheter}
      \textstyle \cO\left(\max\left\{\frac{d}{\bar{w}_1}, \max\limits_{i \in [n]} h_i, \frac{\sigma^2}{n \varepsilon}\left(\frac{1}{n} \sum\limits_{i=1}^n h_i\right)\right\} \frac{L \Delta}{\varepsilon}\right)
    \end{align}
    seconds, where 
    $\bar{w}_1 \equiv \min_{\{i,j\} \in F} w_{ij}$ is the minimum edge weight in the Gomory--Hu tree $T = (V, F, w)$ of $\bar{G}$. Equivalently, $\bar{w}_1$ is the global minimum cut value of $G$.
\end{restatable}
\end{mytheobox}
Unlike the homogeneous setting, Theorem~\ref{thm:sgd_heter} always depends on $\bar{w}_1$, the global minimum cut value, reflecting that the heterogeneous setting is inherently more challenging and is bottlenecked by the communication rate through the smallest cut\footnote{For instance, in Figure~\ref{fig:system_example}, the smallest communication bottleneck $\bar{w}_1$ occurs between workers $5$ and $4$.}. This result is tight by our lower bound in Theorem~\ref{thm:main_heter}.

$\varheart$ In Section~\ref{sec:heter_sim}, we demonstrate an important corollary: for \emph{sparse} graphs, Leon SGD and a naive mini-batch version of Synchronous SGD (Section~\ref{sec:known_methods}) with complexity $\tilde{\Theta}(\nicefrac{d L \Delta}{b \varepsilon} + \nicefrac{h \sigma^2 L \Delta}{n \varepsilon^2} + \nicefrac{h L \Delta}{\varepsilon})$ are optimal (up to logarithmic factors) when $\varepsilon$ is small and $n$ is large, even in the homogeneous setting. This means that a fully synchronized method that sends all stochastic vectors to one node is sufficient and \emph{the first communication term does not scale with} $n$ in sparse graphs. The main difficulty here was to show that this holds even in the homogeneous setting. In Corollaries~\ref{cor:cor_sparse_three} and \ref{cor:cor_sparse_two}, we extend this observation and derive a \emph{fundamental trade-off}. In Sections~\ref{sec:examples} and \ref{sec:examples_heter}, we discuss several other important examples and implications of our results for modern distributed optimization tasks.

\section{Grace SGD: A New Algorithm with Near Optimal Time Complexity}
\label{sec:grace}
We now present our new Algorithm~\ref{alg:main}, Grace SGD, designed for the practical \hyperref[box:comm_model]{\textcolor{gray}{Graph-Bandwidth Communication Model}} setup and achieving a near-optimal time complexity. Let us now explain how Grace SGD works. In a nutshell, Grace SGD is very simple: i) it finds the right subset of workers $S^* \subseteq [n]$ (Algorithm~\ref{alg:preprocess}); ii) asks every worker $i \in S^*$ to calculate a mini-batch of stochastic gradients of size $B_i$ such that the total batch size $\sum_{i \in S^*} B_i = \Theta(\nicefrac{\sigma^2}{\varepsilon});$ iii) runs an optimal-bandwidth AllReduce algorithm (e.g. Algorithm~\ref{alg:allreduce}) so that all workers from $S^*$ receive the whole batch $g^k = \sum_{i \in S^*} \sum_{j=1}^{B_i} \nabla f(x^k;\xi^k_{ij}) / \sum_{i \in S^*} B_i$; iv) every worker locally updates the iterate using the standard SGD step $x^{k+1} = x^k - \gamma g^k.$ However, when it comes to the details, the main novelty lies in choosing $S^*,$ implementing AllReduce in graph $G$ in a right way, and proving that this whole scheme is indeed near-optimal. 
For clarity, assume that $h_i = h$ for all $i \in [n]$ (Corollary~\ref{cor:main}).

\textbf{Choosing the best subset $S^*$ (See visualization in Section~\ref{sec:example_alg1}).} Recall the discussion of Synchronous SGD and Hero SGD, two diametrically opposed methods, neither of which has universally best theoretical guarantees. The main advantage of Synchronous SGD is that it utilizes all workers; however, this also introduces its main drawback: communication among all workers can become a significant bottleneck. On the other hand, Hero SGD does not communicate but utilizes only one worker. Here comes our question: \emph{How to choose a subset of workers that is large enough to cope with noises in gradients, but also does not require significant communication times? And, surprisingly, the answer lies in \emph{Gomory--Hu} trees}, which naturally capture the \emph{connectivity structure and bottlenecks} of $G$.

Using our \emph{directed} graph $G,$ we construct the \emph{undirected} version $\bar{G}=(V,\bar E,b).$ 
For undirected graphs, \citet{gomory1961multi} proved that it is always possible to construct a tree $T=(V,F,w)$ (with the same set of vertices, but different sets of edges and weights). The Gomory--Hu tree $T$ has one essential property. Consider any edge $e \in F$ of the tree with weight $w_e.$ Then $e$ separates the vertices of $T$ into two sets, $S_1$ and $S_2$. 
It turns out that the value $w_e$ is an upper bound on the \emph{minimum $s$--$t$ cut value} for all $s \in S_1$ and $t \in S_2$ (Theorem~\ref{thm:gomory_hu}). 
By the max-flow min-cut theorem, $w_e$ is an upper bound on the number of coordinates per second (maximal flow) that $S_1$ and $S_2$ can transmit to each other. This immediately implies that if group $S_1$ or any worker from $S_1$ wants to send a vector of size $\ell$ to group $S_2$ or any worker from $S_2$, then it is necessary to wait at least $\nicefrac{\ell}{w_e}$ seconds, \emph{no matter what routing strategy one chooses.} Moreover, this communication rate is attained by at least one pair of workers $s \in S_1$ and $t \in S_2$, namely the endpoints of the edge $e$ in $T$.
\begin{algorithm}[t]
\caption{Find Fastest Subset of Workers (Preprocessing in Grace SGD)}
\label{alg:preprocess}
\begin{algorithmic}[1]
\REQUIRE Graph $G=(V,E,b),$ times $\{h_i\}_{i \in [n]},$ ratio $\nicefrac{\sigma^2}{\varepsilon}$
\STATE Construct Gomory-Hu tree $T=(V,F,w)$ of the undirected version $\bar{G}$ 
\STATE Sort $\{w_{ij}\}_{\{i,j\} \in F}$ and obtain $\bar{w}_{1} \leq \dots \leq \bar{w}_{n-1};$ also define $\bar{w}_{n} \eqdef \infty$ \hfill ($\abs{F} = n - 1$)
\STATE Init $t^* = \infty,$ $S^* = \textnormal{null},$ and $k^* = \textnormal{null}$
\FOR{$k=1,\dots,n$}
    \STATE Find connected components in $T$ and denote them by $\mathcal{S}_k = (S_{k,1}, \dots, S_{k,k})$
    \STATE Find $\textstyle \bar{S} = \arg\min\limits_{S \in \mathcal{S}_k} \left\{t_k(S) \eqdef \frac{d}{\bar{w}_k} + \min\limits_{m \in [\abs{S}]} \left[\left(\frac{1}{m} \sum\limits_{i=1}^{m} \frac{1}{h_{\pi_{i}(S)}}\right)^{-1} \left(1 + \frac{\sigma^2}{\varepsilon m}\right)\right]\right\},$ \\
    where $\pi(S)$ is a permutation that sorts $\{h_i\}_{i \in S}:$ $h_{\pi_1(S)} \leq \dots \leq h_{\pi_{\abs{S}}(S)}$ 
    \STATE If $t_k(\bar{S}) < t^*,$ then $t^* = t_k(\bar{S}),$ $S^* = \bar{S},$ and $k^* = k$
    \STATE Remove the edge $\{i,j\}$ in $T$ corresponding to $\bar{w}_k$ (except for $k = n$)
\ENDFOR
\RETURN Subset of workers $S^*$ 
\end{algorithmic}
\end{algorithm}

Using this observation, we design the iterative Algorithm~\ref{alg:preprocess}. In the first iteration, we take $T$ and find all its connected components. Since $T$ is a tree, there is only one connected component, $S_{1,1} = [n]$. The idea of the next steps of the loop is to measure the time complexity of the method when all workers in $S_{1,1} = [n]$ solve \eqref{eq:main_problem} and communicate to each other. When $h_i = h$ for all $i \in [n],$ the time complexity is
  $\textstyle t_1(S_{1,1}) \times \nicefrac{L \Delta}{\varepsilon} = \Theta\left(\nicefrac{d}{\bar{w}_1} + h \max\left\{\nicefrac{\sigma^2}{n \varepsilon}, 1\right\}\right) \nicefrac{L \Delta}{\varepsilon},$
where $h \max\left\{\nicefrac{\sigma^2}{n \varepsilon}, 1\right\}$ is the ``statistical'' term, which improves with $n,$ and $\nicefrac{d}{\bar{w}_1}$ is the ``communication'' term. Notice that $\bar{w}_1$ is the smallest value among the weights in $F$ of the tree $T,$ which characterizes the slowest communication bottleneck and communication speed, and appears because all workers participate in the optimization. The term $\nicefrac{d}{\bar{w}_1}$ can be huge and much larger than $\nicefrac{h \sigma^2}{n \varepsilon}.$ 

The main idea of the last step in the loop is to remove the edge $\{i,j\}$ corresponding to $\bar{w}_1$ from $T$, consider the new graph $T - \{i,j\}$, which has two connected components, $S_{2,1}$ and $S_{2,2}$, and consider an optimization process in which only one of these subsets is used. Then, choose the best subset using the values
  $\textstyle t_2(S_{2,1}) = \nicefrac{d}{\bar{w}_2} + h \max\left\{\nicefrac{\sigma^2}{\abs{S_{2,1}} \varepsilon}, 1\right\} \textnormal{ vs. } t_2(S_{2,2}) = \nicefrac{d}{\bar{w}_2} + h \max\left\{\nicefrac{\sigma^2}{\abs{S_{2,2}} \varepsilon}, 1\right\}.$
Basically, when $h_i = h$ for all $i \in [n]$, we have to choose the largest connected component. \emph{Why is that a good strategy?} By removing the edge with weight $\bar{w}_1$, we disconnect the groups $S_{2,1}$ and $S_{2,2}$. However, due to the properties of the Gomory--Hu tree $T$, the communication bottleneck inside each group becomes greater than or equal to $\bar{w}_2$ (instead of $\bar{w}_1$). An important observation is that the ``communication'' term $\nicefrac{d}{\bar{w}_2} \leq \nicefrac{d}{\bar{w}_1}$; thus, the new time complexity $\min\{t_2(S_{2,1}), t_2(S_{2,2})\}$ can be smaller than $t_1(S_{1,1})$. We repeat this loop until there are $n$ connected components (singletons) $\{1\}, \dots, \{n\},$ which have the time complexity of Hero SGD:
  $\textstyle t_n(S_{n,i}) \times \nicefrac{L \Delta}{\varepsilon} = \Theta\left(h \max\left\{\nicefrac{\sigma^2}{\varepsilon}, 1\right\}\right) \nicefrac{L \Delta}{\varepsilon},$
where $\nicefrac{d}{\bar{w}_n} = 0$ because $\bar{w}_n = \infty$, no communication is required. Using the described strategy, we can find the optimal subset of workers that obtains the time complexity
\begin{align}
  \label{eq:XUbneEloAwYnwVNAurq}
  \textstyle \eqref{eq:WOjQOMzoROeoVwJxv} = t_{k^*}(S^*) \times \frac{L \Delta}{\varepsilon} = \Theta\left(\frac{d}{\bar{w}_{k^*}} + h \max\left\{\frac{\sigma^2}{\abs{S^*} \varepsilon}, 1\right\}\right) \frac{L \Delta}{\varepsilon},
\end{align}
where, roughly speaking, the communication and statistical terms are almost equal. It remains to design a method that indeed achieves this complexity.

\textbf{Optimal-bandwidth AllReduce (See visualization in Section~\ref{sec:example_alg2}).} We have selected the subset of workers $S^*.$ Obtaining the statistical term in \eqref{eq:XUbneEloAwYnwVNAurq} can be easily achieved with Minibatch SGD and $\abs{S^*}$ workers. The main difficulty is achieving the communication term $\nicefrac{d}{\bar{w}_{k^*}}.$ 
Intuitively, since the communication bottleneck in $S^*$ is at least $\bar{w}_{k^*}$, formally, the minimum $S^*$-cut value is at least $\bar{w}_{k^*}$ (Definition~\ref{def:min_cut}), there should be a way to develop an efficient synchronization algorithm with complexity $\cO(\nicefrac{d}{\bar{w}_{k^*}}).$
We now implement an optimal-bandwidth AllReduce algorithm, Algorithm~\ref{alg:allreduce}, that achieves this complexity. 

W.l.o.g., assume that 
$\{b_{ij}\}$ 
are integers, and instead of the undirected graph $\bar{G}$, consider the unweighted multigraph $\hat{G}$ (Algorithm~\ref{alg:allreduce}).
That is, instead of considering one edge $\{i,j\}$ with weight $b_{ij}$, we consider an equivalent graph where $\{i,j\}$ is repeated $b_{ij}$ times, each with unit bandwidth. In practice, the system still has a single edge with bandwidth $b_{ij}$; however, the behavior of $b_{ij}$ parallel unit-bandwidth edges can be simulated by multiplexing transmissions using the \emph{interleaving} strategy described in Section~\ref{sec:intro}.
\begin{algorithm}[t]
\caption{Optimal-Bandwidth AllReduce (AllReduce in Grace SGD and Leon SGD)}
\label{alg:allreduce}
\begin{algorithmic}[1]
\REQUIRE Graph $G=(V,E,b),$ with $\{b_{ij}\}_{\{i,j\} \in \bar{E}}$ being integers (w.l.o.g.)\textsuperscript{\color{blue} (a)}, subset of workers $S \subseteq V$, vector $a_i$ stored at worker $i$ for all $i \in S$ \\
\hspace*{-\algorithmicindent} {\textcolor{gray}{\bf (Preprocessing; performed once)}}
\STATE Find undirected unweighted version of $\bar{G}$: $\hat{G}=(V,\hat E)$ where, if $\{i,j\}\in \bar E,$ then $\{i,j\}$ is repeated $b_{ij}$ times in $\hat{E}$
\STATE Solve \emph{Steiner Tree Packing} problem for $S \subseteq V$ in $\hat{G}$ using \citep[Theorem 1.2]{lau2004approximate} and obtain collection of edge-disjoint trees $\mathcal{T} \eqdef (\hat{T}_1, \dots, \hat{T}_p)$ that each
connects $S$ \\ ($p$ equals the value of minimum $S$-cut in $\hat{G}$ up to a constant factor) 
\STATE Fix any pivot worker $v \in S$ \\
\hspace*{-\algorithmicindent} {\textcolor{gray}{\bf (End of preprocessing. Beginning of AllReduce)}}
\STATE Worker $i \in S$ divides $a_i$ into $p$ blocks of size $\lceil\nicefrac{d}{p}\rceil:$ $(a_{i,1}, \dots, a_{i,p})$ (pad with zeros if needed)
\STATE For all $j \in [p]$ (in parallel), \emph{reduce} blocks $\{a_{i,j}\}_{i \in S}$ to pivot worker $v$ along tree $\hat{T}_j$ (use \emph{interleaving} described in Sections~\ref{sec:grace} and \ref{sec:intro}). Worker $v$ receives $(\sum_{i \in S} a_{i,1}, \dots, \sum_{i \in S} a_{i,p})$
\STATE For all $j \in [p]$ (in parallel), \emph{broadcast} $\sum_{i \in S} a_{i,j}$ 
from $v$ to all other workers from $S$ along $\hat{T}_j$
\ENSURE Sum $\sum_{i \in S} a_i$ is stored at worker $i$ for all $i \in S$ \\
\end{algorithmic}
\hrule
\vspace{2pt}
\scriptsize {\color{blue} (a)}: otherwise, $b_{ij}$ are rational numbers (or approximated with arbitrary precision) and can always be renormalized by the common denominator
\end{algorithm}

By construction, for all $i,j \in S^*,$ if we consider the path of edges $p$ from $i$ to $j$ in the tree $T,$ then $\min_{e \in p} w_e \ge \bar{w}_{k^*}.$ Using the Gomory--Hu tree properties, we can conclude that the value of a minimum $S^*$-cut (Definition~\ref{def:min_cut}) is greater than or equal to $\bar{w}_{k^*}.$ Using the result of \citet[Theorem 1.2]{lau2004approximate}, there exists a polynomial-time algorithm that can find a collection of edge-disjoint trees $\mathcal{T} \eqdef (\hat{T}_1, \dots, \hat{T}_p)$ with $p = \Theta(\bar{w}_{k^*})$, each connecting $S^*$ in $\hat{G}.$ 

In terms of our problem, it means that there exist $\Theta(\bar{w}_{k^*})$ edge-disjoint communication pipes that can communicate coordinates/bits between the workers in $S^*$ with unit bandwidth. Now, we ask every worker to split its local vector $a_i$ (local minibatch) into $\Theta(\bar{w}_{k^*})$ blocks. Using the edge-disjoint trees, the workers can first aggregate the blocks at one of the workers (the pivot worker), and then, using the same trees, the pivot worker can broadcast the blocks to all workers in $S^*.$ Importantly, the key idea of this algorithm is to find edge-disjoint trees that enable independent, parallel reduce and broadcast operations, implemented via interleaving without congestion. Since every edge in every tree from $\mathcal{T}$ has unit bandwidth, we can conclude that the time required for reduce and broadcast is $\Theta\big(\nicefrac{d}{\bar{w}_{k^*}}\big)$ seconds for blocks of size $\Theta\big(\nicefrac{d}{\bar{w}_{k^*}}\big).$ See details in the proof of Theorem~\ref{thm:allreduce_time}. 

This AllReduce works with an arbitrary graph $G$. In practice, graphs often have a particular structure (e.g., all-to-all graphs or torus graphs), where implementing an optimal-bandwidth AllReduce is more straightforward. See practical guidelines and numerical experiments in Section~\ref{sec:exp}.
We also note that Grace SGD is \emph{asynchronous-friendly and robust to heterogeneous fluctuations in computation}, since when communication times are negligible, it reduces to Rennala SGD, which was proved to be optimal under arbitrary computational dynamics \citep{tyurin2024tighttimecomplexitiesparallel}.

\section{Lower Bound in the Homogeneous Setting}
\label{sec:lowerbound}
In this section, we show that Grace SGD and the result in Theorem~\ref{thm:sgd_homog} are optimal up to logarithmic factors within a large family of optimization methods. 
The reader will see that we also use Algorithm~\ref{alg:preprocess} to state the lower bound, which is not a coincidence. The ``communication side'' of the lower bound proof is constructive and enables the design of Grace SGD.

Obtaining lower bounds in the homogeneous setting is significantly more technically challenging, since the standard trick of placing different blocks of a hard function on distant nodes cannot be used. We consider the standard class of zero-respecting algorithms \citep{nesterov2018lectures,arjevani2022lower}, which includes SGD-like methods, Adam, gossip methods (e.g., \citep{nedic2017achieving}), and even compressed decentralized methods (e.g., Choco-SGD \citep{Koloskova2019-DecentralizedEC-2019}, BEER \citep{zhao2022beer}).
The optimization protocol is flexible and presented in Protocol~\ref{alg:simplified_time_multiple_oracle_protocol}. Every worker runs two parallel loops:
i) in the first computation loop, worker $i$ computes stochastic gradients locally, adds these vectors to local information $I_i$, and the algorithm is allowed to generate the next query point using $I_i$;
ii) in the second communication loop, worker $i$, using $I_i$, prepares a set of coordinates, which can be random and not necessarily of size $d$, and the algorithm sends them using \emph{any routing strategy} that obeys Assumption~\ref{ass:time}.
This way, the protocol allows the use of local steps, minibatching, gossip communication, and compressed communication. The formal description of the allowed ways to prepare coordinates is presented in Assumption~\ref{ass:compressors}, where we allow algorithms to use mappings that do not take into account local information and do not depend on $\{I_i\}_{i \in [n]}$ when choosing the indices in the sparsifiers. Assumption~\ref{ass:compressors} covers sending the full vector, sending a predefined block of the vector as is done in AllReduce algorithms, or even a random subset of coordinates, supporting Rand$K$ or Perm$K$ compressors \citep{szlendak2021permutation}.

\begin{definition}
  We define $\mathcal{F}_{\Delta, L}$ as the set of functions $f : \R^p \to \R$ such that $p \geq 1$, $f$ is $L$-smooth, i.e., $\norm{\nabla f(x) - \nabla f(y)} \leq L \norm{x - y}$ for all $x, y \in \R^p$, and $f(0) - \inf_{x \in \R^p} f(x) \leq \Delta$.
\end{definition}
\begin{mytheobox}
\begin{restatable}[Lower Bound in the Homogeneous Setting]{theorem}{MAINTHEOREM}
  \label{thm:main}
  We consider any computational and communication environment, any graph $G$ satisfying Assumption~\ref{ass:time}. 
  Let $L, \Delta, \varepsilon, n, \sigma^2, d > 0$ be any numbers such that $\nicefrac{L \Delta}{\varepsilon} \geq \bar{c}_1 \log^{10}(n + 1)$ and dimension $d \geq \nicefrac{\bar{c}_3 L \Delta}{\log^3(n + 1) \varepsilon}.$ 
  Consider Protocol~\ref{alg:simplified_time_multiple_oracle_protocol}. 
  For all $i \in [n]$ and $k \geq 0,$ mappings $\cC^k_i$ satisfy Assumption~\ref{ass:compressors}. Then, for any algorithm
  $A,$ 
  there exists a function $f \in \mathcal{F}_{\Delta, L},$ 
  exists an oracle that satisfies Assumption~\ref{ass:stochastic_variance_bounded}, and 
  $\Exp{\inf_{y \in G_t}\norm{\nabla f(y)}^2} > \varepsilon$ 
  for all 
  \begin{align}
  \label{eq:BNjTPNhGYfaixu}
  \textstyle t \leq \frac{\bar{c}_2}{\log^{14} (n + 1)} \min\limits_{k \in [n]} \left(\frac{d}{\bar{w}_{k}} + \min\limits_{p \in [k]}\min\limits_{m \in [\abs{S_{k,p}}]} \left[\left(\frac{1}{m} \sum\limits_{i=1}^{m} \frac{1}{h_{\pi_{i}(S_{k,p})}}\right)^{-1} \left(1 + \frac{\sigma^2}{m \varepsilon}\right)\right] \right) \frac{L \Delta}{\varepsilon},
  \end{align}
  where $G_t$ is the set of all possible points that can be constructed by $A$ up to time $t$ based on $\{I_i\}_{i \in [n]}.$ The quantities $\pi(S),$ $\{S_{k,p}\}_{k \in [n], p \in [k]},$ and $\{\bar{w}_{k}\}_{k \in [n]}$ are defined in Algorithm~\ref{alg:preprocess}, and $\bar{c}_1,$ $\bar{c}_2,$ and $\bar{c}_3$ are universal constants.
\end{restatable}
\end{mytheobox}
Hence, Grace SGD matches \eqref{eq:BNjTPNhGYfaixu} up to the multiplicative factor $\nicefrac{\bar{c}_2}{\log^{14}(n + 1)}$. Asymptotically, ignoring the polylogarithmic factor, Grace SGD is optimal. We believe that the $\log^{14}(n + 1)$ term is an artifact of our proof techniques, and we conjecture that with a slightly more careful analysis around \eqref{eq:hCvVLgndsboCzRbvZrp}, it should be possible to improve this term to $\log^{6}(n + 1)$ with more technical steps. Nevertheless, eliminating the logarithmic factor completely would likely require a different proof strategy. Moreover, following \citet{tyurin2025proving}, Assumption~\ref{ass:compressors} covers only a subfamily of unbiased compressors, including Rand$K$ for all $K \in [d].$ As in \citep{huang2022lower,he2023unbiased}, the lower bound remains valid for the full class of compressed methods if the adversary is allowed to choose a ``worst-case'' compressor. Extending this result to all compressors and non-zero-respecting algorithms is an important future direction. The proof sketch and the full proof are in Sections~\ref{sec:proof_sketch} and~\ref{sec:proof_full}.

\section{Examples with Different Graph Structures in the Homogeneous Setting}
\label{sec:examples}
While Theorem~\ref{thm:sgd_homog} applies to any graph, we illustrate the time complexities for typical graphs and give explicit formulas. For clarity, we assume $h_i = h$ for all $i \in [n]$ (Corollary~\ref{cor:main}).

\textbf{Example: One Worker.} Assume that $n = 1,$ i.e., $G$ is a graph with one node. In this case, Gomory-Hu $T$ is the same as $G,$ $\bar{w}_1 = \infty$ and $S_{1,1} = \{1\}$ in Algorithm~\ref{alg:preprocess}, and $\eqref{eq:WOjQOMzoROeoVwJxv} = \cO\left(\nicefrac{h \sigma^2 L \Delta}{\varepsilon^2} + \nicefrac{h L \Delta}{\varepsilon}\right),$ which restores the classical result of \citet{ghadimi2013stochastic,arjevani2022lower}.

\textbf{Example: Star Graph (Centralized Setting).} Consider another example where workers communicate through a server. It can equivalently be represented by the graph in Figure~\ref{fig:star_graph_gomory_hu}. One can show that a Gomory-Hu tree $T$ of $G$ is shown in Figure~\ref{fig:star_graph_gomory_hu}. In the tree, sorting the weights, we get $\bar{w}_1 = b,$ \dots, $\bar{w}_{n-1} = b,$ $\bar{w}_n = \infty.$ Using this, one can show that the total time complexity is
  $\textstyle \cO\left(\min\left\{\left(\frac{d}{b} + \frac{h \sigma^2}{n \varepsilon}\right) \frac{L \Delta}{\varepsilon} + \frac{h  L \Delta}{\varepsilon}, \frac{h \sigma^2 L \Delta}{\varepsilon^2} + \frac{h  L \Delta}{\varepsilon}\right\}\right)$
(see Section~\ref{sec:star_graph}). Notice that the first term corresponds to the complexity of Synchronous SGD with mini-batching (and Grace SGD), method that communicates the full vectors to and from the server, and the second term corresponds to Hero SGD. This optimality of this result was also proven in \citep{tyurin2025proving}.

\textbf{Example: $p$-Torus.} We now assume that $G$ is a $p$-Torus (Figure~\ref{fig:torus}), a popular architecture in large-scale model training \citep{jouppi2020domain}. A Gomory-Hu tree $T$ presented in Figure~\ref{fig:torus_gomory_hu}. Sorting the weights, we get $\bar{w}_1 = 2 p b,$ \dots, $\bar{w}_{n-1} = 2 p b,$ $\bar{w}_n = \infty.$ Using the same steps as in the previous example, the time complexity of Grace SGD is
  $\textstyle \cO\left(\min\left\{\left(\frac{d}{{\color{orange} p} b} + \frac{h \sigma^2}{n \varepsilon}\right) \frac{L \Delta}{\varepsilon} + \frac{h  L \Delta}{\varepsilon}, \frac{h \sigma^2 L \Delta}{\varepsilon^2} + \frac{h  L \Delta}{\varepsilon}\right\}\right).$
Note that the first term improves with $p$, formalizing why a $p$-Torus is preferable in practice.

\textbf{Example: all-to-all.} In the all-to-all graph $G$, where all nodes are connected to each other, one can similarly derive the time complexity
  $\textstyle \cO\left(\min\left\{\left(\frac{d}{{\color{orange} (n - 1)} b} + \frac{h \sigma^2}{n \varepsilon}\right) \frac{L \Delta}{\varepsilon} + \frac{h L \Delta}{\varepsilon}, \frac{h \sigma^2 L \Delta}{\varepsilon^2} + \frac{h L \Delta}{\varepsilon}\right\}\right),$
where the communication term has the best scaling among the previous examples, which is expected since the all-to-all graph has the best connectivity.

\textbf{Example: $K$ clusters.}
Consider a practical setup (Figure~\ref{fig:k_clusters_ring}), where we have $K$ clusters of $n / K$ workers. The \emph{intra}-communication is fast within each cluster, is all-to-all, and has bandwidth $\infty$. However, the \emph{inter}-communication between the clusters is relatively slow, with bandwidth $b_{\textnormal{slow}}$ (since, for instance, the clusters may be located in different cities). In this case, Theorem~\ref{thm:sgd_homog} yields the time complexity
  $\textstyle \cO\left(\min\left\{\frac{h \sigma^2 L \Delta}{{\color{orange}(n / K)}\varepsilon^2} + \frac{h L \Delta}{\varepsilon}, \frac{d L \Delta}{{\color{orange} b_{\textnormal{slow}}} \varepsilon} + \frac{h \sigma^2 L \Delta}{{\color{orange} n} \varepsilon^2} + \frac{h L \Delta}{\varepsilon}\right\}\right),$
which formalizes the simple idea that either we use one cluster $S^* = [n / K]$ with $n/K$ workers, or we use all clusters ($S^* = [n]$) together but pay the communication cost $\nicefrac{d L \Delta}{\color{orange} b_{\textnormal{slow}}}$ (which may or may not dominate). Notice that we can even consider the case where clusters use different GPUs. In this case, the complexity becomes
  $\textstyle \cO\left(\min\left\{\min\limits_{i \in [K]} h_i\left(1 + \frac{\sigma^2}{(n / K)\varepsilon}\right), \frac{d}{b_{\textnormal{slow}}} + \min\limits_{m \in [K]}\left(\frac{1}{m} \sum\limits_{i=1}^{m} \frac{1}{h_{i}}\right)^{-1} \left(1 + \frac{\sigma^2}{\varepsilon m (n / K)}\right)\right\} \frac{L \Delta}{\varepsilon}\right),$
where $h_1 \leq \dots \leq h_{K}$ are the computation times of the clusters.

\textbf{Example: optimization with switches.} In practice, workers are not directly connected; instead, there are intermediate nodes, switches, that do not compute stochastic gradients and only transmit data between the workers. Our theory even supports this setting by simply setting $h_i < \infty$ if node $i$ is a worker, and $h_i = \infty$ if it is a switch. Then, one can apply Theorem~\ref{thm:sgd_homog} with any graph topology.


\section{Heterogeneous Setting: Optimal Method and Lower Bound}
\label{sec:heter}
\begin{algorithm}[t]
\caption{Leon SGD}
\label{alg:main_leon}
\begin{algorithmic}[1]
\REQUIRE Graph $G=(V,E,b),$ ratio $\nicefrac{\sigma^2}{\varepsilon},$ smoothness constant $L,$ point $x^0$
\STATE All workers $[n]$ locally initialize starting point with $x^0$ and step size $\gamma = \nicefrac{1}{2 L}$
\FOR{$k = 0, 1, \ldots$}
\STATE \label{line:leon_batch} All workers compute $\frac{1}{B_i}\sum_{j=1}^{B_i} \nabla f_i(x^k;\xi^k_{ij}),$ increasing local batch sizes $B_i$ over time,\\ until $\textstyle \left(\frac{1}{n} \sum_{i=1}^n \frac{1}{B_i}\right)^{-1} \geq \max\{\lceil\nicefrac{\sigma^2}{\varepsilon}\rceil, n\} / n$
\STATE Run optimal-bandwidth AllReduce (e.g., Algorithm~\ref{alg:allreduce}) so workers \hfill (\textbf{crucial new step}) \\
in $[n]$ obtain $\frac{1}{n}\sum_{i=1}^n \frac{1}{B_i} \sum_{j=1}^{B_i} \nabla f_i(x^k;\xi^k_{ij})$
\STATE Every worker in $[n]$ updates the iterate: $x^{k+1} = x^k - \frac{\gamma}{n}\sum_{i=1}^n \frac{1}{B_i} \sum_{j=1}^{B_i} \nabla f_i(x^k;\xi^k_{ij})$
\ENDFOR
\end{algorithmic}
\end{algorithm}

We also consider and analyze the \emph{heterogeneous} problem of minimizing
\begin{align}
  \label{eq:main_heter}
  \textstyle \min \limits_{x \in \R^d} \Big\{f(x) \eqdef \frac{1}{n} \sum\limits_{i=1}^n f_i(x)\Big\},
\end{align}
where $f_i\,:\,\R^d \rightarrow \R$ for all $i \in [n].$ Unlike the homogeneous setup, for all $i \in [n],$ worker $i$ can only access stochastic gradients of $f_i$.

\begin{assumption}[Heterogeneous setting]
  \label{ass:stochastic_variance_bounded_heter}
  For all $i \in [n],$ worker $i$ can only calculate $\nabla f_i(x;\xi)$ and ${\rm \mathbb{E}}_{\xi}[\nabla f_i(x;\xi)] = \nabla f_i(x)$ and 
    ${\rm \mathbb{E}}_{\xi}[\|\nabla f_i(x;\xi) - \nabla f_i(x)\|^2] \leq \sigma^2$ for all $x \in \R^d,$ where $\sigma^2 \geq 0.$
\end{assumption}

In Theorem~\ref{thm:sgd_heter}, we show that Leon SGD (Algorithm~\ref{alg:main_leon}) is an optimal method, where the first term \eqref{eq:GvDIFIzvheter} arises from performing AllReduce across all workers, and the second and third terms are the same as in Malenia SGD \citep{tyurin2023optimal}, which is not a coincidence, since the aggregation phase in Line~\ref{line:leon_batch} is the same as there.

Similarly to Section~\ref{sec:lowerbound}, we can provide a lower bound, Theorem~\ref{thm:main_heter}, for the heterogeneous setting. 
\begin{mytheobox}
\begin{restatable}[Lower Bound in the Heterogeneous Setting]{theorem}{MAINTHEOREMHETER}
  \label{thm:main_heter}
  We consider any computational and communication environment satisfying Assumption~\ref{ass:time}. 
  Let $L, \Delta, \varepsilon, n, \sigma^2, d > 0$ be any numbers such that $\nicefrac{L \Delta}{\varepsilon} \geq \bar{c}_1$ and dimension $d \geq \nicefrac{\bar{c}_3 L \Delta}{\varepsilon}.$ 
  Consider Protocol~\ref{alg:simplified_time_multiple_oracle_protocol}. 
  For all $i \in [n]$ and $k \geq 0,$ mappings $\cC^k_i$ satisfy Assumption~\ref{ass:compressors}. Then, for any algorithm $A,$ there exists a function $f \in \mathcal{F}_{\Delta, L},$ 
  exists oracles that satisfy Assumption~\ref{ass:stochastic_variance_bounded_heter}, and 
  $\Exp{\inf_{y \in G_t}\norm{\nabla f(y)}^2} > \varepsilon$ 
  for all 
  \begin{align}
  \label{eq:kLWINMGFknhpudhwgO}
  \textstyle t \leq \bar{c}_2 \times \max\left\{\frac{d}{\bar{w}_1}, \max\limits_{i \in [n]} h_i, \frac{\sigma^2}{n \varepsilon}\left(\frac{1}{n} \sum\limits_{i=1}^n h_i\right)\right\} \frac{L \Delta}{\varepsilon} ,
  \end{align}
  where $G_t$ is the set of all possible points that can be constructed by $A$ up to time $t$ based on $\{I_i\}_{i \in [n]},$ and $\bar{w}_1 \equiv \min_{\{i,j\} \in F} w_{ij}$ is the smallest weight in a Gomory-Hu tree $T = (V, F, w)$ of $\bar{G}.$ The quantities $\bar{c}_1,$ $\bar{c}_2,$ and $\bar{c}_3$ are universal constants.
\end{restatable}
\end{mytheobox}
This result matches Theorem~\ref{thm:sgd_heter} up to a constant factor. Unlike the homogeneous setting, Theorems~\ref{thm:sgd_heter} and \ref{thm:main_heter} are much more pessimistic: i) the time complexity depends on the worst bottleneck $\nicefrac{d}{\bar{w}_1},$ where $\bar{w}_1$ is the smallest weight in the Gomory–Hu tree $T$, or equivalently the minimum value over all cuts in the graph $G$; ii) similarly to \citep{tyurin2023optimal}, the dependence on $\{h_i\}_{i \in [n]}$ is arithmetic-like in \eqref{eq:GvDIFIzvheter} compared to the harmonic-like dependence in \eqref{eq:BNjTPNhGYfaixu}.


\section{Fundamental Trade-offs in Optimization over Graphs}
\label{sec:heter_sim}
We now want to explain that the optimal time complexity \eqref{eq:GvDIFIzvheter}, derived for the heterogeneous setting, cannot be significantly improved when the graph is \emph{sparse}, even in the homogeneous setting. Sparse graphs include graphs with a ring structure and other low-connectivity topologies commonly used in decentralized optimization (see examples in \citep{Koloskova2019-DecentralizedEC-2019}).
\begin{restatable}[Sparse Graphs; Proof in Section~\ref{sec:cor_sparse}]{corollary}{CORSPARSE}
  \label{cor:cor_sparse}
  In view of Theorems~\ref{thm:main} and \ref{thm:sgd_heter}, if i) $h_i = h$ for all $i \in [n];$ ii) the degree of every node in $G$ is bounded by a small number $p = \Theta(1);$ iii) and all edges have the same bandwidth $b.$ When
  $\nicefrac{\sigma^2}{\varepsilon} \geq \Omega(\nicefrac{d}{b h}),$ the optimal time complexity for solving both \eqref{eq:main_problem} and \eqref{eq:main_heter} is 
  $\textstyle \tilde{\Theta}\left(\textstyle \nicefrac{d L \Delta}{b \varepsilon} +  \nicefrac{h \sigma^2 L \Delta}{n \varepsilon^2} + \nicefrac{h L \Delta}{\varepsilon}\right)$
  and it is achieved by Leon SGD.
\end{restatable}
Note that this complexity can also be achieved with a naive mini-batch version of Synchronous SGD: $x^{k+1} = x^k - \frac{\gamma}{n B} \sum_{i=1}^{n} \sum_{j=1}^{B} \nabla f_i(x^k;\xi^k_{ij})$ with the simple communication approach from Section~\ref{sec:sync}. One of the main challenges was to show that this holds even in the homogeneous setting. The following corollary naturally generalizes Corollary~\ref{cor:cor_sparse} and provides additional insight:
\begin{restatable}[Proof in Section~\ref{sec:cor_sparse_two}; see also dual Corollary~\ref{cor:cor_sparse_two}]{corollary}{CORSPARSETHREE}
  \label{cor:cor_sparse_three}
  In view of Theorems~\ref{thm:main}, if $h_i = h$ for all $i \in [n]$ and $b_{ij} = b$ for all $(i,j) \in E,$ then the lower bound in all settings is 
  $\textstyle \tilde{\Omega}\left(\min\left\{\textstyle \min_{m \in \{2, \dots, n\}}\left[\frac{d L \Delta}{{\color{orange} k(m)} b \varepsilon} + \frac{h \sigma^2 L \Delta}{{\color{orange} m} \varepsilon^2}\right] + \frac{h L \Delta}{\varepsilon}, \frac{h L \Delta}{\varepsilon} + \frac{h \sigma^2 L \Delta}{\varepsilon^2}\right\}\right),$ where ${\color{orange} k(m)}$ is the ${\color{orange} m}$-th largest edge degree.
\end{restatable}
Corollaries~\ref{cor:cor_sparse_three} and \ref{cor:cor_sparse_two} are not necessarily tight (the tight bound is given in Theorem~\ref{thm:main}), but they provide important intuition about learning on graphs: \emph{in both homogeneous and heterogeneous settings, the more workers that compute stochastic gradients, the better the scaling in the stochastic term $\nicefrac{h \sigma^2 L \Delta}{{\color{orange} m} \varepsilon^2}$; however, the scaling in the communication term $\nicefrac{d L \Delta}{{\color{orange} k(m)} b \varepsilon}$ may become worse.} Having more workers compute stochastic gradients may lead to slower overall communication due to the low connectivity of some workers, and there exists a golden mean that balances the number of active workers, depending on the structure of the graph and the $d$ vs.\ $\nicefrac{\sigma^2}{\varepsilon}$ regime.

\section{Related Work}
\label{sec:relatedwork}
\textbf{Classical theory.} Starting with one of the seminal works \citep{nemirovskij1983problem}, the optimization field began investigating the optimality of methods. Initially, the focus was mainly on the oracle complexity, e.g., the number of function and gradient evaluations needed to achieve a given accuracy. One of the first results was obtained in the convex setting where it was shown that the optimal complexity for finding an $\varepsilon$-approximate solution is $\Theta(\sqrt{L R^2 / \varepsilon})$ \citep{nesterov1983method,nesterov2018lectures}, where $R$ is the distance between the initial point and an optimal point. For the nonconvex setting, the optimal complexity $\Theta(\nicefrac{L \Delta}{\varepsilon})$ was shown by \citet{carmon2020lower} in the deterministic setting and $\Theta(\nicefrac{L \Delta}{\varepsilon} + \nicefrac{\sigma^2 L \Delta}{\varepsilon^2})$ by \citet{arjevani2022lower} in the stochastic setting.

\textbf{Gossip protocol.} 
One way to extend the classical results to the decentralized setting is to use the gossip protocol. As in \hyperref[box:comm_model]{\textcolor{gray}{Graph-Bandwidth Communication Model}}, workers are associated with nodes and communication links with edges of the graph. Then, a mixing matrix $W \in \R^{n \times n}$ is constructed such that $w_{ij} > 0$ if nodes $i$ and $j$ are connected, and $w_{ij} = 0$ otherwise. The upper and lower bounds are then constructed in terms of the number of communication rounds required to find an $\varepsilon$-stationary point \citep{boyd2006randomized,scaman2017optimal,Koloskova2019-DecentralizedEC-2019,lu2021optimal}. However, this approach arguably has the following drawbacks: (i) this protocol allows methods to communicate only with neighbors and to use a single operation $y_i \leftarrow \sum_j w_{ij} x_j,$ where $x_j$ is the sent vector and $y_i$ is the received vector. Such a protocol restricts many practical ways to disseminate information in the graph. For instance, the split vector and online in-network aggregation strategies, discussed in Section~\ref{sec:intro} and commonly used in practice, are ignored; (ii) the obtained complexities are typically defined in terms of \emph{iteration complexity} and depend on the spectral gap of a mixing matrix, whereas our approach uses \emph{time complexities} with explicit dependence on computation and communication time parameters $\{h_i\}$ and $\{b_{ij}\}$; (iii) the obtained results use ``one-zero'' encoding, assuming that all communication links are the same, whereas our approach uses bandwidths, a practical way of defining connectivity between two workers or switches.

\textbf{Time complexities.} The classical oracle complexity is a valid and intuitive metric for comparing methods with one worker/GPU/CPU/server. However, modern optimization requires hundreds, thousands, or even millions of workers, requiring the analysis of \emph{parallel and asynchronous algorithms}, where a conceptually different metric is needed. Let us consider a recent paper by \citet{mishchenko2022asynchronous}, where the authors compare Synchronous SGD and Asynchronous SGD. The \emph{oracle complexity} of both methods is the same; moreover, the \emph{iteration complexity} of Synchronous SGD is better. 
However, under the \emph{time complexity}, a more appropriate metric for parallel methods, Asynchronous SGD is provably better. In particular, assuming the \hyperref[box:computation_model]{\textcolor{gray}{Computation Model}} and using the analysis by \citet{cohen2021asynchronous,koloskova2022sharper,mishchenko2022asynchronous}, it is possible to show that Asynchronous SGD has provably better time complexity than Synchronous SGD. Further, \citet{tyurin2023optimal} formalized the notion of time complexity and proved that the optimal time complexity is
$\Theta(\textstyle \min_{m \in [n]} [(\nicefrac{1}{m} \sum_{i=1}^{m} \nicefrac{1}{h_{\pi_{i}}})^{-1} (1 + \nicefrac{\sigma^2}{m \varepsilon})]\nicefrac{L \Delta}{\varepsilon}),$
achieved\footnote{It was also achieved by another method called Ringmaster ASGD \citep{maranjyan2025ringmaster}. Moreover, somewhat surprisingly, it was later shown that this complexity, up to a logarithmic factor, can also be achieved by a synchronous method \citep{begunov2026we}.} by a method called Rennala SGD, where $\pi$ is a permutation that sorts $\{h_i\}_{i \in [n]}$.

\textbf{Time complexities with communication times.} A natural question is to determine the optimal time complexities when communication between workers cannot be ignored. One of the early works studying such complexities in the decentralized setting is \citet{scaman2017optimal}, where the complexity was established for heterogeneous objectives, without stochastic gradients, and under the assumption that workers exchange full vectors with their neighbors at every step (gossip communication). At the same time, our \hyperref[box:comm_model]{\textcolor{gray}{Graph-Bandwidth Communication Model}} is much more flexible, while also supporting the homogeneous setup and stochastic gradients. The idea of using bandwidths on the communication links was considered in \citep{tyurin2024shadowheart,tyurin2025proving} (in fact, they define $\tau$s instead as the time to send one coordinate, which are the inverses of the bandwidths). However, an important limitation of these works is that they only consider the centralized setup (star graphs), while ours considers arbitrary graphs. The analysis of communication in star graphs is much easier due to the lack of congestion on the communication links and the uniqueness of the paths between the workers. A work by \citet{tyurin2024optimalgraph} is closest to our setting, since it also proves optimal time complexities in the decentralized setup with both homogeneous and heterogeneous functions and stochastic gradients. However, it assumes that algorithms are only allowed to send full vectors to neighbors, an important limitation that does not capture, for instance, optimal-bandwidth AllReduce strategies.


\bibliography{example_paper}
\bibliographystyle{apalike}

\appendix
\newpage

\tableofcontents

\newpage
\section{Notations}

\begin{table}[h]
\centering
\caption{List of notations used throughout the paper.}
\begin{tabular}{cl}
\toprule
\textbf{Notation} & \textbf{Meaning} \\
\midrule
$[n]$ & Denotes the finite set $\{1, \dots, n\}$. \\
$\norm{x}$ & Euclidean norm of a vector $x$. \\
$g = \cO(f)$ & There exists $C > 0$ such that $g(z) \le C \, f(z)$ for all $z \in \cZ$. \\
$g = \Omega(f)$ & There exists $C > 0$ such that $g(z) \ge C \, f(z)$ for all $z \in \cZ$. \\
$g = \Theta(f)$ & There exist $C_1, C_2 > 0$ such that $C_1 f(z) \le g(z) \le C_2 f(z)$ for all $z \in \cZ$. \\
$\tilde{\cO},\tilde{\Omega},$ and $\tilde{\Theta}$ & The same as $\cO$, $\Omega,$ and $\Theta,$ but up to logarithmic factors. \\
$\Delta$ & Initial optimality gap, $\Delta \eqdef f(x^0) - f^*$. \\
$\Exp{\cdot}$ & Full expectation. \\
$\ExpSub{\xi}{\cdot}$ & Expectation w.r.t. random variable $\xi$. \\
$h_i$ & Time required by worker $i$ to compute one stochastic gradient. \\
$G = (V,E,b)$ & Directed weighted communication graph. \\
$\bar{G} = (V,\bar E,b)$ & Undirected version of the communication graph $G$. \\
$T = (V,F,w)$ & Gomory--Hu tree of the undirected graph $\bar G$. \\
$\bar{w}_1 \le \dots \le \bar{w}_{n-1}$ & Sorted edge weights of the Gomory--Hu tree. \\
$\bar{w}_n$ & Auxiliary value defined as $\infty$. \\
$\pi(S)$ & Permutation that sorts $\{h_i\}_{i \in S}$ increasingly. \\
$\hat{G} = (V,\hat E)$ & Unweighted multigraph obtained from $\bar G$ by repeating $\{i,j\}$ edge $b_{ij}$ times. \\
$\mathcal{T} = (\hat T_1,\dots,\hat T_p)$ & Collection of edge-disjoint Steiner trees connecting a subset $S \subseteq V$ in $\hat G$. \\
\bottomrule
\end{tabular}
\end{table}

\newpage
\section{Visualization of Algorithm~\ref{alg:preprocess} on Gomory-Hu Tree from Figure~\ref{fig:system_example}}
\label{sec:example_alg1}
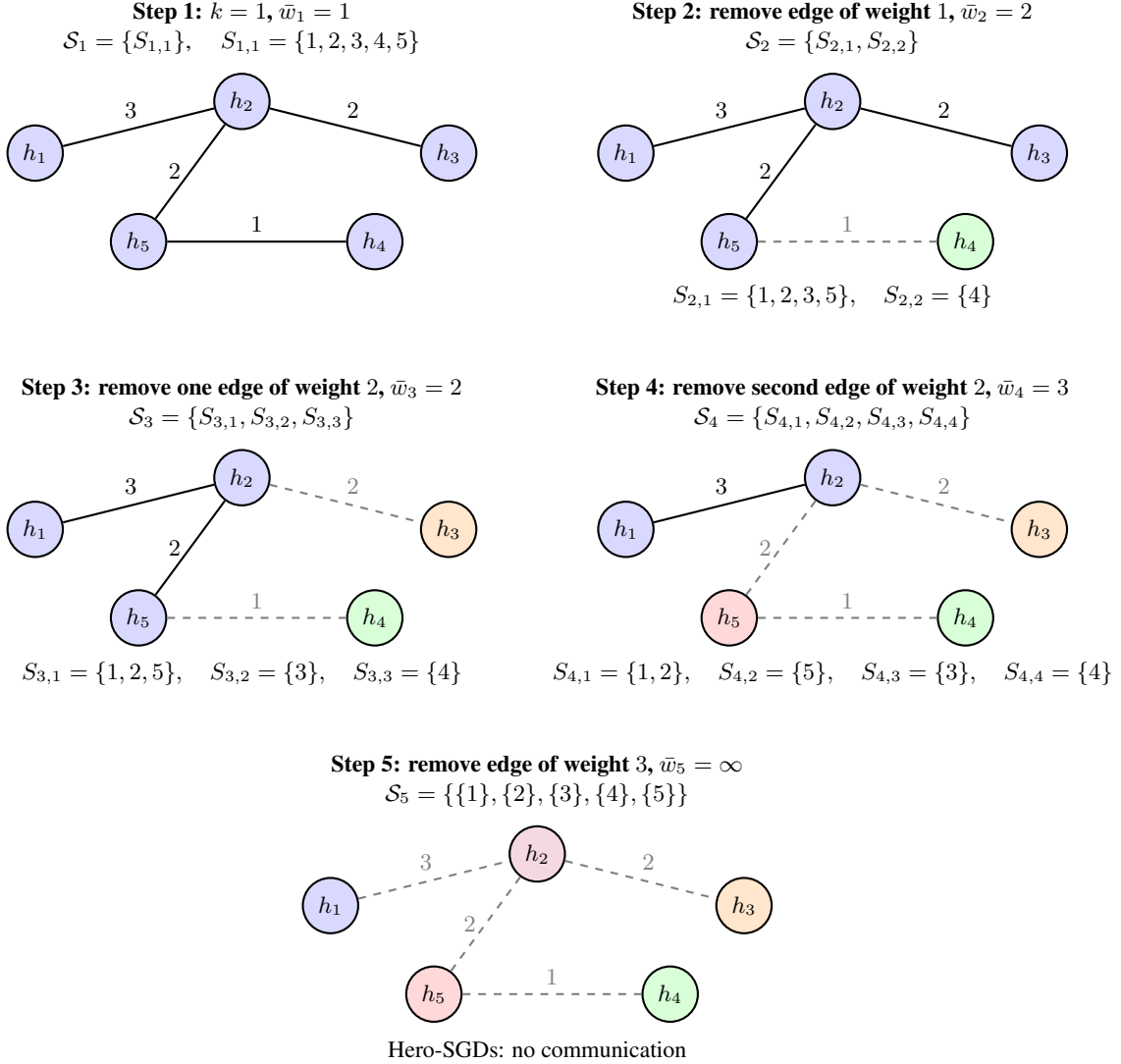
\begin{figure*}[h]
\centering
\begin{tikzpicture}[
    scale=1,
    every node/.style={font=\small},
    worker/.style={
        circle,
        draw=black,
        thick,
        minimum size=0.55cm,
        fill=blue!8
    },
    activeedge/.style={
        thick
    },
    removededge/.style={
        thick,
        dashed,
        gray
    },
    compA/.style={fill=blue!15},
    compB/.style={fill=green!15},
    compC/.style={fill=orange!20},
    compD/.style={fill=red!15},
    compE/.style={fill=purple!15}
]

\begin{scope}[xshift=0cm, yshift=0cm]
\node at (2.8,2.5) {\textbf{Step 1: \(k=1\), \(\bar w_1 = 1\)}};
\node at (2.8,2.1) {$\mathcal S_1 = \{S_{1,1}\}$, \quad $S_{1,1} = \{1,2,3,4,5\}$};

\node[worker, compA] (a1) at (0,0.6) {$h_1$};
\node[worker, compA] (a2) at (2.8,1.3) {$h_2$};
\node[worker, compA] (a3) at (5.6,0.6) {$h_3$};
\node[worker, compA] (a4) at (4.6,-0.6) {$h_4$};
\node[worker, compA] (a5) at (1.4,-0.6) {$h_5$};

\draw[activeedge] (a1) -- node[midway, above,xshift=-3pt] {$3$} (a2);
\draw[activeedge] (a2) -- node[midway, above,xshift=3pt] {$2$} (a3);
\draw[activeedge] (a2) -- node[midway, left] {$2$} (a5);
\draw[activeedge] (a5) -- node[midway, above] {$1$} (a4);
\end{scope}

\begin{scope}[xshift=8.0cm, yshift=0cm]
\node at (2.8,2.5) {\textbf{Step 2: remove edge of weight \(1\), \(\bar w_2 = 2\)}};
\node at (2.8,2.1) {$\mathcal S_2 = \{S_{2,1}, S_{2,2}\}$};

\node[worker, compA] (b1) at (0,0.6) {$h_1$};
\node[worker, compA] (b2) at (2.8,1.3) {$h_2$};
\node[worker, compA] (b3) at (5.6,0.6) {$h_3$};
\node[worker, compB] (b4) at (4.6,-0.6) {$h_4$};
\node[worker, compA] (b5) at (1.4,-0.6) {$h_5$};

\draw[activeedge] (b1) -- node[midway, above,xshift=-3pt] {$3$} (b2);
\draw[activeedge] (b2) -- node[midway, above,xshift=3pt] {$2$} (b3);
\draw[activeedge] (b2) -- node[midway, left] {$2$} (b5);
\draw[removededge] (b5) -- node[midway, above, gray] {$1$} (b4);

\node at (2.8,-1.35) {$S_{2,1}=\{1,2,3,5\}, \quad S_{2,2}=\{4\}$};
\end{scope}

\begin{scope}[xshift=0cm, yshift=-5.1cm]
\node at (2.8,2.5) {\textbf{Step 3: remove one edge of weight \(2\), \(\bar w_3 = 2\)}};
\node at (2.8,2.1) {$\mathcal S_3 = \{S_{3,1}, S_{3,2}, S_{3,3}\}$};

\node[worker, compA] (c1) at (0,0.6) {$h_1$};
\node[worker, compA] (c2) at (2.8,1.3) {$h_2$};
\node[worker, compC] (c3) at (5.6,0.6) {$h_3$};
\node[worker, compB] (c4) at (4.6,-0.6) {$h_4$};
\node[worker, compA] (c5) at (1.4,-0.6) {$h_5$};

\draw[activeedge] (c1) -- node[midway, above,xshift=-3pt] {$3$} (c2);
\draw[removededge] (c2) -- node[midway, above,xshift=3pt, gray] {$2$} (c3);
\draw[activeedge] (c2) -- node[midway, left] {$2$} (c5);
\draw[removededge] (c5) -- node[midway, above, gray] {$1$} (c4);

\node at (2.8,-1.35) {$S_{3,1}=\{1,2,5\}, \quad S_{3,2}=\{3\}, \quad S_{3,3}=\{4\}$};
\end{scope}

\begin{scope}[xshift=8.0cm, yshift=-5.1cm]
\node at (2.8,2.5) {\textbf{Step 4: remove second edge of weight \(2\), \(\bar w_4 = 3\)}};
\node at (2.8,2.1) {$\mathcal S_4 = \{S_{4,1}, S_{4,2}, S_{4,3}, S_{4,4}\}$};

\node[worker, compA] (d1) at (0,0.6) {$h_1$};
\node[worker, compA] (d2) at (2.8,1.3) {$h_2$};
\node[worker, compC] (d3) at (5.6,0.6) {$h_3$};
\node[worker, compB] (d4) at (4.6,-0.6) {$h_4$};
\node[worker, compD] (d5) at (1.4,-0.6) {$h_5$};

\draw[activeedge] (d1) -- node[midway, above,xshift=-3pt] {$3$} (d2);
\draw[removededge] (d2) -- node[midway, above,xshift=3pt, gray] {$2$} (d3);
\draw[removededge] (d2) -- node[midway, left, gray] {$2$} (d5);
\draw[removededge] (d5) -- node[midway, above, gray] {$1$} (d4);

\node at (2.8,-1.35) {$S_{4,1}=\{1,2\}, \quad S_{4,2}=\{5\}, \quad S_{4,3}=\{3\}, \quad S_{4,4}=\{4\}$};
\end{scope}

\begin{scope}[xshift=4.0cm, yshift=-10.2cm]
\node at (2.8,2.5) {\textbf{Step 5: remove edge of weight \(3\), \(\bar w_5 = \infty\)}};
\node at (2.8,2.1) {$\mathcal S_5 = \{\{1\},\{2\},\{3\},\{4\},\{5\}\}$};

\node[worker, compA] (e1) at (0,0.6) {$h_1$};   
\node[worker, compE] (e2) at (2.8,1.3) {$h_2$}; 
\node[worker, compC] (e3) at (5.6,0.6) {$h_3$}; 
\node[worker, compB] (e4) at (4.6,-0.6) {$h_4$}; 
\node[worker, compD] (e5) at (1.4,-0.6) {$h_5$}; 

\draw[removededge] (e1) -- node[midway, above,xshift=-3pt, gray] {$3$} (e2);
\draw[removededge] (e2) -- node[midway, above,xshift=3pt, gray] {$2$} (e3);
\draw[removededge] (e2) -- node[midway, left, gray] {$2$} (e5);
\draw[removededge] (e5) -- node[midway, above, gray] {$1$} (e4);

\node at (2.8,-1.35) {Hero-SGDs: no communication};
\end{scope}

\end{tikzpicture}
\caption{Step-by-step visualization of Algorithm~\ref{alg:preprocess} on the Gomory--Hu tree $T.$}
\label{fig:gh_step_by_step}
\end{figure*}
In Step 1, we assume that all workers communicate to each other. In this case, if we implement Grace SGD (with $S^* \to S_{1,1}$), then the guaranteed complexity would be
\begin{align*}
  \cO\left(\frac{d}{\bar{w}_1} \times \frac{L \Delta}{\varepsilon} + \min\limits_{m \in [\abs{S_{1,1}}]} \left[\left(\frac{1}{m} \sum\limits_{i=1}^{m} \frac{1}{h_{\pi_{i}(S_{1,1})}}\right)^{-1} \left(1 + \frac{\sigma^2}{\varepsilon m}\right)\right] \frac{L \Delta}{\varepsilon}\right).
\end{align*}
In Step 2, we split worker $4$ from the other workers in $T$ and evaluate the performance of Grace SGD (with $S^* \to S_{2,1}$ and $S^* \to S_{2,2}$), and record the best performance $t^*$ found so far (Algorithm~\ref{alg:preprocess}). We repeat these steps until all workers are split.

\textbf{Important observation:} Assume that $S^* = S_{4,1} = \{1,2\}$ and Grace SGD decides to use only these workers in the optimization process. This does not mean that, for instance, node $5$ would not participate. On the contrary, node $5$ should participate, since part of the information would flow through it (see Figure~\ref{fig:system_example}, where node $5$ is needed to achieve the total max flow from node $1$ to node $2$). However, node $5$ would not perform any computations or produce any gradients. It would act as a switch. See also Section~\ref{sec:example_alg2}.

\section{Visualization of Algorithm~\ref{alg:allreduce}}
\label{sec:example_alg2}
In this section, we consider the graph from Figure~\ref{fig:allreduce_step_by_step_S126} and visualize the behavior of Algorithm~\ref{alg:allreduce}.
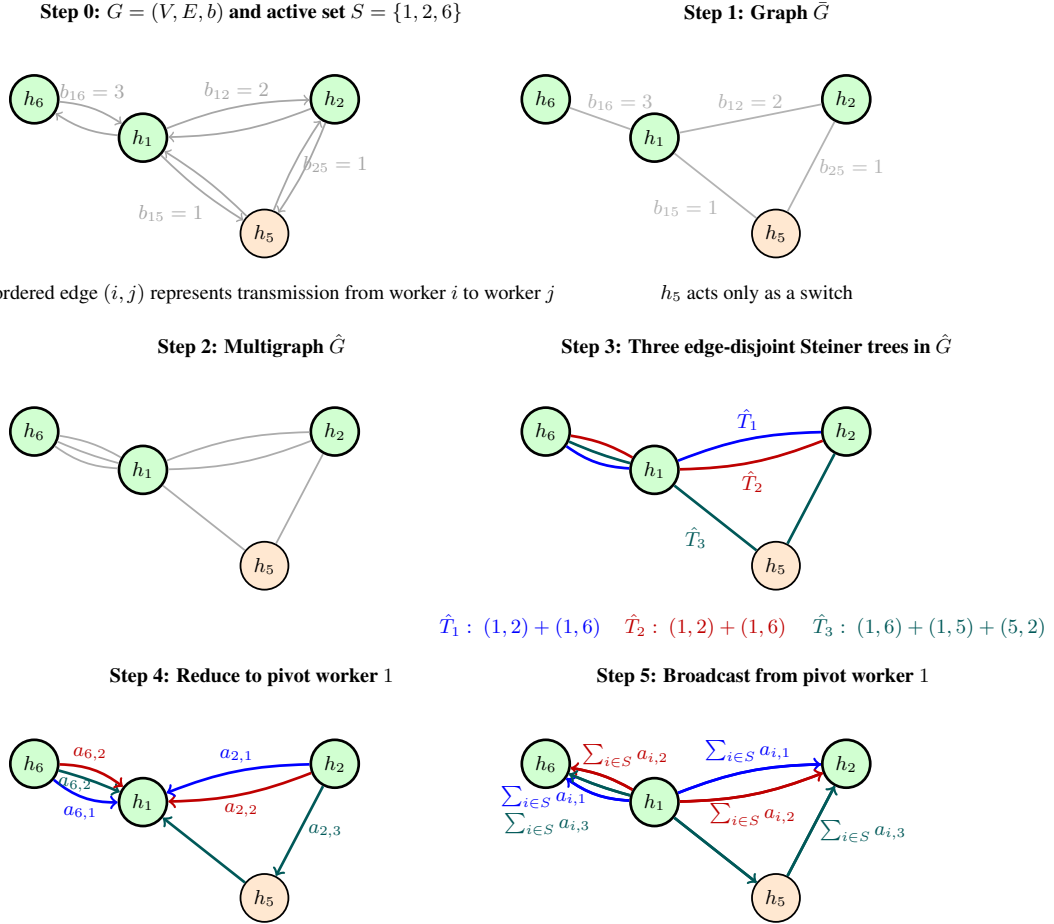
\begin{figure*}[h]
\centering
\scalebox{0.85}{
\begin{tikzpicture}[
    scale=1,
    every node/.style={font=\small},
    worker/.style={
        circle,
        draw=black,
        thick,
        minimum size=0.54cm,
        fill=blue!8
    },
    active/.style={
        circle,
        draw=black,
        very thick,
        minimum size=0.54cm,
        fill=green!18
    },
    switch/.style={
        circle,
        draw=black,
        thick,
        minimum size=0.54cm,
        fill=orange!18
    },
    edgebase/.style={
        gray!60,
        thick
    },
    multiedge/.style={
        gray!65,
        thick
    },
    treeA/.style={
        very thick,
        blue
    },
    treeB/.style={
        very thick,
        red!75!black
    },
    treeC/.style={
        very thick,
        teal!70!black
    },
    msgA/.style={
        ->,
        very thick,
        blue
    },
    msgB/.style={
        ->,
        very thick,
        red!75!black
    },
    msgC/.style={
        ->,
        very thick,
        teal!70!black
    },
    dirlink/.style={
        ->,
        thick,
        gray!70
    }
]

\begin{scope}[xshift=0cm, yshift=0cm]
\node at (1.7,2.45) {\textbf{Step 0: \(G=(V,E,b)\) and active set \(S=\{1,2,6\}\)}};

\node[active] (z1) at (0,0.5) {$h_1$};
\node[active] (z2) at (3.0,1.1) {$h_2$};
\node[switch] (z5) at (1.9,-1.0) {$h_5$};
\node[active] (z6) at (-1.7,1.1) {$h_6$};

\draw[dirlink] (z1) to[bend left=10] node[midway, above] {$b_{12}=2$} (z2);
\draw[dirlink] (z2) to[bend left=10] (z1);

\draw[dirlink] (z1) to[bend left=14] (z6);
\draw[dirlink] (z6) to[bend left=14] node[midway, above] {$b_{16}=3$} (z1);

\draw[dirlink] (z1) to[bend right=6] node[midway, left, yshift=-10pt, xshift=5pt] {$b_{15}=1$} (z5);
\draw[dirlink] (z5) to[bend right=6] (z1);

\draw[dirlink] (z5) to[bend left=6] node[midway, right] {$b_{25}=1$} (z2);
\draw[dirlink] (z2) to[bend left=6] (z5);

\node at (1.7,-1.95) {Each ordered edge \((i,j)\) represents transmission from worker \(i\) to worker \(j\)};
\end{scope}

\begin{scope}[xshift=8.0cm, yshift=0cm]
\node at (1.6,2.45) {\textbf{Step 1: Graph \(\bar{G}\)}};

\node[active] (a1) at (0,0.5) {$h_1$};
\node[active] (a2) at (3.0,1.1) {$h_2$};
\node[switch] (a5) at (1.9,-1.0) {$h_5$};
\node[active] (a6) at (-1.7,1.1) {$h_6$};

\draw[edgebase] (a1) -- node[midway, above] {$b_{12}=2$} (a2);
\draw[edgebase] (a1) -- node[midway, above left, xshift=27pt] {$b_{16}=3$} (a6);
\draw[edgebase] (a1) -- node[midway, left, yshift=-10pt, xshift=5pt] {$b_{15}=1$} (a5);
\draw[edgebase] (a5) -- node[midway, right] {$b_{25}=1$} (a2);

\node at (1.6,-1.95) {\(h_5\) acts only as a switch};
\end{scope}

\begin{scope}[xshift=0cm, yshift=-5.2cm]
\node at (1.7,2.45) {\textbf{Step 2: Multigraph \(\hat{G}\)}};

\node[active] (b1) at (0,0.5) {$h_1$};
\node[active] (b2) at (3.0,1.1) {$h_2$};
\node[switch] (b5) at (1.9,-1.0) {$h_5$};
\node[active] (b6) at (-1.7,1.1) {$h_6$};

\draw[multiedge] (b1) to[bend left=10] (b2);
\draw[multiedge] (b1) to[bend right=10] (b2);

\draw[multiedge] (b1) to[bend left=16] (b6);
\draw[multiedge] (b1) to[bend left=4] (b6);
\draw[multiedge] (b1) to[bend right=10] (b6);

\draw[multiedge] (b1) -- (b5);
\draw[multiedge] (b5) -- (b2);

\end{scope}

\begin{scope}[xshift=8.0cm, yshift=-5.2cm]
\node at (1.6,2.45) {\textbf{Step 3: Three edge-disjoint Steiner trees in \(\hat{G}\)}};

\node[active] (c1) at (0,0.5) {$h_1$};
\node[active] (c2) at (3.0,1.1) {$h_2$};
\node[switch] (c5) at (1.9,-1.0) {$h_5$};
\node[active] (c6) at (-1.7,1.1) {$h_6$};

\draw[edgebase] (c1) to[bend left=10] (c2);
\draw[edgebase] (c1) to[bend right=10] (c2);

\draw[edgebase] (c1) to[bend left=16] (c6);
\draw[edgebase] (c1) to[bend left=4] (c6);

\draw[edgebase] (c1) -- (c5);
\draw[edgebase] (c5) -- (c2);

\draw[treeA] (c1) to[bend left=10] node[midway, above] {\(\hat T_1\)} (c2);
\draw[treeA] (c1) to[bend left=16] (c6);

\draw[treeB] (c1) to[bend right=10] node[midway, below] {\(\hat T_2\)} (c2);
\draw[treeB] (c1) to[bend right=10] (c6);

\draw[treeC] (c1) to[bend left=4] (c6);
\draw[treeC] (c1) -- node[midway, below left] {\(\hat T_3\)} (c5);
\draw[treeC] (c5) -- (c2);

\node[blue] at (-2.1,-1.95) {\(\hat T_1:\ (1,2) + (1,6)\)};
\node[red!75!black] at (0.8,-1.95) {\(\hat T_2:\ (1,2) + (1,6)\)};
\node[teal!70!black] at (4.3,-1.95) {\(\hat T_3:\ (1,6)+(1,5)+(5,2)\)};
\end{scope}

\begin{scope}[xshift=0cm, yshift=-10.4cm]
\node at (1.7,2.45) {\textbf{Step 4: Reduce to pivot worker \(1\)}};

\node[active] (d1) at (0,0.5) {$h_1$};
\node[active] (d2) at (3.0,1.1) {$h_2$};
\node[switch] (d5) at (1.9,-1.0) {$h_5$};
\node[active] (d6) at (-1.7,1.1) {$h_6$};

\draw[msgB] (d2) to[bend left=10] node[midway, below] {$a_{2,2}$} (d1);
\draw[msgA] (d6) to[bend right=20] node[midway, above left,xshift=10pt,yshift=-14pt] {$a_{6,1}$} (d1);

\draw[msgA] (d2) to[bend right=10] node[midway, above] {$a_{2,1}$} (d1);
\draw[msgC] (d6) to[bend left=4] node[midway, below left,xshift=5pt,yshift=5pt] {$a_{6,2}$} (d1);

\draw[msgC] (d2) to node[midway, right] {$a_{2,3}$} (d5);
\draw[msgC] (d5) to node[midway, below] {} (d1);
\draw[msgB] (d6) to[bend left=18] node[midway, above left,xshift=10pt] {$a_{6,2}$} (d1);

\end{scope}

\begin{scope}[xshift=8.0cm, yshift=-10.4cm]
\node at (1.7,2.45) {\textbf{Step 5: Broadcast from pivot worker \(1\)}};

\node[active] (e1) at (0,0.5) {$h_1$};
\node[active] (e2) at (3.0,1.1) {$h_2$};
\node[switch] (e5) at (1.9,-1.0) {$h_5$};
\node[active] (e6) at (-1.7,1.1) {$h_6$};

\draw[edgebase] (e1) to[bend left=10] (e2);
\draw[edgebase] (e1) to[bend right=10] (e2);

\draw[edgebase] (e1) to[bend left=16] (e6);
\draw[edgebase] (e1) to[bend left=4] (e6);
\draw[edgebase] (e1) to[bend right=10] (e6);

\draw[edgebase] (e1) -- (e5);
\draw[edgebase] (e5) -- (e2);

\draw[treeA] (e1) to[bend left=10] (e2);
\draw[treeA] (e1) to[bend left=16] (e6);

\draw[treeB] (e1) to[bend right=10] (e2);
\draw[treeB] (e1) to[bend right=10] (e6);

\draw[treeC] (e1) to[bend left=4] (e6);
\draw[treeC] (e1) -- (e5);
\draw[treeC] (e5) -- (e2);

\draw[msgA] (e1) to[bend left=10] node[midway, above] {$\sum_{i\in S} a_{i,1}$} (e2);
\draw[msgA] (e1) to[bend left=16] node[midway, above left,yshift=-10pt] {$\sum_{i\in S} a_{i,1}$} (e6);

\draw[msgB] (e1) to[bend right=10] node[midway, below] {$\sum_{i\in S} a_{i,2}$} (e2);
\draw[msgB] (e1) to[bend right=10] node[midway, above,xshift=10pt] {$\sum_{i\in S} a_{i,2}$} (e6);

\draw[msgC] (e1) to node[midway, below] {} (e5);
\draw[msgC] (e5) to node[midway, right] {$\sum_{i\in S} a_{i,3}$} (e2);
\draw[msgC] (e1) to[bend left=4] node[midway, left, yshift=-17pt] {$\sum_{i\in S} a_{i,3}$} (e6);

\end{scope}

\end{tikzpicture}
}
\caption{
Step-by-step visualization of the optimal-bandwidth AllReduce for \(S=\{1,2,6\}\).
Step 0 shows the initial directed graph \(G=(V,E,b)\), where each ordered edge \((i,j)\) represents a communication link from worker \(i\) to worker \(j\) with bandwidth \(b_{ij}\).
Step 1 shows the corresponding undirected graph \(\bar G\).
Step 2 shows the multigraph \(\hat G\), where each edge of bandwidth \(b_{ij}\) is replaced by \(b_{ij}\) parallel unit-bandwidth edges.
Step 3 illustrates three edge-disjoint Steiner trees in \(\hat G\).
Step 4 shows the reduce phase: each worker sends one block through each tree to the pivot worker \(1\).
Step 5 shows the broadcast phase: the pivot sends the aggregated blocks back through the same trees. \textbf{Important observation:} Node $5$ is still required in the optimization process, but only as a switch.
}
\label{fig:allreduce_step_by_step_S126}
\end{figure*}

\section{Examples with Different Graph Structures in the Heterogeneous Setting}
\label{sec:examples_heter}
Similarly to Section~\ref{sec:examples}, we consider the same examples in the heterogeneous setting. We also assume that $h_i = h$ for all $i \in [n].$

\textbf{Example: One Worker.} Assume that $n = 1,$ i.e., $G$ is a graph with one node. In this case, Gomory-Hu $T$ is the same as $G,$ $\bar{w}_1 = \infty$ and $S_{1,1} = \{1\}$ in Algorithm~\ref{alg:preprocess}, and $\eqref{eq:GvDIFIzvheter} = \cO\left(\nicefrac{h \sigma^2 L \Delta}{\varepsilon^2} + \nicefrac{h L \Delta}{\varepsilon}\right),$ and we get the same result as in the homogeneous setup.

\textbf{Example: Star Graph (Centralized Setting; Figure~\ref{fig:star_graph_gomory_hu}).} In this case, $\bar{w}_1 = b.$  Thus,
\begin{align*}
  \textstyle \eqref{eq:GvDIFIzvheter} = \cO\left(\left(\frac{d}{b} + \frac{h \sigma^2}{n \varepsilon}\right) \frac{L \Delta}{\varepsilon} + \frac{h L \Delta}{\varepsilon}\right).
\end{align*}

\textbf{Example: $p$-Torus.} Since $\bar{w}_1 = 2 p b$ in Figure~\ref{fig:torus_gomory_hu},
\begin{align*}
  \textstyle \eqref{eq:GvDIFIzvheter} = \cO\left(\left(\frac{d}{p b} + \frac{h \sigma^2}{n \varepsilon}\right) \frac{L \Delta}{\varepsilon} + \frac{h L \Delta}{\varepsilon}\right).
\end{align*}

\textbf{Example: all-to-all.} In the all-to-all graph $G$, 
\begin{align*}
  \textstyle \eqref{eq:GvDIFIzvheter} = \cO\left(\left(\frac{d}{(n - 1) b} + \frac{h \sigma^2}{n \varepsilon}\right) \frac{L \Delta}{\varepsilon} + \frac{h L \Delta}{\varepsilon}\right),
\end{align*}
for all $n \geq 2.$

\textbf{Example: $K$ clusters.}
Using Figure~\ref{fig:k_clusters_ring}, one can show that $\bar{w}_1$ and Theorem~\ref{thm:sgd_heter} yields the time complexity
\begin{align*}
  \textstyle \cO\left(\frac{d L \Delta}{b_{\textnormal{slow}} \varepsilon} + \frac{h \sigma^2 L \Delta}{n \varepsilon^2} + \frac{h L \Delta}{\varepsilon}\right).
\end{align*}
Unlike the homogeneous setting (Section~\ref{sec:examples}), this complexity tends to $\infty$ if $b_{\textnormal{slow}} \to 0.$

\textbf{Example: optimization with switches.} The heterogeneous setting can also support switches, but this would require extending the setup described in Section~\ref{sec:heter}. Briefly, under Assumption~\ref{ass:time}, we assume that there are two subsets, $S_{w}$ and $S_{s},$ such that $S_{w} \cup S_{s} = [n]$ and $S_{w} \cap S_{s} = \emptyset,$ where $S_{w}$ is the set of real workers and $S_{s}$ is the set of switches. Then, we solve
\begin{align}
  \label{eq:main_heter_switch}
  \textstyle \min \limits_{x \in \R^d} \left\{f(x) \eqdef \frac{1}{\abs{S_{w}}} \sum_{i \in S_{w}} f_i(x)\right\}.
\end{align}
Using a modified version of Algorithm~\ref{alg:main_leon}, where only $\abs{S_{w}}$ workers compute stochastic gradients and run AllReduce, one can extend Theorem~\ref{thm:sgd_heter} and prove the time complexity
\begin{align*}
  \textstyle \cO\left(\max\left\{\frac{d}{\alpha_{\bar{G}}(S_w)}, \max\limits_{i \in S_{w}} h_i, \frac{\sigma^2}{\abs{S_{w}} \varepsilon}\left(\frac{1}{\abs{S_{w}}} \sum\limits_{i \in S_{w}} h_i\right)\right\} \frac{L \Delta}{\varepsilon}\right)
\end{align*}
seconds, where $\alpha_{\bar{G}}(S_w)$ is the minimum value of an $S_{w}$-cut in $\bar{G}$ (Definition~\ref{def:min_cut} and Theorem~\ref{thm:allreduce_time}). When $S_w = [n]$ and $S_s = \emptyset,$ this complexity reduces to \eqref{eq:GvDIFIzvheter}.

\section{Definitions and Results from Graph Theory}
In this section, we list the standard definition and results from graph theory.
\begin{definition}
Given an undirected weighted graph $\bar{G} = (V, \bar E, b)$ and a vertex set $U \subseteq V$, we define
\[
\delta(U) \;\coloneqq\; \{\, e = \{x,y\} \in \bar E \;:\; x \in U,\; y \in V \setminus U \,\},
\]
that is, the set of edges with exactly one endpoint in $U$.
\end{definition}

\begin{definition}
Given an undirected weighted graph $\bar{G} = (V, \bar E, b)$ and two vertices $u, v \in V$, a \emph{$u$--$v$ cut} is a partition $(U, V \setminus U)$ of $V$ such that $u \in U$ and $v \in V \setminus U$. The value of the cut is defined as
\[
\sum_{e \in \delta(U)} b_e.
\]
\end{definition}

\begin{definition}
Given an undirected weighted graph $\bar{G} = (V, \bar E, b)$ and two vertices $u, v \in V$, we define $\alpha_{\bar{G}}(u, v)$ as the minimum value of a $u$--$v$ cut in $\bar{G}$. Moreover, any $u$--$v$ cut that attains this minimum value is called a \emph{minimum $u$--$v$ cut}.
\end{definition}

\begin{definition}
Given an undirected weighted graph $\bar{G} = (V, \bar E, b)$ and a vertex set $S \subseteq V$, an \emph{$S$-cut} is a partition $(U, V \setminus U)$ of $V$ such that 
\[
U \cap S \neq \emptyset
\quad\text{and}\quad
(V \setminus U) \cap S \neq \emptyset.
\]
The value of the cut is defined as
\[
\sum_{e \in \delta(U)} b_e.
\]
\end{definition}

\begin{definition}
\label{def:min_cut}
Given an undirected weighted graph $\bar{G} = (V, \bar E, b)$ and a vertex set $S \subseteq V$, we define $\alpha_{\bar{G}}(S)$ as the minimum value of an $S$-cut in $\bar{G}$. Moreover, any $S$-cut that attains this minimum value is called a \emph{minimum $S$-cut}.
\end{definition}

\begin{definition}[\citet{gomory1961multi}]
  \label{def:gh}
  A Gomory--Hu tree $T = (V, F, w)$ of an undirected weighted graph $\bar{G} = (V, \bar E, b)$ is a tree $T = (V, F, w),$ which consists of a tree edge set $F$ and capacities $w$ such that, for every edge $\{i, j\} \in F$, $\delta(W)$ is a minimum $i-j$ cut in $\bar{G}$ and $w_{ij}$ is the value of this cut, where $W$ is one component\footnote{Removing the edge $\{i, j\}$ from the tree $T$ yields two sets of nodes $W_1$ and $W_2$, and $W$ can be either of them, without loss of generality, since $\delta(W_1) = \delta(W_2).$} of $T - \{i, j\}$.
\end{definition}
Notice that, by definition, we have $w_{ij} = \alpha_{\bar{G}}(i, j)$ for all $\{i, j\} \in F$. Let us also recall the following important theorem.
\begin{theorem}[e.g. \citep{schrijver2003combinatorial}]
\label{thm:gomory_hu}
Let $T = (V, F, w)$ be a Gomory--Hu tree of $\bar{G}$. For any two vertices $s, t \in V$, consider the unique $s$--$t$ path in $T$, and let $\{u, v\} \in F$ be an edge on this path minimizing the weight $w_{uv}$. Then,
\[
\alpha_{\bar{G}}(s, t) = \alpha_{\bar{G}}(u, v),
\]
and for any component $K$ of $T - \{u, v\}$, the cut $\delta(K)$ is a minimum $s$--$t$ cut in $\bar{G}$.
\end{theorem}
It turns out that the Gomory--Hu tree naturally captures the \emph{connectivity structure and bottlenecks} of both graphs $\bar{G}$ and $G$. Indeed, consider any tree edge $\{u, v\} \in F$ and the corresponding components $W_1$ and $W_2$ obtained by removing the edge $\{u, v\}$ from $T$. Then, $\alpha_{\bar{G}}(s, t) \le \alpha_{\bar{G}}(u, v) = w_{uv}$ for all $s \in W_1$ and $t \in W_2$ by Theorem~\ref{thm:gomory_hu}. By the max-flow min-cut theorem, the \emph{maximum $s$--$t$ flow} in $\bar{G}$ is upper bounded by $w_{uv}$.

\section{Upper Bounds}
\subsection{Proof of Theorem~\ref{thm:sgd_homog}}
\THEOREMSGDHOMOG*

\begin{proof}
(\textbf{Iteration rate}). The proof of the iteration is standard and we prove it for completeness.
Fix any subset of workers \(S\subseteq [n]\), and let
\[
B \;=\; \max\!\left\{\left\lceil \frac{\sigma^2}{\varepsilon}\right\rceil,1\right\}.
\]
At iteration \(k\), Grace SGD forms the mini-batch estimator
\[
g^k \;=\; \frac{1}{B}\sum_{i\in S}\sum_{j=1}^{B_i^k} \nabla f(x^k;\xi_{i,j}^k),
\qquad \sum_{i\in S} B_i^k = B.
\]
By Assumption~\ref{ass:stochastic_variance_bounded},
$\ExpCond{g^k}{x^k} = \nabla f(x^k)$ and $\ExpCond{\|g^k-\nabla f(x^k)\|^2}{x^k} \le \frac{\sigma^2}{B}\le \varepsilon.$

Using \(L\)-smoothness of \(f\) and the update
\[
x^{k+1}=x^k-\gamma g^k,
\qquad \gamma=\frac{1}{2L},
\]
we get
\begin{align*}
\ExpCond{f(x^{k+1})}{x^k}
&\le
f(x^k)
-\gamma \|\nabla f(x^k)\|^2
+\frac{L\gamma^2}{2}\ExpCond{\|g^k\|^2}{x^k}.
\end{align*}
Since
\[
\ExpCond{\|g^k\|^2}{x^k}
=
\|\nabla f(x^k)\|^2+\ExpCond{\|g^k-\nabla f(x^k)\|^2}{x^k}
\le
\|\nabla f(x^k)\|^2+\varepsilon,
\]
it follows that
\begin{align*}
\ExpCond{f(x^{k+1})}{x^k}
&\le
f(x^k)
-\left(\gamma-\frac{L\gamma^2}{2}\right)\|\nabla f(x^k)\|^2
+\frac{L\gamma^2}{2}\varepsilon \\
&=
f(x^k)
-\frac{3}{8L}\|\nabla f(x^k)\|^2
+\frac{\varepsilon}{8L}.
\end{align*}
Rearranging and taking full expectation,
\[
\frac{3}{8L}\Exp{\norm{\nabla f(x^k)}^2}
\le
\Exp{f(x^k)}-\Exp{f(x^{k+1})}+\frac{\varepsilon}{8L}.
\]
Summing over \(k=0,\dots,K-1\) gives
\[
\frac{3}{8L}\sum_{k=0}^{K-1}\Exp{\|\nabla f(x^k)\|^2}
\le
f(x^0)-\inf_{x \in \R^d} f(x)+\frac{K\varepsilon}{8L}
=
\Delta+\frac{K\varepsilon}{8L}.
\]
Hence
\[
\frac{1}{K}\sum_{k=0}^{K-1}\Exp{\norm{\nabla f(x^k)}^2}
\le
\frac{8L\Delta}{3K}+\frac{\varepsilon}{3}.
\]
So it is enough to choose \(K\) such that the right-hand side is at most \(\varepsilon\). In particular,
\[
K=\left\lceil \frac{4 L\Delta}{\varepsilon}\right\rceil
\]
is sufficient, which proves the iteration complexity.

(\textbf{Time complexity}). It remains to upper bound the time complexity of one iteration. Consider the chosen subset of workers $S^* \subseteq [n]$ in Algorithm~\ref{alg:preprocess}. Under Assumption~\ref{ass:time} (\hyperref[box:computation_model]{\textcolor{gray}{Computation Model}}), the time required to collect a batch of size $B = \max\!\left\{\left\lceil \frac{\sigma^2}{\varepsilon}\right\rceil,1\right\}$ by the workers from $S^*$ is at most
\begin{align*}
  \min_{m \in [\abs{S^*}]} \left[\left(\frac{1}{m}\sum_{i=1}^m \frac{1}{h_{\pi_i(S^*)}}\right)^{-1}\right]
\left(1+\frac{\sigma^2}{m\varepsilon}\right)
\end{align*}
seconds (Lemma~\ref{lem:batch_time}). Moreover, by our construction in Algorithm~\ref{alg:preprocess}, we know that the value of a minimum $S^*$-cut (Definition~\ref{def:min_cut}) is greater than or equal to $\bar{w}_{k^*}.$ Thus, there exists an optimal-bandwidth AllReduce algorithm (Algorithm~\ref{alg:allreduce}, Theorem~\ref{thm:allreduce_time}) such that the allreduce operation is bounded by $\cO\left(\nicefrac{d}{\bar{w}_{k^*}}\right)$ seconds. Hence one iteration on this subset costs at most
\[
\cO\left(\frac{d}{\bar w_{k^*}} + \min_{m \in [\abs{S^*}]} \left[\left(\frac{1}{m}\sum_{i=1}^m \frac{1}{h_{\pi_i(S^*)}}\right)^{-1}\right]
\left(1+\frac{\sigma^2}{m\varepsilon}\right)\right).
\]
Algorithm~\ref{alg:preprocess} chooses the subset \(S^*\) minimizing this quantity over all candidates
\(\{S_{k,p}\}\). Therefore, the per-iteration time is at most
\[
\min_{k \in [n]} \left(\frac{d}{\bar{w}_{k}} + \min\limits_{p \in [k]}\min\limits_{m \in [\abs{S_{k,p}}]} \left[\left(\frac{1}{m} \sum\limits_{i=1}^{m} \frac{1}{h_{\pi_{i}(S_{k,p})}}\right)^{-1} \left(1 + \frac{\sigma^2}{m \varepsilon}\right)\right] \right).
\]
Multiplying by the number of iterations
\(
K=\left\lceil \frac{4L\Delta}{\varepsilon}\right\rceil,
\)
we get \eqref{eq:GvDIFIzv}.
\end{proof}

\begin{theorem}
\label{thm:allreduce_time}
Let $\bar{G}=(V,\bar E,b)$ be the undirected version of $G:$ $\{i,j\} \in \bar E$ with weight $b_{ij}$ iff $(i,j)\in E$ with weight $b_{ij}.$ Under Assumption~\ref{ass:time} (\hyperref[box:comm_model]{\textcolor{gray}{Graph-Bandwidth Communication Model}}), the time complexity of the AllReduce algorithm in Algorithm~\ref{alg:allreduce} is
\begin{align*}
    \mathcal{O}\!\left(\frac{d}{\alpha_{\bar{G}}(S)}\right),
\end{align*}
where $\alpha_{\bar{G}}(S)$ denotes the value of a minimum $S$-cut in the graph $\bar{G}$ (Definition~\ref{def:min_cut}).
\end{theorem}

\begin{proof}
Consider the unweighted multigraph $\hat G=(V,\hat E)$ obtained from $\bar G=(V,\bar E,b)$ by replacing every edge $\{i,j\}\in \bar E$ of weight $b_{ij}$ with $b_{ij}$ parallel edges of unit bandwidth. Clearly, any cut in $\bar G$ and the corresponding cut in $\hat G$ have the same value. Therefore,
\[
\alpha_{\bar G}(S)=\alpha_{\hat G}(S).
\]
We explain in the main part that in practice, the system still has a single edge with bandwidth $b_{ij}$; however, the behavior of $b_{ij}$ parallel unit-bandwidth edges can be simulated by multiplexing transmissions using the \emph{interleaving} strategy described in Section~\ref{sec:intro}.

By the Steiner tree packing result of \citet[Theorem 1.2]{lau2004approximate}, there exists a polynomial-time algorithm that finds a collection of
\[
p=\Theta\!\big(\alpha_{\hat G}(S)\big)=\Theta\!\big(\alpha_{\bar G}(S)\big)
\]
edge-disjoint trees
\[
\mathcal T=(\hat T_1,\dots,\hat T_p)
\]
in $\hat G$, each of which connects all vertices of $S$.

We use these trees as parallel communication pipes. Split every local vector $a_i\in\R^d$ stored at worker $i\in S$ into $p$ disjoint blocks,
\[
a_i=\big(a_{i,1},\dots,a_{i,p}\big),
\]
where each block has size $\cO(d/p)$ coordinates (w.l.o.g., we assume that $d$ is divisible by $p;$ otherwise, we can pad with zero values). Assign block $\ell$ to tree $\hat T_\ell$.

Fix any pivot worker $r\in S$. We first perform a reduce operation. 
For every $\ell \in [p]$ and any worker $i \in S,$ find the unique path from $i$ to the root $r.$ 
Each worker $i \in S$ starts streaming the coordinates of its block 
$a_{i,\ell}$ along the unique path in $\hat T_\ell$ toward $r$. 
Every intermediate node waits until it receives the next coordinate from all of its children, i.e., the neighbors of the intermediate node that send coordinates to the root through this intermediate node. Importantly, the intermediate node does not wait to receive the entire block of size $\Theta(d/p)$. As soon as the first coordinate is received from all children, the node aggregates these values (adding its own coordinate if it belongs to $S$) and immediately forwards the result to its parent, i.e., the next node on the path towards the root $r.$ The same procedure is then applied to subsequent coordinates.

Since the trees are edge-disjoint, these communications do not interfere with one another. Moreover, every edge in $\hat G$ has unit bandwidth, so the time needed to transmit one block through one tree is proportional to the block size, that is,
\[
\cO(d/p).
\]
Thus, after $\cO(d/p)$ seconds, the pivot worker $r$ has obtained the sum of block $\ell$ for every $\ell\in[p]$, and hence the full sum $\sum_{i\in S} a_i$. In practice, of course, in addition to sending the values of the coordinates, the workers might also transfer the indices of these coordinates and other meta-information; nevertheless, this would increase the cost by at most a multiplicative constant factor $\Theta(1)$.

Next, we perform a broadcast operation. Using the same collection of trees, the pivot worker sends block $\ell$ of the aggregated vector through $\hat T_\ell$, and intermediate nodes forward the received information further. Importantly, the intermediate node does not wait to receive the entire block and immediately broadcasts a new coordinate upon receiving it, without waiting for the next ones. Again, because the trees are edge-disjoint, all broadcasts proceed in parallel without congestion, and the required time is
\[
\cO(d/p).
\]

Combining the reduce and broadcast phases, the total communication time is
\[
\cO(d/p)+\cO(d/p)=\cO(d/p).
\]
Finally, since
\[
p=\Theta\!\big(\alpha_{\bar G}(S)\big),
\]
we obtain that the AllReduce time complexity is
\[
\cO\!\left(\frac{d}{\alpha_{\bar G}(S)}\right).
\]
\end{proof}

\begin{lemma}
\label{lem:batch_time}
Let $S \subseteq [n]$ be a subset of workers. The time required to collect a batch of size $B$ using workers in $S$ under the \hyperref[box:computation_model]{\textcolor{gray}{Computation Model}} is at most
\begin{align}
\label{eq:batch_time}
\min_{m \in [\abs{S}]} 
\left[
\left(\frac{1}{m}\sum_{i=1}^{m} \frac{1}{h_{\pi_i(S)}}\right)^{-1}
\left(1 + \frac{B}{m}\right)
\right]
\end{align}
seconds, $\pi(S)$ is a permutation that sorts $\{h_i\}_{i \in S}:$ $h_{\pi_1(S)} \leq \dots \leq h_{\pi_{\abs{S}}(S)}.$ 
\end{lemma}

\newcommand{\parens}[1]{\left(#1\right)}

\begin{proof}
Let
\begin{align*}
t
=
\min_{m \in [\abs{S}]}
\parens{
\parens{\sum_{i=1}^{m} \frac{1}{h_{\pi_i(S)}}}^{-1}
(B + m)
}.
\end{align*}

As soon as a worker finishes computing a stochastic gradient, it immediately starts computing the next one. Hence, by time $t$, worker $i$ will have computed at least
\[
\left\lfloor \frac{t}{h_i} \right\rfloor
\]
stochastic gradients. Therefore the total number of gradients computed by workers in $S$ by time $t$ is at least
\begin{align*}
\sum_{i \in S} \left\lfloor \frac{t}{h_i} \right\rfloor
\ge
\sum_{i=1}^{m^*} \left\lfloor \frac{t}{h_{\pi_i(S)}} \right\rfloor,
\end{align*}
where
\begin{align*}
m^* =
\arg\min_{m \in [\abs{S}]}
\parens{
\parens{\sum_{i=1}^{m} \frac{1}{h_{\pi_i(S)}}}^{-1}
(B + m)
}.
\end{align*}

Since $\lfloor x \rfloor \ge x - 1$ for all $x \ge 0$, we obtain
\begin{align*}
&\sum_{i=1}^{m^*}
\left\lfloor
\frac{t}{h_{\pi_i(S)}}
\right\rfloor
\ge
\sum_{i=1}^{m^*}
\left(
\frac{t}{h_{\pi_i(S)}} - 1
\right)\\
&=
t \sum_{i=1}^{m^*} \frac{1}{h_{\pi_i(S)}} - m^*
=
\left(\sum_{i=1}^{m^*} \frac{1}{h_{\pi_i(S)}}\right)
\left(
\left(\sum_{i=1}^{m^*} \frac{1}{h_{\pi_i(S)}}\right)^{-1}
(B + m^*)
\right)
- m^*
=
B.
\end{align*}

Thus, by time $t$, at least $B$ stochastic gradients have been computed.
\end{proof}

\subsection{Proof of Theorem~\ref{thm:sgd_heter}}

\THEOREMSGDHETER*

\begin{proof}
  Notice that Leon SGD (Algorithm~\ref{alg:main_leon}) is a Minibatch SGD method with steps
  \begin{align}
    \label{eq:bdljpmibzuStiLIITiRA}
    x^{k+1} = x^{k} - \frac{\gamma}{n}\sum_{i=1}^n \frac{1}{B_i} \sum_{j=1}^{B_i} \nabla f_i(x^k;\xi^k_{ij}).
  \end{align}
  Similarly to \citep{tyurin2023optimal,tyurin2024optimalgraph}, using the standard SGD analysis \citep{lan2020first}, one can show that this method converges after $\cO\left(\frac{L \Delta}{\varepsilon}\right)$ iterations due to the fact that $\left(\frac{1}{n} \sum_{i=1}^n \frac{1}{B_i}\right)^{-1} \geq \max\{\lceil\nicefrac{\sigma^2}{\varepsilon}\rceil, n\} / n.$ It is left to bound the time of one iteration under Assumption~\ref{ass:time}. The time to collect the minibatch in \eqref{eq:bdljpmibzuStiLIITiRA} is the same as in \citep{tyurin2023optimal,tyurin2024optimalgraph} and can be bounded by
  \begin{align}
    \label{eq:ImunmOozIc}
    \cO\left(\max\left\{\max\limits_{i \in [n]} h_i, \frac{\sigma^2}{n \varepsilon}\left(\frac{1}{n} \sum\limits_{i=1}^n h_i\right)\right\}\right)
  \end{align}
  (e.g., see Theorem~A.4 in \citep{tyurin2023optimal}). The time to run AllReduce can be upper bounded by 
  \begin{align}
    \label{eq:DdmpIeQVMaaycgCcH}
    \cO\left(\frac{d}{\min\limits_{\{i,j\} \in F} w_{ij}}\right)
  \end{align}
  due to Theorem~\ref{thm:allreduce_time}, since $\bar{w}_1 \eqdef \min\limits_{\{i,j\} \in F} w_{ij}$ is the smallest value of a min-cut in the graph $\bar{G}$, which is the same as the value of a minimum $[n]$-cut in the graph $\bar{G}.$ It remains to sum \eqref{eq:ImunmOozIc} and \eqref{eq:DdmpIeQVMaaycgCcH} and multiply them by $\cO\left(\frac{L \Delta}{\varepsilon}\right)$.
\end{proof}

\subsection{Proof of Corollary~\ref{cor:cor_sparse}}
\label{sec:cor_sparse}
\CORSPARSE*
\begin{proof}
  Under the new assumptions, the lower bound in Theorem~\ref{thm:main}, proved for both the \emph{homogeneous} and \emph{heterogeneous} settings, is greater than or equal to
  \begin{align}
    \textstyle \tilde{\Omega}\left(\min\left\{\underbrace{\textstyle \frac{d L \Delta}{{\color{orange} p} b \varepsilon} + \frac{h L \Delta}{\varepsilon} +  \frac{h \sigma^2 L \Delta}{n \varepsilon^2}}_{L_1 \eqdef }, \underbrace{\textstyle \frac{h L \Delta}{\varepsilon} + \frac{h \sigma^2 L \Delta}{\varepsilon^2}}_{L_2 \eqdef }\right\}\right)
  \end{align}
  since the weights in the corresponding Gomory-Hu tree are less than or equal to $p b.$ On the other hand, the upper bound by Leon SGD in Theorem~\ref{thm:sgd_heter} is less than or equal $U \eqdef \cO\big(\nicefrac{d L \Delta}{b \varepsilon} + \nicefrac{h L \Delta}{\varepsilon} + \nicefrac{h \sigma^2 L \Delta}{n \varepsilon^2}\big)$ since all edges have the same bandwidth $b.$ Hero SGD with complexity $L_2$ can potentially improve $U$ when $\varepsilon$ is not too large and $n$ is small. However, in the regime when $\varepsilon$ is small and $\sigma$ is large ($\nicefrac{\sigma^2}{\varepsilon} \geq \Omega(\nicefrac{d}{b h})$), the term $L_1$ can be smaller than or equal to $L_2,$ and in this case, comparing $L_1$ and $U,$ one can see that $L_1$ can improve the communication term by at most $p$ times. If $p = \Theta(1),$ i.e., the graph is sparse, then it is infeasible to improve the time complexity $U$ even in the homogeneous setting. Note that $U$ can also be achieved with a naive mini-batch version of Synchronous SGD: $x^{k+1} = x^k - \frac{\gamma}{n B} \sum_{i=1}^{n} \sum_{j=1}^{B} \nabla f_i(x^k;\xi^k_{ij})$ with the simple communication approach from Section~\ref{sec:sync}.
\end{proof}

\subsection{Proof of Corollaries~\ref{cor:cor_sparse_two} and \ref{cor:cor_sparse_three}}
\label{sec:cor_sparse_two}

\begin{restatable}[Proof in Section~\ref{sec:cor_sparse_two}]{corollary}{CORSPARSETWO}
  \label{cor:cor_sparse_two}
  In view of Theorems~\ref{thm:main}, if $h_i = h$ for all $i \in [n]$ and $b_{ij} = b$ for all $(i,j) \in E,$ then the lower bound is 
  $\textstyle \tilde{\Omega}\left(\min\left\{\textstyle \min_{m \in [n - 1]}\left[\frac{d L \Delta}{{\color{orange} m} b \varepsilon} +  \frac{h \sigma^2 L \Delta}{{\color{orange} n(m)} \varepsilon^2}\right] + \frac{h L \Delta}{\varepsilon}, \frac{h L \Delta}{\varepsilon} + \frac{h \sigma^2 L \Delta}{\varepsilon^2}\right\}\right),$ where ${\color{orange} n(m)}$ is the number of nodes having the number of incident edges greater than or equal ${\color{orange} m}.$
\end{restatable}

\begin{proof}
  Using \eqref{eq:BNjTPNhGYfaixu}, the lower bound is
  \begin{align*}
    LB \eqdef \tilde{\Omega}\left(\min\limits_{k \in [n]} \left(\frac{d}{\bar{w}_{k}} + \min\limits_{p \in [k]} \frac{h \sigma^2}{\abs{S_{k,p}} \varepsilon}\right) \frac{L \Delta}{\varepsilon} + \frac{h L \Delta}{\varepsilon}\right)
  \end{align*}
  in both homogeneous and heterogeneous settings when $h_i = h$ for all $i \in [n].$ Clearly,
  \begin{align*}
    LB = \tilde{\Omega}\left(\left(\frac{d}{\bar{w}_{k^*}} + \frac{h \sigma^2}{\abs{S_{k^*,p^*}} \varepsilon}\right) \frac{L \Delta}{\varepsilon} + \frac{h L \Delta}{\varepsilon}\right)
  \end{align*}
  for some $k^* \in [n]$ and $p^* \in [k^*].$ If $\abs{S_{k^*,p^*}} = 1,$ then Corollary~\ref{cor:cor_sparse_two} is true. Let $\abs{S_{k^*,p^*}} \geq 2,$ then necessarily $k^* \in [n - 1].$ Notice that $\bar{w}_{k^*} = i_{k^*} b$ for some $i_{k^*} \in [n - 1]$ since all bandwidths equal to $b.$ Thus
  \begin{align*}
    LB = \tilde{\Omega}\left(\left(\frac{d}{b i_{k^*}} + \frac{h \sigma^2}{\abs{S_{k^*,p^*}} \varepsilon}\right) \frac{L \Delta}{\varepsilon} + \frac{h L \Delta}{\varepsilon}\right).
  \end{align*}
  It is left to show that $\abs{S_{k^*,p^*}} \leq n(i_{k^*}).$ By construction, $S_{k^*,p^*}$ is a set of nodes in which each node has at least $i_{k^*}$ incident edges (if one of them had fewer than $i_{k^*}$ incident edges, then the nodes in $S_{k^*,p^*}$ would have been disconnected in Algorithm~\ref{alg:preprocess} at some iteration $j \in [k^* - 1]$, because there would exist an edge $e$ in $T$ with $w_e < b i_{k^*}$ that separates two nodes in $S_{k^*,p^*}$; this leads to a contradiction). Thus, $\abs{S_{k^*,p^*}} \leq n(i_{k^*})$ and 
  \begin{align*}
    LB 
    &\geq \tilde{\Omega}\left(\left(\frac{d}{b i_{k^*}} + \frac{h \sigma^2}{n(i_{k^*}) \varepsilon}\right) \frac{L \Delta}{\varepsilon} + \frac{h L \Delta}{\varepsilon}\right) \geq \tilde{\Omega}\left(\min_{m \in [n - 1]} \left(\frac{d}{b m} + \frac{h \sigma^2}{n(m) \varepsilon}\right) \frac{L \Delta}{\varepsilon} + \frac{h L \Delta}{\varepsilon}\right).
  \end{align*}
\end{proof}

\CORSPARSETHREE*
\begin{proof}
  Similarly to the previous proof, a lower bound is
  \begin{align*}
    LB = \tilde{\Omega}\left(\left(\frac{d}{\bar{w}_{k^*}} + \frac{h \sigma^2}{\abs{S_{k^*,p^*}} \varepsilon}\right) \frac{L \Delta}{\varepsilon} + \frac{h L \Delta}{\varepsilon}\right)
  \end{align*}
  for some $k^* \in [n]$ and $p^* \in [k^*].$ If $\abs{S_{k^*,p^*}} = 1,$ then the lower bound is true. Let $\abs{S_{k^*,p^*}} \geq 2$ and $\bar{d}_1 \geq \dots \geq \bar{d}_{\abs{S_{k^*,p^*}}}$ be the node degrees of the nodes in set $S_{k^*,p^*}.$ Notice that $\bar{w}_{k^*} \leq b \times \bar{d}_{\abs{S_{k^*,p^*}}}$ (the case $\bar{w}_{k^*} > b \times \bar{d}_{\abs{S_{k^*,p^*}}}$ is impossible since there exists a node in $S_{k^*,p^*}$ with degree $\bar{d}_{\abs{S_{k^*,p^*}}},$ which would be separated by Algorithm~\ref{alg:preprocess} in some iteration $j \in [k^* - 1]$ from another node in $S_{k^*,p^*}$). Thus,
  \begin{align*}
    LB 
    &\geq \tilde{\Omega}\left(\left(\frac{d}{\bar{d}_{\abs{S_{k^*,p^*}}} b} + \frac{h \sigma^2}{\abs{S_{k^*,p^*}} \varepsilon}\right) \frac{L \Delta}{\varepsilon} + \frac{h L \Delta}{\varepsilon}\right) \\
    &\geq \tilde{\Omega}\left(\left(\frac{d}{k(\abs{S_{k^*,p^*}}) b} + \frac{h \sigma^2}{\abs{S_{k^*,p^*}} \varepsilon}\right) \frac{L \Delta}{\varepsilon} + \frac{h L \Delta}{\varepsilon}\right) \\
    &\geq \tilde{\Omega}\left(\min_{m \in \{2, \dots, n\}} \left(\frac{d}{k(m) b} + \frac{h \sigma^2}{m \varepsilon}\right) \frac{L \Delta}{\varepsilon} + \frac{h L \Delta}{\varepsilon}\right)
  \end{align*}
  where we use that $\bar{d}_{\abs{S_{k^*,p^*}}} \leq k(\abs{S_{k^*,p^*}}).$
\end{proof}

\section{Lower Bounds}
\subsection{Proof sketch}
\label{sec:proof_sketch}
In this section, we give a proof sketch of Theorem~\ref{thm:sgd_homog}. 

\textbf{(``Worst-case'' function and stochastic oracle).} In the first step of the proof, we follow \citep{carmon2020lower,arjevani2022lower,huang2022lower,tyurin2023optimal} and construct a ``worst-case'' function. However, one crucial detail is that we use the function $F_{T,K,a}$ from \citep{tyurin2025proving} instead of the function from \citep{carmon2020lower}. Recall the function $F_{T}$ from \citep{carmon2020lower}, which has two important properties: i) if an algorithm wants to find an $\varepsilon$--stationary point, then it is necessary to discover\footnote{In the paper, when we say that a worker $i$ \emph{discovers} a coordinate with index $j$, it means that it adds a vector to $I_i$ in which the corresponding value of that coordinate is non-zero.} the last $T$\textsuperscript{th} coordinate; ii) an algorithm can discover the new coordinate $j + 1$ only if the $j$\textsuperscript{th} coordinate is discovered. The new construction $F_{T,K,a}$ by \citet{tyurin2025proving} generalizes the last property: an algorithm can discover the new coordinate $j + 1$ only if the $j$\textsuperscript{th}, $(j - 1)$\textsuperscript{th}, \dots, $(j - K + 1)$\textsuperscript{th} coordinates are discovered. The stochastic oracle is the same as in \citep{arjevani2022lower}: it simply zeros out the newly discovered coordinate, the one with the largest index among all discovered coordinates, with probability $1 - p_{\sigma}$, where $p_{\sigma} = \Theta\left(\nicefrac{\varepsilon}{\sigma^2}\right)$. 

\textbf{(Time complexity).} Recall Algorithm~\ref{alg:preprocess} that generates the sequences $\{\bar{w}_k\}_{k \in [n]}$ and $\{S_{k,p}\}_{k \in [n], p \in [k]}$ (see the description in Section~\ref{sec:grace}). Now, consider the lower bound \eqref{eq:BNjTPNhGYfaixu}, which is equivalent to 
\begin{align}
  \label{eq:HQRYFCvv}
  t^* \eqdef \min_{k \in [n]} \bar{t}(k).
\end{align}
where
\newcommand{\tktk}{\frac{1}{c_1 \log^{c_2} (n + 1)} \times \max\left\{\frac{d}{\bar{w}_{k}}, \min_{p \in [k]}B_{h}(\nicefrac{\sigma^2}{\varepsilon},S_{k,p})\right\} \frac{L \Delta}{\varepsilon}}
\begin{align}
  \label{eq:XoNhojP}
  \bar{t}(k) \eqdef \tktk,
\end{align}
$c_1 \eqdef 2^{43} 3^7 5^{14} \pi^2 e^{10},$ $c_2 \eqdef 14,$ $B_{h}(\nicefrac{\sigma^2}{\varepsilon},S) \eqdef \min_{m \in [\abs{S}]} \left[\left(\frac{1}{m} \sum_{i=1}^{m} \frac{1}{h_{\pi_{i}(S)}}\right)^{-1} \left(1 + \frac{\sigma^2}{\varepsilon m}\right)\right],$
and $\pi(S)$ is a permutation that sorts $\{h_i\}_{i \in S}:$ $h_{\pi_1(S)} \leq \dots \leq h_{\pi_{\abs{S}}(S)}.$ 

Our goal now is to show that \eqref{eq:HQRYFCvv} is a valid lower bound. Notice that $\bar{w}_1 \leq \dots \leq \bar{w}_n$, and $\{S_{k + 1,p}\}_{p \in [k + 1]}$ is constructed from $\{S_{k,p}\}_{p \in [k]}$ by splitting one of the sets.

\textbf{High-level intuition of why \eqref{eq:XoNhojP} is valid lower bound:} Assume that an algorithm decides to use all workers ($k = 1$). In this case, all workers are allowed to work together to obtain the best possible computation term $B_{h}(\nicefrac{\sigma^2}{\varepsilon},S_{1,1}) = B_{h}(\nicefrac{\sigma^2}{\varepsilon},[n])$. However, intuitively, the algorithm must pay the communication price $\nicefrac{\cdot}{\bar{w}_1}$ because $\bar{w}_1$ is the smallest possible value of a min-cut in the graph $G$, and there exist two workers such that the maximum flow between them is bounded by $\bar{w}_1$. \textbf{No matter what routing strategy the algorithm chooses, if it wants to transfer a vector of size $\ell$ between these two workers, it is necessarily required to wait $\nicefrac{\ell}{\bar{w}_1}$ seconds. Therefore, the only way to remove the dependence on $\nicefrac{\cdot}{\bar{w}_1}$ is for the algorithm to ``disconnect'' the workers separated by the bottleneck min-cut $\bar{w}_1$.} This way, we get the sets $S_{2,1}$ and $S_{2,2}.$ Notice that $\min\{B_{h}(\nicefrac{\sigma^2}{\varepsilon},S_{2,1}), B_{h}(\nicefrac{\sigma^2}{\varepsilon},S_{2,2})\} \leq B_{h}(\nicefrac{\sigma^2}{\varepsilon},S_{1,1}).$

Recall that $\{\bar{w}_k\}_{k \in [n]}$ and $\{\min_{p \in [k]}B_{h}(\nicefrac{\sigma^2}{\varepsilon},S_{k,p})\}_{k \in [n]}$ are non-decreasing. Thus, expect one corner case when $\bar{w}_1$ is large, it means that there exists $\bar{k}$ such that $\nicefrac{d}{\bar{w}_{\bar{k}}} \approx \min_{p \in [\bar{k}]}B_{h}(\nicefrac{\sigma^2}{\varepsilon},S_{\bar{k},p}).$ In particular, we can show that there exists $\bar{k}$ such that
\begin{align}
  \label{eq:YFHbZlt}
  t^* = \frac{1}{c_1 \log^{c_2} (n + 1)} \times \min\left\{\frac{d}{\bar{w}_{\bar{k}}}, \min_{p \in [\bar{k} + 1]}B_{h}(\nicefrac{\sigma^2}{\varepsilon},S_{\bar{k} + 1,p})\right\} \frac{L \Delta}{\varepsilon}.
\end{align} 
In other words, either $\bar{k}$ or $\bar{k}+1$ is the optimal index that balances the two sequences, and the partition $\{S_{\bar{k}+1,p}\}_{p \in [\bar{k}+1]}$ is one that an optimal algorithm would choose to identify the best subset of workers (see Figure~\ref{fig:meta_tree_partition}). Thus, it remains to use this partition of workers and show that, with this partition, the workers would require \eqref{eq:YFHbZlt} seconds with high probability.
\begin{figure}[t]
\centering
\scalebox{0.8}{
\begin{tikzpicture}[
    scale=1,
    every node/.style={font=\small},
    meta/.style={
        circle,
        draw=black,
        thick,
        minimum size=1.5cm,
        fill=blue!8
    },
    edge/.style={
        thick
    }
]

\node[meta] (s1) at (0,3.6) {$S_{\bar{k}+1,1}$};

\node[meta] (s2) at (-5,1.6) {$S_{\bar{k}+1,2}$};
\node[meta] (s3) at (0,1.6) {$S_{\bar{k}+1,3}$};
\node[meta] (s4) at (5,1.6) {$S_{\bar{k}+1,4}$};

\node[meta] (s9) at (-6,-0.2) {$S_{\bar{k}+1,9}$};
\node[meta] (s8) at (-4,-0.2) {$S_{\bar{k}+1,8}$};

\node[meta] (s7) at (-6.5,-2.0) {$S_{\bar{k}+1,7}$};

\node[meta] (s5) at (-7,-3.8) {$S_{\bar{k}+1,5}$};
\node[meta] (s6) at (-3,-3.8) {$S_{\bar{k}+1,6}$};
\node[meta] (sk) at (3,-3.8) {$S_{\bar{k}+1,\bar{k}}$};
\node[meta] (sk1) at (7,-3.8) {$S_{\bar{k}+1,\bar{k}+1}$};

\node (d1) at (-0.9,-3.8) {$\ldots$};
\node (d2) at (0.9,-3.8) {$\ldots$};

\draw[edge] (s1) -- node[midway,left,yshift=10pt] {$\leq \bar w_{\bar{k}}$} (s2);
\draw[edge] (s1) -- node[midway,right] {$\leq \bar w_{\bar{k}}$} (s3);
\draw[edge] (s1) -- node[midway,right,yshift=10pt] {$\leq \bar w_{\bar{k}}$} (s4);

\draw[edge] (s2) -- node[midway,left,yshift=10pt] {$\leq \bar w_{\bar{k}}$} (s9);
\draw[edge] (s2) -- node[midway,right,yshift=10pt] {$\leq \bar w_{\bar{k}}$} (s8);

\draw[edge] (s9) -- node[midway,left,yshift=5pt] {$\leq \bar w_{\bar{k}}$} (s7);

\draw[edge] (s7) -- node[midway,left,yshift=5pt] {$\leq \bar w_{\bar{k}}$} (s5);
\draw[edge] (s8) -- node[midway,right,yshift=10pt] {$\leq \bar w_{\bar{k}}$} (s6);

\draw[edge] (s3) -- node[midway,left,yshift=10pt] {$\leq \bar w_{\bar{k}}$} (d1);
\draw[edge] (s3) -- node[midway,right,yshift=10pt] {$\leq \bar w_{\bar{k}}$} (d2);

\draw[edge] (s4) -- node[midway,left,yshift=10pt] {$\leq \bar w_{\bar{k}}$} (sk);
\draw[edge] (s4) -- node[midway,right,yshift=10pt] {$\leq \bar w_{\bar{k}}$} (sk1);

\end{tikzpicture}
}

\caption{
Tree of metanodes corresponding to the partition $\{S_{\bar{k}+1,p}\}_{p\in[\bar{k}+1]}$. Each node represents a subset of workers, and every edge corresponds to a communication link whose bandwidth is at most $\bar w_{\bar{k}}$.
}
\label{fig:meta_tree_partition}
\end{figure}
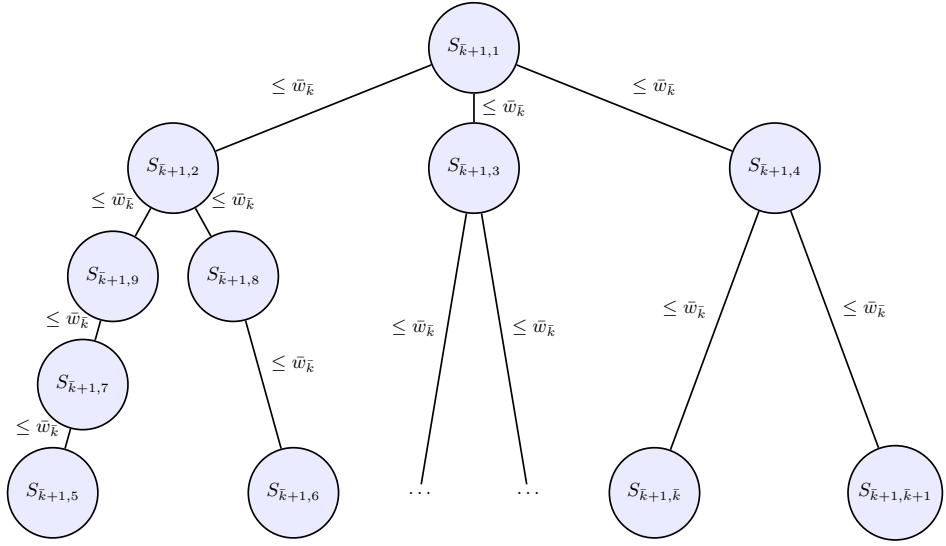

(\textbf{Analysis of the partition}.) In Figure~\ref{fig:meta_tree_partition}, consider the (meta) node $S_{\bar{k}+1,5}$. For simplicity, assume that $h_i = h$ for all $i \in [n]$.
There are $\abs{S_{\bar{k}+1,5}}$ workers in $S_{\bar{k}+1,5}$. If $S_{\bar{k}+1,5}$ were on its own and isolated, then the required time to find an $\varepsilon$--stationary point would be $$\Theta\left(B_{h}(\nicefrac{\sigma^2}{\varepsilon},S_{\bar{k}+1,p})\right) \frac{L \Delta}{\varepsilon} = h \times \left(\frac{L \Delta}{\varepsilon} + \frac{\sigma^2 L \Delta}{\abs{S_{\bar{k}+1,p}} \varepsilon^2}\right)$$ \citep{arjevani2022lower,tyurin2023optimal}. However, in the tree, $S_{\bar{k}+1,5}$ is connected to other workers through \emph{one} edge with weight $\leq \bar{w}_k.$ It means that the maximal number of coordinates per second received from $[n] \setminus S_{\bar{k}+1,5}$ is bounded by $\bar{w}_k.$

In the construction of the ``worst-case'' function, we randomly permute the coordinates, meaning that when workers $[n] \setminus S_{\bar{k}+1,5}$ send a sequence of coordinates to $S_{\bar{k}+1,5}$, the probability of sending the ``right'' coordinate is less than or equal to $\tilde{\cO}\left(\nicefrac{1}{d}\right)$. Thus, on average, they have to send $d$ coordinates to provide a ``useful'' coordinate, which would take at least $\nicefrac{d}{\bar{w}_{\bar{k}}}$ seconds. In this way, we can show that it would require at least $\approx \min\left\{\frac{d}{\bar{w}_{\bar{k}}}, B_{h}(\nicefrac{\sigma^2}{\varepsilon}, S_{\bar{k}+1,5})\right\} \frac{L \Delta}{\varepsilon}$ seconds to solve the problem by one of the workers from $S_{\bar{k}+1,5}.$

(\textbf{Recursive analysis}.) The final main challenge was to extend this idea not only to the leaves ($\bar{\mathcal{L}}_1 = \{S_{\bar{k}+1,5}, S_{\bar{k}+1,6}, S_{\bar{k}+1,\bar{k}}, S_{\bar{k}+1,\bar{k} + 1}, \dots\}$), but to all nodes in Figure~\ref{fig:meta_tree_partition}. Moreover, it should hold for all nodes at the same time. In the full proof, we discovered a recursive technique and the \emph{leaf-branch peeling} procedure that, starting from the leaves, recursively proves a similar result for all nodes. Roughly speaking, using the Chernoff's method, we show 
\begin{align*}
  \approx \min\left\{\frac{d}{\bar{w}_{\bar{k}}}, B_{h}(\nicefrac{\sigma^2}{\varepsilon}, S)\right\} \frac{L \Delta}{\varepsilon}
\end{align*}
is a lower bound for all leaves $S \in \bar{\mathcal{L}}_1,$ including $S_{\bar{k}+1,5}.$ Then, we show that this is true for all $S \in \bar{\mathcal{B}}_1,$ where $\bar{\mathcal{B}}_1$ is the set of ``line nodes'' incident to $\bar{\mathcal{L}}_1.$ In Figure~\ref{fig:meta_tree_partition}, $\bar{\mathcal{B}}_1 = \{S_{\bar{k}+1,7}, S_{\bar{k}+1,9}, S_{\bar{k}+1,8}\}.$ Then, we ``remove'' $\bar{\mathcal{L}}_1 \cup \bar{\mathcal{B}}_1$ and repeat the steps again. In the next step, the new ``leaves'' $\bar{\mathcal{L}}_2$ are $S_{\bar{k}+1,2}$ and $S_{\bar{k}+1,4}.$ Notice that $S_{\bar{k}+1,2}$ has \emph{one edge} connected to $S_{\bar{k}+1,1}$ and two edges connected to $\bar{\mathcal{L}}_1 \cup \bar{\mathcal{B}}_1.$ Thus, it has many more sources from which to discover a new coordinate. Nevertheless, recursively, we have already shown that $\bar{\mathcal{L}}_1 \cup \bar{\mathcal{B}}_1$ cannot provide much useful information to $S_{\bar{k}+1,2};$ thus, the only possible direction is the edge corresponding to $S_{\bar{k}+1,1}.$ Controlling the ``information leak'' from only one edge is possible and easier.

In the proof, due to the union bounds, Chernoff's method ``accumulates probabilities,'' and we obtain an exponential dependence on the number of steps $\bar{d}$ in the \emph{leaf-branch peeling} procedure. Luckily, we can show that $\bar{d} = \cO(\log n),$ which is sufficient to obtain an optimal lower bound up to logarithmic factors.

The proof of Theorem~\ref{thm:sgd_heter} is arguably more straightforward and self-contained, and should be clear once the proof of Theorem~\ref{thm:sgd_homog} is understood.

\subsection{Full proof}
In this section, we provide the full proof of the lower bound in the homogeneous setting. We consider Protocol~\ref{alg:simplified_time_multiple_oracle_protocol} and Assumption~\ref{ass:compressors}.
\label{sec:proof_full}
\begin{protocol}[h]
  \caption{}
  \label{alg:simplified_time_multiple_oracle_protocol}
  \begin{algorithmic}[1]
  \STATE \textbf{Input:} Algorithm $A$
  \STATE Init $I_i = \emptyset$ (all available information) on worker $i$ for all $i \in [n]$
  \STATE Run the following two loops in parallel on the workers.
  \FOR{$i = 1, \dots, n$ (in parallel on the workers)}
    \WHILE{true}
    \STATE Algorithm $A$ calculates a new point $x$ based on local information $I_i$: \hfill (takes $0$ seconds)\\
    any vector $x \in \R^d$ such that $\textnormal{supp}(x) \subseteq \cup_{y \in I_i} \textnormal{supp}(y)$ \hfill ($\textnormal{supp}(v) \eqdef \{\, i \in [d] : v_i \neq 0 \,\}$)
    \STATE Calculate one stochastic gradient $\nabla f(x;\xi),$ $\quad \xi \sim \mathcal{D}_{\xi}$ \, ($\xi$ are i.i.d.) \hfill (takes $h_i$ seconds)
    \STATE Add $\nabla f(x;\xi)$ to $I_i$ \hfill (takes $0$ seconds)
    \STATE Optionally wait before starting the next calculation
    \ENDWHILE
  \ENDFOR
  \FOR{$i = 1, \dots, n$ (in parallel on the workers)}
  \FOR{$k = 0, 1, \dots$}
    \STATE Algorithm $A$ calculates a new point $s$ based on local information $I_i$: \hfill (takes $0$ seconds) \\
    any vector $s \in \R^d$ such that $\textnormal{supp}(s) \subseteq \cup_{y \in I_i} \textnormal{supp}(y)$
    \STATE Algorithm $A$ chooses any destination worker $j$ \hfill (takes $0$ seconds)
    \STATE Compute the transformation $\cC^k_i(s;\zeta),$ $\quad \zeta \sim \mathcal{D}_{\zeta}$ \hfill (takes $0$ seconds; Assumption~\ref{ass:compressors}) 
    \STATE Asynchronously send $\cC^k_i(s;\zeta)$ to worker $j$ through $G$ using any feasible routing strategy (the protocol does not wait for the send to complete); add\textsuperscript{\color{blue} (a)} $\cC^k_i(s;\zeta)$ to $I_j$ once it arrives \\
    \hfill (the time depends on $G$, the chosen routing strategy, and the current network load)
    \STATE Optionally wait before starting the next send
  \ENDFOR
  \ENDFOR

  (a vector may be added to $I_j$ while the algorithm computes a new point; in this case, the protocol first adds the vector, incurring no delay since the operation takes $0$ seconds) \\
  {\color{blue} (a)}: The worker that receives the compressed vector decompresses using the operation $\sum_{(\nu, v) \in \cC^k_i(s;\zeta)} v \cdot e_{\nu},$ where $e_{\nu}$ is the $\nu$\textsuperscript{th} vector of the standard basis.
  \end{algorithmic}
\end{protocol}
\begin{assumption}[Predefined and Random Sparsifiers]
  \label{ass:compressors}
  Assume that the optimized function $f \,:\, \R^d \to \R$ is a randomly generated function such that $f(x) = \hat{f}(P x),$ where $\hat{f}$ is a deterministic function and $P$ is a random permutation drawn from a distribution $\mathcal{D}_P.$ For all $i \in [n]$ and $k \geq 0,$ an algorithm $A$ is allowed to use any mapping $\cC^k_i\,:\,\R^d \times \mathbb{S}_{\zeta} \to [d]^{p^k_i} \times \R^{p^k_i}$ with any $p^k_i \geq 1$ such that 
  $[\cC^k_i(s;\zeta)]_j = (\nu^k_{i,j}, c^k_{i,j} \cdot [s]_{\nu^k_{i,j}})$ for all $s \in \R^d,$ $\zeta \in \mathbb{S}_{\zeta},$ and $j \in [p^k_i],$ where $c^k_{i,j} \equiv c^k_{i,j}(\zeta) \in \R$ is an arbitrary random value, and $\nu^k_{i,j} = \nu^k_{i,j}(\zeta) \in [d]$ is an arbitrary random coordinate such that the distribution of $\{\nu^k_{i,j}\}_{k \geq 0, i \in [n], j \in [p^k_i]}$
  is independent of $P \sim \mathcal{D}_{P}.$
\end{assumption}

In other words, if we change the optimized function or permute the coordinates in $f$, then the generated indices do not change. Assumption~\ref{ass:compressors} is one way of expressing that the workers and the algorithm \emph{do not} take into account the local information and do not depend on $\{I_i\}_{i \in [n]}$ or $s$ when choosing the indices in the sparsifiers. It is general enough to support sending the full vector, sending a predefined block of the vector as is done in AllReduce algorithms, or even a random subset of coordinates, supporting Rand$K$ or Perm$K$ compressed communication. For instance, for all $K \in [d]$, Rand$K$, which sends a random subset of coordinates (scaled by $\nicefrac{d}{K}$) satisfies this assumption with $p = K$ because it samples subsets independently.

\MAINTHEOREM*
\begin{proof}
  \mbox{}\\
  \textbf{(Step 1: ``Worst-case'' function).}
  In our proof, we use a slightly modified function by \citet{tyurin2025proving}, which is based on the ``worst-case'' function by \citet{carmon2020lower}. For any $T, K \in \N,$ and $e \geq a > 1,$ \citet{tyurin2025proving} defined the function $F_{T,K,a} \,:\, \R^T \to \R$ such that
\begin{align}
  \label{eq:worst_case_copy}
  F_{T,K,a}(x) = -\sum_{i=1}^T \Psi_a(x_{i-K}) \dots \Psi_a(x_{i-2}) \Psi_a(x_{i-1})\Phi(x_i) + \sum_{i=1}^T \Gamma(x_i),
\end{align}
where $x_i$ is the $i$\textsuperscript{th} coordinate of a vector $x \in \R^T$ and
\begin{align*}
    \Psi_a(x) = \begin{cases}
        0, & x \leq 1/2, \\
        \exp\left(\log a \cdot \left(1 - \frac{1}{(2x - 1)^2}\right)\right), & x > 1/2,
    \end{cases}
    \qquad
    \Phi(x) = \sqrt{e} \int_{-\infty}^{x}e^{-\frac{1}{2}t^2}dt,
\end{align*}
and
\begin{align*}
  \Gamma(x) = \begin{cases}
    -x e^{1/x + 1},& x < 0, \\
    0, & x \geq 0.
\end{cases}
\end{align*}
We assume that $x_{0} = \dots = x_{-K + 1} \equiv 1.$ Throughout the lower bound analysis, we assume that $e \geq a > 1$ in $\Psi_a.$ We also define 
\begin{align*}
  \textnormal{supp}(v) \eqdef \{\, i \in [d] : v_i \neq 0 \,\} \qquad \forall v \in \R^d
\end{align*}
and
\begin{align*}
  &\textnormal{prog}^K(x) \eqdef \max \{i \geq 0\,|\,x_i \neq 0, x_{i - 1} \neq 0, \dots, x_{i - K + 1} \neq 0\} \\
  &(x_{0} = \dots = x_{-K + 1} \equiv 1),
\end{align*}
which extends the standard progress operator $\textnormal{prog}(x) \equiv \textnormal{prog}^1(x) \eqdef \max \{i \geq~0\,|\,x_i \neq 0\} \,\, (x_0 \equiv 1).$
This function has the following properties proved in \citep{tyurin2025proving}:
\begin{restatable}{lemma}{LEMMATHREE}
  \label{lemma:prog}
  For all $x \in \R^T,$ $\textnormal{supp}(\nabla F_{T,K,a}(x)) \subseteq \{1, \dots, \textnormal{prog}^K(x) + 1\} \cup \textnormal{supp}(x),$ where $\textnormal{supp}(v) \eqdef \{\, i \in [d] : v_i \neq 0 \,\}$ for all $d \geq 1$ and $v \in \R^d$ 
\end{restatable}
\begin{restatable}{lemma}{LEMMAFOUR}
  \label{lemma:worst_function_general}
  For all $x \in \R^T,$ if $\textnormal{prog}^K(x) < T,$ then $\norm{\nabla F_{T,K,a}(x)} > 1.$
\end{restatable}

\begin{restatable}{lemma}{LEMMAFIVE}
  \label{lemma:delta}
  Function $F_{T,K,a}$ satisfies
  \begin{align*}
    F_{T,K,a}(0) - \inf_{x \in \R^T} F_{T,K,a}(x) \leq \Delta^0(K, a) \cdot T,
  \end{align*}
  where $\Delta^0(K, a) \eqdef \sqrt{2 \pi e} \cdot a^K.$
\end{restatable}

\begin{restatable}{lemma}{LEMMASIX}
  \label{lemma:bound_grad}
  For all $x \in \R^T,$ $\norm{\nabla F_{T,K,a}(x)}_{\infty} \leq \gamma_{\infty}(K,a),$ where $\gamma_{\infty}(K,a) \eqdef 6 \sqrt{2 \pi} e^{3/2} \cdot \frac{K a^K}{\sqrt{\log a}}.$
\end{restatable}

\begin{restatable}{lemma}{LEMMASEVEN}
  \label{lemma:smooth}
  The function $F_{T,K,a}$ is $\ell_1(K, a)$--smooth, i.e., $\norm{\nabla^2 F_{T,K,a}(x)} \leq \ell_1(K, a)$
  for all $x \in \R^T,$ where $\ell_1(K, a) \eqdef 12 \sqrt{2 \pi} e^{5/2} \cdot \frac{K^2 a^{K}}{\log a}.$
\end{restatable}

Using this construction, we define a scaled version with random coordinates. We first sample a uniformly random subset $R_T = [R_{T,1}, \dots, R_{T,T}]$ without repetitions from set $[d]$ ($R_{T,1} < \dots < R_{T,T}$). Let us take any $\lambda > 0,$ $d, T \in \N,$ $d \geq T,$ and take the function $f\,:\,\R^d \to \R$ such that
  \begin{align}
    \label{eq:SypbeSSj}
    f(x) \eqdef \frac{L \lambda^2}{\ell_1(K, a)} F_{T,K,a}\left(\frac{x_{[R_T]}}{\lambda}\right),
  \end{align}
  where $\ell_1(K, a)$ is defined in Lemma~\ref{lemma:smooth} and 
  $x_{[R_T]} \in \R^T$ is the subvector of size $T$ of vector $x \in \R^d$ such that $[x_{[R_T]}]_{i} = [x]_{R_{T,i}}$ for all $i \in [T].$ Notice that the $d - T$ coordinates are artificial. We have to ensure that $f$ is $L$-smooth and $f(0) - \inf_{x \in \R^d} f(x) \leq \Delta,$ Using Lemma~\ref{lemma:smooth},
  \begin{align*}
      \norm{\nabla f(x) - \nabla f(y)} 
      &= \frac{L \lambda}{\ell_1(K, a)} \norm{\nabla F_{T,K,a}\left(\frac{x_{[R_T]}}{\lambda}\right) - \nabla F_{T,K,a}\left(\frac{y_{[R_T]}}{\lambda}\right)} \leq L \lambda \norm{\frac{x_{[R_T]}}{\lambda} - \frac{y_{[R_T]}}{\lambda}} \\
      &= L \norm{x_{[R_T]} - y_{[R_T]}} \leq L \norm{x - y} \quad \forall x, y \in \R^d.
  \end{align*}
  We choose $$T = \left\lfloor\frac{\Delta \cdot \ell_1(K, a)}{L \lambda^2 \cdot \Delta^0(K, a)}\right\rfloor.$$ Due to Lemma~\ref{lemma:delta},
  \begin{align*}
      f(0) - \inf_{x \in \R^d} f(x) = \frac{L \lambda^2}{\ell_1(K, a)} (F_{T,K,a}\left(0\right) - \inf_{x \in \R^T} F_{T,K,a}(x)) \leq \frac{L \lambda^2 \Delta^0(K, a) T}{\ell_1(K, a)} \leq \Delta,
  \end{align*}
  where $\Delta^0(K, a)$ is defined in Lemma~\ref{lemma:delta}.
  We also take
  \begin{align}
  \label{eq:lambda}
  \lambda = \frac{\sqrt{2 \varepsilon} \ell_1(K, a)}{L}
  \end{align}
  to get
  \begin{align}
    \label{eq:ElKnuvDAUgPNjjaLAn}
    \norm{\nabla f(x)}^2 = \frac{L^2 \lambda^2}{\ell_1^2(K, a)}\norm{\nabla F_{T,K,a} \left(\frac{x_{[R_T]}}{\lambda}\right)}^2 = 2 \varepsilon \norm{\nabla F_{T,K,a} \left(\frac{x_{[R_T]}}{\lambda}\right)}^2 > 2 \varepsilon \cdot \mathbbm{1}\left[\textnormal{prog}^K(x_{[R_T]}) < T\right],
  \end{align} 
  where the last inequality due to Lemma~\ref{lemma:worst_function_general}. Note that
  \begin{align}
    \label{eq:WzHRTJlPowGgA}
    T = \left\lfloor\frac{L \Delta}{2 \Delta^0(K, a) \cdot \ell_1(K, a) \cdot \varepsilon}\right\rfloor.
  \end{align}
\textbf{(Step 2: Stochastic Oracle).}
We consider a stochastic oracle similar to \citep{arjevani2022lower}.
We define
  \begin{align}
    \label{eq:stoch_constr}
    [\nabla f(x; \xi)]_j \eqdef 
    \begin{cases}
      [\nabla f(x)]_j \times \left(\frac{\xi}{p_{\sigma}}\right), & \textnormal{for } j = R_{T,i} \textnormal{ if $i \in [T],$} \textnormal{ where } i = \textnormal{prog}^K\left(x_{[R_T]}\right) + 1, \,\\
      [\nabla f(x)]_j, & \textnormal{for all other $j \in [d]$},
    \end{cases}
  \end{align} for all $\forall x \in \R^d$ and take $\mathcal{D}_{\xi} = \textnormal{Bernouilli}(p_{\sigma}),$ where $p_{\sigma} \in (0, 1].$ 
  For all $x \in \R^d,$ $[x]_j$ is the $j$\textsuperscript{th} coordinate of $x.$ Similarly to \citep{arjevani2022lower}, we now show this oracle is unbiased and $\sigma^2$-variance-bounded. Clearly, for $j \neq R_{T,i},$ $\ExpSub{\xi}{[\nabla f(x, \xi)]_j}  = [\nabla f(x)]_j.$ Otherwise, 
  $$\ExpSub{\xi}{[\nabla f(x, \xi)]_j} = [\nabla f(x)]_j \frac{\Exp{\xi}}{p_{\sigma}} = [\nabla f(x)]_j,$$ and
  \begin{align*}
      \ExpSub{\xi}{\norm{\nabla f(x; \xi) - \nabla f(x)}^2} \leq \max_{j \in [d]} \left|[\nabla f(x)]_j\right|^2\Exp{\left(\frac{\xi}{p_{\sigma}} - 1\right)^2}
  \end{align*}
  because the difference is non-zero only in one coordinate. Thus
  \begin{align*}
      \ExpSub{\xi}{\norm{\nabla f(x, \xi) - \nabla f(x)}^2} &\leq \frac{\norm{\nabla f(x)}_{\infty}^2(1 - p_{\sigma})}{p_{\sigma}} = \frac{L^2 \lambda^2 \norm{\nabla F_{T,K,a}\left(\frac{x_{[R_T]}}{\lambda}\right)}_{\infty}^2(1 - p_{\sigma})}{\ell^2_1(K, a) p_{\sigma}} \\
      &\leq \frac{L^2 \lambda^2 \gamma_{\infty}^2(K, a) (1 - p_{\sigma})}{\ell^2_1(K, a) p_{\sigma}},
  \end{align*}
  where we apply Lemma~\ref{lemma:bound_grad}. Taking
  \begin{align}
  p_{\sigma} = \min\left\{\frac{L^2 \lambda^2 \gamma_{\infty}^2(K, a)}{\sigma^2 \ell_1^2(K, a)}, 1\right\} \overset{\eqref{eq:lambda}}{=} \min\left\{\frac{2 \varepsilon \gamma_{\infty}^2(K, a)}{\sigma^2}, 1\right\},
  \label{eq:JvLFTLPbL}
  \end{align}
  we get $\ExpSub{\xi}{\norm{\nabla f(x, \xi) - \nabla f(x)}^2} \leq \sigma^2.$

\textbf{(Step 3: Graph Analysis).}
  We now consider the defined graph $G = (V, E, b).$ The first step is to construct an undirected version $\bar{G}$ of the graph $G$, where we keep the same set of nodes and retain only one edge from $\{(i,j),(j,i)\}$ for each $(i,j)\in E$ with the same weight $b_{ij}$, making this edge unordered. Thus, $\bar{G}=(V,\bar{E},b)$, where $\{i,j\}\in\bar{E}$ with weight $b_{ij}>0$ if and only if $(i,j)\in E$ with weight $b_{ij}>0$. 
  
  The second substep is to construct a Gomory--Hu tree $T = (V, F, w)$ of $\bar{G}$, which in general is not unique but always exists. Notice that, by definition, we have $w_{ij} = \alpha_{\bar{G}}(i, j)$ for all $\{i, j\} \in F$. Recall the important Theorem~\ref{thm:gomory_hu}.
  By the max-flow min-cut theorem, the \emph{maximum $s$--$t$ flow} in $\bar{G}$ is upper bounded by $w_{uv}$. The \emph{maximum $s$--$t$ flow} in $G$ is equal to that in $\bar{G}$ and is therefore also upper bounded by $w_{uv}$, since $G$ is obtained from $\bar{G}$ by replacing each undirected edge $\{i,j\}$ of capacity $b_{ij}$ with two directed edges $(i,j)$ and $(j,i)$, each of capacity $b_{ij} = b_{ji}$. In terms of our problem, the maximum number of coordinates per second that nodes $s$ and $t$ can send to each other is bounded by $w_{uv}$. 

  Next, we take the values $\{w_{ij}\}_{ij \in F} \cup \{\infty\}$, sort them in ascending order, and define this list as $\bar{w} \eqdef (\bar{w}_1, \dots, \bar{w}_n).$ The idea is to take all bandwidth thresholds in $T$; we also include $\infty$ to capture the case when the nodes do not communicate.

  Consider the Gomory--Hu tree $T$ and Algorithm~\ref{alg:preprocess}. At the beginning, we define the triple $(T_1, (S_{1,1}), \bar{w}_1),$ where $T_1 = T$ and $S_{1,1} = [n]$ is the only connected component of $T_1$ (since $G$ is a connected graph; see Section~\ref{sec:intro}).
  Now, inductively, given a triple $(T_k, (S_{k,1}, \dots, S_{k,k}), \bar{w}_k),$ we remove an edge with value $\bar{w}_k$ in $T_k$ and define the corresponding graph as $T_{k+1}.$ Then, we define the triple $(T_{k+1}, (S_{k+1,1}, \dots, S_{k + 1,k+1}), \bar{w}_{k+1}),$ where $(S_{k+1,1}, \dots, S_{k + 1,k+1})$ are the connected components of $T_{k+1}.$ We repeat this procedure until $k = \abs{F} + 1 = n,$ when the corresponding triple is $(T_k, (S_{k,1}, \dots, S_{k,k}), \bar{w}_k)$ with $\bar{w}_k = \infty$ and $(S_{k,1}, \dots, S_{k,k}) = (\{1\}, \dots, \{n\}),$ which is a collection of singletons.

  The idea of this procedure is to take the initial Gomory-Hu tree, and one by one remove the edges, from the smallest to the largest value. This way, we can construct a sequence of connected components/partitions $(S_{k,1}, \dots, S_{k,k}),$ where the number of partitions increases by one since we remove only one edge.

  \textbf{(Step 4: Proposed lower bound).}
  Once the necessary constructions are defined, we now consider
  \begin{align}
    \label{eq:lfLRoptNasIqCWquDC}
    t^* \eqdef \min_{k \in [n]} \bar{t}(k).
  \end{align}
  where
  \begin{align*}
    \bar{t}(k) \eqdef \tktk,
  \end{align*}
  $c_1 \eqdef 2^{43} 3^7 5^{14} \pi^2 e^{10},$ $c_2 \eqdef 14,$
  \begin{align}
  \label{eq:moquqcDcUoBnatBybKQ}
  B_{h}(\nicefrac{\sigma^2}{\varepsilon},S) \eqdef \min_{m \in [\abs{S}]} \left[\left(\frac{1}{m} \sum_{i=1}^{m} \frac{1}{h_{\pi_{i}(S)}}\right)^{-1} \left(1 + \frac{\sigma^2}{\varepsilon m}\right)\right],
  \end{align}
  and $\pi(S)$ is a permutation that sorts $\{h_i\}_{i \in S}:$ $h_{\pi_1(S)} \leq \dots \leq h_{\pi_{\abs{S}}(S)}.$ We will show that \eqref{eq:lfLRoptNasIqCWquDC} is a valid lower bound.
  
  \textbf{(Corner case).} 
  If $\min_{p \in [k]}B_{h}(\nicefrac{\sigma^2}{\varepsilon},S_{k,p}) \geq \frac{d}{\bar{w}_{k}}$ for all $k \in [n],$ then $t^* = \bar{t}(k)$ with $k = 1$ because $B_{h}(\nicefrac{\sigma^2}{\varepsilon},S_1) \geq B_{h}(\nicefrac{\sigma^2}{\varepsilon},S_2)$ if $S_1 \subseteq S_2.$ In this case, 
  \begin{align}
    t^* 
    &= \frac{1}{c_1 \log^{c_2} (n + 1)} \times \min_{p \in [1]}B_{h}(\nicefrac{\sigma^2}{\varepsilon},S_{1,p}) \frac{L \Delta}{\varepsilon} = \frac{1}{c_1 \log^{c_2} (n + 1)} \times B_{h}(\nicefrac{\sigma^2}{\varepsilon},S_{1,1}) \frac{L \Delta}{\varepsilon} \nonumber \\
    &= \frac{1}{c_1 \log^{c_2} (n + 1)} \times \min_{m \in [n]} \left[\left(\frac{1}{m} \sum_{i=1}^{m} \frac{1}{h_{\pi_{i}}}\right)^{-1} \left(1 + \frac{\sigma^2}{\varepsilon m}\right)\right] \frac{L \Delta}{\varepsilon}, \label{eq:geXYzIhwJNB}
  \end{align}
  since $S_{1,1} = [n],$ where $\pi$ is a permutation that sorts $\{h_i\}_{i \in [n]}.$ In this corner case, \eqref{eq:geXYzIhwJNB} is a lower bound due to \citep{tyurin2023optimal}, where the authors consider the setting without communication times. If one starts taking communication times into account, the time complexity of algorithms can only increase.

  \textbf{(General cases).} Starting from this point, we only consider general cases when there exists the largest index $\bar{k}$ such that $\min_{p \in [\bar{k}]}B_{h}(\nicefrac{\sigma^2}{\varepsilon},S_{\bar{k},p}) < \frac{d}{\bar{w}_{\bar{k}}}.$ Notice that $\bar{k} < \abs{F} + 1,$ since $\bar{w}_{\abs{F} + 1} = \infty,$ and $\min_{p \in [\bar{k} + 1]}B_{h}(\nicefrac{\sigma^2}{\varepsilon},S_{\bar{k} + 1,p}) \geq \frac{d}{\bar{w}_{\bar{k} + 1}}.$ Moreover, the sequence $\left\{\frac{d}{\bar{w}_{\bar{k}}}\right\}_k$ is non-increasing and $\left\{\min_{p \in [k]}B_{h}(\nicefrac{\sigma^2}{\varepsilon},S_{k,p})\right\}_k$ is non-decreasing, where the latter follows from the definition of the partitions $\{(S_{k,1}, \dots, S_{k,k})\}_{k}$ and the fact that $B_{h}(\nicefrac{\sigma^2}{\varepsilon},S_1) \geq B_{h}(\nicefrac{\sigma^2}{\varepsilon},S_2)$ if $S_1 \subseteq S_2.$ In total, we can conclude that
  
  \begin{align}
    \label{eq:KFdeMzpCTPTvM}
    t^* = \min\left\{\bar{t}(\bar{k}), \bar{t}(\bar{k} + 1)\right\} = \frac{1}{c_1 \log^{c_2} (n + 1)} \times \min\left\{\frac{d}{\bar{w}_{\bar{k}}}, \min_{p \in [\bar{k} + 1]}B_{h}(\nicefrac{\sigma^2}{\varepsilon},S_{\bar{k} + 1,p})\right\} \frac{L \Delta}{\varepsilon}.
  \end{align}

  \textbf{(Step 5: Concentration analysis).}
  Recall the construction \eqref{eq:SypbeSSj}, where the function depends only on the coordinates indexed by $R_T.$ We now split $R_T$ into $B = \flr{T / K}$ blocks of size $K$ with a possible residue of size $T - B K.$ The coordinates $B_1 \eqdef [R_{T,1}, \dots, R_{T,K}]$ belong to block $1,$ $B_2 \eqdef [R_{T,K+1}, \dots, R_{T,2 K}]$ belong to block $2,$ and so forth. Let us define $y^k_i$ as the first time moment when worker $i$ can start discovering\footnote{In the paper, when we say that a worker $i$ \emph{discovers} a coordinate with index $j$, it means that it adds a vector to $I_i$ in which the corresponding value of that coordinate is non-zero.} a new coordinate in the $k + 1$\textsuperscript{th} block, and $y^{k} \eqdef \min_{i \in [n]} y^{k}_i$ as the time when any of them. Notice that $y^0_i = 0$ for all $i \in [n].$ 

  Consider the partition $(S_{\bar{k} + 1,1}, \dots, S_{\bar{k} + 1,\bar{k} + 1})$ of the workers. Recall that this partition is constructed from the tree $T$ by removing $\bar{k}$ edges such that all removed edges $\{u, v\}$ satisfy $w_{uv} \leq \bar{w}_{\bar{k}}.$ Notice that this partition decomposes the tree $T$ into subtrees $\bar{T}_1, \dots, \bar{T}_{\bar{k} + 1}$, which are connected to each other by the removed edges. Thus, we can consider a new meta tree $\bar{T}$ whose nodes are $\bar{V} \eqdef (S_{\bar{k} + 1,1}, \dots, S_{\bar{k} + 1,\bar{k} + 1})$ and whose edges are the removed edges. For all $S \in \bar{V},$ we define $\bar{E}(S)$ as the set of edges incident to $S.$

  In this meta tree $\bar{T},$ we run the following \emph{leaf-branch peeling} procedure. For all $p \geq 1$, let 
  $\bar{\mathcal{L}}_p$ 
  be the (non-empty) set of leaves of $\bar{T}_p \eqdef \bar{T} \setminus \bigcup_{i < p} \left(\bar{\mathcal{L}}_i \bigcup \bar{\mathcal{B}}_i\right)$ of size $\bar{n}_p > 0.$ Define $\bar{\mathcal{B}}_p = \emptyset,$ and recursively add to $\bar{\mathcal{B}}_p$ all nodes from $\bar{T}_p$ with $2$ edges that have a neighbor in $\bar{\mathcal{L}}_p$ or $\bar{\mathcal{B}}_p:$ first, add all nodes with $2$ edges that have a neighboring node in $\bar{\mathcal{L}}_p;$ then add all nodes with $2$ edges that have a neighboring node in $\bar{\mathcal{B}}_p,$ and repeat the last step until no candidate nodes remain. 
  
  Stop the procedure when $\bar{T}_{\bar{d} + 1}$ is empty, where $\bar{d}$ is the number of steps in the procedure. 
  It is possible to prove the following logarithmic bound on the number of steps $\bar{d}$:
  \begin{restatable}[Proof in Section~\ref{proof:leaf}]{lemma}{LEMMABOUNDD}
  \label{lem:leaf_branch_peeling_depth}
  The number of steps $\bar d$ in the leaf-branch peeling procedure satisfies
  \[
  \bar{d} \le \flr{\log_2(n + 2)}.
  \]
  \end{restatable}

  In the meta graph $\bar{T}$, consider any node $\bar{S}$ in $\bar{V}.$ 
  Also, consider any $i \in \bar{S}.$ There are \textbf{two ways} to discover one of the $K$ coordinates in block $k$: (1) either one of the workers (potentially worker $i$ itself, but not necessarily) from the group $\bar{S}$ discovers a coordinate by computing a stochastic gradient with $\xi = 1$, or (2) group $\bar{S}$ discovers a coordinate through one of the communication channels from one of other groups in $\bar{T}.$

  \textbf{First way to discover coordinates.} Before time $y^{k},$ for all $i \in [n],$ worker $i$ has discovered at most coordinates $[R_{T,1}, \dots, R_{T,(k - 1) K}] \cup [j_1, \dots, j_p],$ where $[j_1, \dots, j_p] \subset B_k \eqdef [R_{T,1 + (k-1) K}, \dots, R_{T,k K}]$ and $p \geq 0.$ Due to Lemma~\ref{lemma:prog}, worker $i$ can discover at most one coordinate in $B_k$ when it calculates a stochastic gradient, at position $\bar{j}$ such that $\bar{j}$ is the smallest index in $B_k \setminus (j_1, \dots, j_p).$ However, due to the construction \eqref{eq:stoch_constr}, worker $i$ can discover $\bar{j}$ only if it receives a ``lucky'' random Bernoulli variable with value $1.$ The time required to calculate one stochastic gradient is $h_i.$ Therefore, worker $i$ requires at least $h_i \eta_{k,i,1}$ seconds to discover a new coordinate in $B_k,$ where $\eta_{k,i,1} \sim \textnormal{Geom}(p_{\sigma}).$
  
  Similarly, group $\bar{S}$ can discover the first coordinate, that was not discovered by any other node from this group and by computing stochastic gradients, after at least $\min_{i \in \bar{S}} h_i \eta_{k,i,1}$ seconds, where $\{\eta_{k,i,1}\}_{k \geq 1, i \in [n]}$ are i.i.d random variables from $\textnormal{Geom}(p_{\sigma}),$ and the $m$\textsuperscript{th} coordinate after at least $\sum_{j=1}^{m}\min_{i \in \bar{S}} h_i \eta_{k,i,j}$ seconds, where $\{\eta_{k,i,j}\}_{k \geq 1, i \in [n],j \geq 1}$ are i.i.d random variables from $\textnormal{Geom}(p_{\sigma}).$ The $\min$ comes from the fact that the workers can calculate in parallel.

  \textbf{Second way to discover coordinates (see Figure~\ref{fig:meta_tree_partition}).} At the same time, workers from $\bar{V} \setminus \bar{S}$ can share a new coordinate with the workers from $\bar{S}$ via communication through the graph $G$. Consider any $\hat{S} \in \bar{V} \setminus \bar{S}.$ Recall that the workers in $\bar{S}$ are separated from the workers in $\hat{S}$ in the Gomory--Hu tree $T$ by the $\bar{S}$'edge of weight $w_{uv}$ such that $w_{uv} \leq \bar{w}_{\bar{k}}$, meaning that worker $i \in \bar{S}$ is separated from the workers in $\hat{S}$ by a cut of value less than $\bar{w}_{\bar{k}}$ (Theorem~\ref{thm:gomory_hu}). Thus, the maximal flow (number of coordinates per second) that worker $i \in \bar{S}$ can receive from $\hat{S}$ is less than or equal $\bar{w}_{\bar{k}}$. After at least time 
  $y^{k},$ 
  workers $\hat{S}$ can start sending a sequence of coordinates with indices $(\nu_1, \nu_2, \dots),$ where worker $i \in \bar{S}$ can discover a new coordinate. In order to discover a new coordinate, worker $i \in \bar{S}$ should receive some $\nu_j$ such that $\nu_j \in B_k \eqdef [R_{T,1 + (k-1) K}, \dots, R_{T,k K}].$ 
  
  Every edge $e$ of $\bar{S}$ in $\bar{T}$ separates $\bar{S}$ from the set of groups $\mathcal{S}$ on the other side of the edge. For all $k \geq 1,$ let $\{\nu_{k, \bar{S}, e,j}\}_{j \geq 1}$ be the sequence of coordinates sent by $\mathcal{S}$ to $\bar{S}$ after the moment when one of the workers from $\mathcal{S}$ can start discovering $k$\textsuperscript{th} block. We define $\mu_{k,\bar{S},e,1}$ as the number of received coordinates $\{\nu_{k,\bar{S}, e,j}\}_{j \geq 1}$ from $\mathcal{S}$ until the moment when a received coordinate belongs to block $k$ in any worker $i$ from $\bar{S}.$ Similarly, let $\mu_{k,\bar{S},e,p}$ be the number of received coordinates until the moment when a received coordinate belongs to block $k$, after the $(p - 1)$\textsuperscript{th} time this has happened in any worker $i$ from $\bar{S},$ and this coordinate does not equal to the previous $(p - 1)$ coordinates belonging to block $k.$ In total, starting from $y^k,$ group $\bar{S}$ can discover $m$ coordinates from $\mathcal{S}$ in block $k$ after at least 
  \begin{align*}
    \sum_{j=1}^{m} \frac{\mu_{k, \bar{S}, e, j}}{\bar{w}_{\bar{k}}}
  \end{align*}
  seconds since it takes $\frac{\ell}{\bar{w}_{\bar{k}}}$ seconds to send $\ell$ coordinates through the communication channel.

  \textbf{Properties of the random variables.} For all $k \geq 0,$ we define $\mathcal{G}_{k}$ as the sigma-algebra generated by $\{\eta_{k',i,j}\}_{1 \leq k' \leq k, i \in [n],j \geq 1},$ $\{\mu_{k',S,e,p}\}_{p \geq 1, 1 \leq k' \leq k, S \in \bar{V}, e \in \bar{E}(S)},$ and $\{\nu_{k,S,e,j}\}_{k \geq 1, j \geq 1, S \in \bar{V}, e \in \bar{E}(S)}.$ 
  
  In Section~\ref{sec:aux}, we prove the following bounds on the probabilities.
  \begin{restatable}[Proof in Section~\ref{sec:aux}]{lemma}{LEMMAPROB}
  \label{lemma:prob}
  For all $i \in [n],$
  \begin{align*}
    \ProbCond{\eta_{k,i,j} \leq t}{\{\eta_{k,i,j'}\}_{i \in [n], 1 \leq j' < j}, \mathcal{G}_{k-1}} \leq \flr{t} p_{\sigma}
  \end{align*}
  for all $t \geq 0,$ $k \geq 1$ and $j \geq 1.$ Moreover, consider any $\bar{S} \in \bar{V}$ and $e \in \bar{E}(\bar{S}),$ then
  \begin{align*}
    \ProbCond{\mu_{k,\bar{S}, e,m} \leq t}{\mu_{k,\bar{S}, e,m - 1}, \dots, \mu_{k,\bar{S}, e,1}, \mathcal{G}_{k-1}} \leq \frac{K t}{\max\{\frac{d}{2} - \sum_{p = 1}^{m - 1}\mu_{k,\bar{S}, e,p}, 0\}}
  \end{align*}
  for all $t \geq 0,$ $k \geq 1,$ and $m \geq 1.$
  \end{restatable}

  \textbf{Auxiliary bounds.} 
  Before continuing with the proof of Theorem~\ref{thm:main}, we now consider two important lemmas, which follow from Lemma~\ref{lemma:prob}.
  \begin{restatable}{lemma}{LEMMAPROBTWO}
  \label{lemma:prob_two}
  For all $S \in \bar{V},$ $k \geq 1,$$K \geq 1,$ and $s \geq \frac{128}{B_{h}(\nicefrac{1}{p_{\sigma}},S)},$
  \begin{align}
    \label{eq:ZJwoGhKhTZNlf}
    \ExpCond{\exp\left(- s \sum_{j=1}^K \min_{i \in S} h_i \eta_{k,i,j}\right)}{\mathcal{G}_{k - 1}} \leq \frac{1}{8^K}.
  \end{align}
  \end{restatable}

  \begin{restatable}{lemma}{LEMMAPROBTHREE}
  \label{lemma:prob_three}
  For all $S \in \bar{V},$ $e \in \bar{E}(S),$ $k \geq 1,$$K \geq 1,$ $w > 0,$ and $s \geq \frac{1024 K w}{d},$
  \begin{align}
    \label{eq:CtQGxjjECmFTmdsN}
    \ExpCond{\exp\left(- s \sum_{j=1}^{K} \frac{\mu_{k, S, e, j}}{w}\right)}{\mathcal{G}_{k - 1}} \leq \frac{1}{8^K}.
  \end{align}
  \end{restatable}
  We now continue with the proof of Theorem~\ref{thm:main}.

  \textbf{(Step 6: Concentration bound).} 
  Recall that $\bar{\mathcal{L}}_1$ is the set of leaves in $\bar{V},$ where each leave has at most one neighbor. Moreover, recall that we have two  ways to discover $K$ coordinates in block $k.$ Since one of the two discovery methods must discover at least $\frac{K}{3}$ coordinates, we can conclude that
  \begin{align*}
    \min_{i \in S} y^{k+1}_i \geq \min\left\{y^{k} + \sum_{j=1}^{\frac{K}{3}} \min_{i \in S} h_i \eta_{k,i,j}, y^{k} + \sum_{j=1}^{\frac{K}{3}} \frac{\mu_{k, S, e, j}}{\bar{w}_{\bar{k}}}\right\}
  \end{align*}
  for all $S \in \bar{\mathcal{L}}_1,$ where $e$ is the only edge of $S$ that separates $S$ from all other groups (there may be no other groups, then the second term under the $\min$ does not appear). 
  Similarly, for all $p \in [\bar{d}]$ and $S \in \bar{\mathcal{B}}_p,$ using the same reasoning,
  \begin{align*}
    \min_{i \in S} y^{k+1}_i \geq \min\left\{y^{k} + \sum_{j=1}^{\frac{K}{3}} \min_{i \in S} h_i \eta_{k,i,j}, y^{k} + \sum_{j=1}^{\frac{K}{3}} \frac{\mu_{k, S, e_1, j}}{\bar{w}_{\bar{k}}}, y^{k} + \sum_{j=1}^{\frac{K}{3}} \frac{\mu_{k, S, e_2, j}}{\bar{w}_{\bar{k}}}\right\}
  \end{align*}
  for all $S \in \bar{\mathcal{B}}_p$ and $p \in [\bar{d}],$ where $e_1$ and $e_2$ are the only edges of $S.$ It is left to get a bound for $\bar{\mathcal{L}}_{p + 1}$ with $p \geq 1.$ Unlike $\bar{\mathcal{L}}_{1},$ $\bar{\mathcal{L}}_{p + 1}$ with $p \geq 1$ is not a set of leaves and each $S \in \bar{\mathcal{L}}_{p + 1}$ might have potentially a large number of edges. Consider any $S \in \bar{\mathcal{L}}_{p + 1}.$ It has at most one outgoing edge $e$ that connects $S$ with the nodes from $\bar{V} \setminus \bigcup_{j=1}^{p} \left(\bar{\mathcal{L}}_j \cup \bar{\mathcal{B}}_j\right),$ and all other edges connect it directly to $\bigcup_{j=1}^{p} \left(\bar{\mathcal{L}}_j \cup \bar{\mathcal{B}}_j\right)$ by the construction. Thus, for all $S \in \bar{\mathcal{L}}_{p+1},$
  \begin{align*}
    \min_{i \in S} y^{k+1}_i \geq \min\left\{y^{k} + \sum_{j=1}^{\frac{K}{3}} \min_{i \in S} h_i \eta_{k,i,j}, y^{k} + \sum_{j=1}^{\frac{K}{3}} \frac{\mu_{k, S, e, j}}{\bar{w}_{\bar{k}}}, \min_{S \in \cup_{j=1}^p \left(\bar{\mathcal{L}}_j \cup \bar{\mathcal{B}}_j\right)} \min_{i \in S} y^{k}_i\right\},
  \end{align*}
  where the last term in $\min$ comes from the fact that $\min_{S \in \cup_{j=1}^p \left(\bar{\mathcal{L}}_j \cup \bar{\mathcal{B}}_j\right)} \min_{i \in S} y^{k}_i$ is the earliest time when nodes from $\cup_{j=1}^p \left(\bar{\mathcal{L}}_j \cup \bar{\mathcal{B}}_j\right)$ can start sharing coordinates from block $k.$

  Since 
  \begin{align*}
    y^{k} = \min_{S \in \cup_{j=1}^{\bar{d}} \left(\bar{\mathcal{L}}_j \cup \bar{\mathcal{B}}_j\right)} \min_{i \in S} y^k_i
  \end{align*}
  for all $k \geq 0,$ it is sufficient to analyze the sequences 
  \begin{align}
    \label{eq:y_l_1}
    \bar{y}^{k+1}_S \eqdef \min\left\{\bar{y}^{k} + \sum_{j=1}^{\frac{K}{3}} \min_{i \in S} h_i \eta_{k,i,j}, \bar{y}^{k} + \sum_{j=1}^{\frac{K}{3}} \frac{\mu_{k, S, e, j}}{\bar{w}_{\bar{k}}}\right\}
  \end{align}
  for all $S \in \bar{\mathcal{L}}_1,$
  \begin{align}
    \label{eq:y_b_p}
    \bar{y}^{k+1}_S \eqdef \min\left\{\bar{y}^{k} + \sum_{j=1}^{\frac{K}{3}} \min_{i \in S} h_i \eta_{k,i,j}, \bar{y}^{k} + \sum_{j=1}^{\frac{K}{3}} \frac{\mu_{k, S, e_1, j}}{\bar{w}_{\bar{k}}}, \bar{y}^{k} + \sum_{j=1}^{\frac{K}{3}} \frac{\mu_{k, S, e_2, j}}{\bar{w}_{\bar{k}}}\right\}
  \end{align}
  for all $S \in \bar{\mathcal{B}}_p$ and $p \in [\bar{d}],$
  \begin{align}
    \label{eq:y_l_p}
    \bar{y}^{k+1}_S \eqdef \min\left\{\bar{y}^{k} + \sum_{j=1}^{\frac{K}{3}} \min_{i \in S} h_i \eta_{k,i,j}, \bar{y}^{k} + \sum_{j=1}^{\frac{K}{3}} \frac{\mu_{k, S, e, j}}{\bar{w}_{\bar{k}}}, \min_{S \in \cup_{j=1}^p \left(\bar{\mathcal{L}}_j \cup \bar{\mathcal{B}}_j\right)} \bar{y}^k_S\right\},
  \end{align}
  for all $p \in \{1, \dots, \bar{d} - 1\}$ and $S \in \bar{\mathcal{L}}_{p+1},$ and 
  \begin{align*}
    \bar{y}^{k+1} \eqdef \min_{S \in \cup_{j=1}^{\bar{d}} \left(\bar{\mathcal{L}}_j \cup \bar{\mathcal{B}}_j\right)} \bar{y}^{k+1}_S
  \end{align*}
  with $\bar{y}^k = 0$ and $\bar{y}^k_S = 0$ for $k = 0$ and $S \in \cup_{j=1}^{\bar{d}} \left(\bar{\mathcal{L}}_j \cup \bar{\mathcal{B}}_j\right).$
  Inductively, one can easily show that $y^k \geq \bar{y}^k$ for all $k \geq 0.$

  Using mathematical induction, taking 
  \begin{align}
    \label{eq:HYqffdQkytmukaW}
    s = 1024 \max\left\{\frac{1}{\min\limits_{S \in \bar{V}} B_{h}(\nicefrac{1}{p_{\sigma}},S)}, \frac{K \bar{w}_{\bar{k}}}{d}\right\},
  \end{align}
  we now prove that 
  \begin{align}
    \label{eq:kdULIbZc}
    \Exp{\exp\left(- s \bar{y}^{k} \right)} \leq e^{-(k + 1 - K)}
  \end{align}
  and 
  \begin{align}
    \label{eq:hCvVLgndsboCzRbvZrp}
    \Exp{\exp\left(- s \min_{S \in \bar{\mathcal{L}}_p \cup \bar{\mathcal{B}}_p} \bar{y}^{k}_{S}\right)} \leq \frac{3 \cdot 2^{p-1} e^{p-1} n^p p^p}{2^K} e^{-(k-K)}
  \end{align}
  for all $k \geq 0$ and $p \in [\bar{d}].$ 
  Notice that it is true for $k = 0.$ 
  Consider the inequalities
  \begin{equation}
    \label{eq:BgTYvCQv}
  \begin{aligned}
    \Exp{\exp\left(- s \min_{S \in \bar{\mathcal{L}}_1 \cup \bar{\mathcal{B}}_1} \bar{y}^{k+1}_{S}\right)} 
    &\leq \Exp{\exp\left(- s \min_{S \in \bar{\mathcal{L}}_1} \bar{y}^{k+1}_{S}\right)} + \Exp{\exp\left(- s \min_{S \in \bar{\mathcal{B}}_1} \bar{y}^{k+1}_{S}\right)} \\
    &\leq \sum_{S \in \bar{\mathcal{L}}_1}  \Exp{\exp\left(- s \bar{y}^{k+1}_{S}\right)} + \sum_{S \in \bar{\mathcal{B}}_1} \Exp{\exp\left(- s \bar{y}^{k+1}_{S}\right)}.
  \end{aligned}
  \end{equation}
  Using \eqref{eq:y_l_1},
  \begin{align*}
    &\Exp{\exp\left(- s \bar{y}^{k+1}_{S}\right)} \\
    &\leq \Exp{\exp\left(- s \min\left\{\bar{y}^{k} + \sum_{j=1}^{\frac{K}{3}} \min_{i \in S} h_i \eta_{k,i,j}, \bar{y}^{k} + \sum_{j=1}^{\frac{K}{3}} \frac{\mu_{k, S, e, j}}{\bar{w}_{\bar{k}}}\right\}\right)} \\
    &\leq \Exp{\exp\left(- s \left(\bar{y}^{k} + \sum_{j=1}^{\frac{K}{3}} \min_{i \in S} h_i \eta_{k,i,j}\right)\right)} + \Exp{\exp\left(- s \left(\bar{y}^{k} + \sum_{j=1}^{\frac{K}{3}} \frac{\mu_{k, S, e, j}}{\bar{w}_{\bar{k}}}\right)\right)} \\
    &= \Exp{\ExpCond{\exp\left(- s \left(\sum_{j=1}^{\frac{K}{3}} \min_{i \in S} h_i \eta_{k,i,j}\right)\right)}{\mathcal{G}_{k-1}} \exp\left(- s \bar{y}^{k}\right)} \\
    &\quad + \Exp{\ExpCond{\exp\left(- s \left(\sum_{j=1}^{\frac{K}{3}} \frac{\mu_{k, S, e, j}}{\bar{w}_{\bar{k}}}\right)\right)}{\mathcal{G}_{k-1}} \exp\left(- s \bar{y}^{k}\right)}
  \end{align*}
  for all $S \in \bar{\mathcal{L}}_1.$ Using \eqref{eq:ZJwoGhKhTZNlf}, \eqref{eq:CtQGxjjECmFTmdsN}, and \eqref{eq:kdULIbZc}, 
  \begin{align*}
    &\Exp{\exp\left(- s \bar{y}^{k+1}_{S}\right)} \leq \frac{2}{2^K} e^{-(k + 1 - K)} = \frac{2}{2^K} e^{-(k + 1 - K)}
  \end{align*}
  for all $S \in \bar{\mathcal{L}}_{1}$ and our choice of $s$ in \eqref{eq:HYqffdQkytmukaW}.
  For all $S \in \bar{\mathcal{B}}_1,$ using the same derivations, we can show that  
  \begin{align*}
    &\Exp{\exp\left(- s \bar{y}^{k+1}_{S}\right)} \\
    &\leq \Exp{\exp\left(- s \left(\bar{y}^{k} + \sum_{j=1}^{\frac{K}{3}} \min_{i \in S} h_i \eta_{k,i,j}\right)\right)} + \Exp{\exp\left(- s \left(\bar{y}^{k} + \sum_{j=1}^{\frac{K}{3}} \frac{\mu_{k, S, e_1, j}}{\bar{w}_{\bar{k}}}\right)\right)} \\
    &\quad + \Exp{\exp\left(- s \left(\bar{y}^{k} + \sum_{j=1}^{\frac{K}{3}} \frac{\mu_{k, S, e_2, j}}{\bar{w}_{\bar{k}}}\right)\right)} \leq \frac{3}{2^K} e^{-(k + 1 - K)} = \frac{3}{2^K} e^{-(k + 1 - K)}.
  \end{align*}
  Substituting to \eqref{eq:BgTYvCQv},
  \begin{align*}
    \Exp{\exp\left(- s \min_{S \in \bar{\mathcal{L}}_1 \cup \bar{\mathcal{B}}_1} \bar{y}^{k+1}_{S}\right)} \leq \frac{3 \left(\abs{\bar{\mathcal{L}}_1} + \abs{\bar{\mathcal{B}}_1}\right)}{2^K} e^{-(k + 1 - K)} \leq \frac{3 \cdot n}{2^K} e^{-(k + 1 - K)}.
  \end{align*}
  We have proved \eqref{eq:hCvVLgndsboCzRbvZrp} for $p = 1$ and $k \to k + 1.$
  For $p \geq 1,$ let us consider
  \begin{align*}
    \Exp{\exp\left(- s \min_{S \in \bar{\mathcal{L}}_{p+1} \cup \bar{\mathcal{B}}_{p+1}} \bar{y}^{k+1}_{S}\right)} \leq \sum_{S \in \bar{\mathcal{L}}_{p+1}}  \Exp{\exp\left(- s \bar{y}^{k+1}_{S}\right)} + \sum_{S \in \bar{\mathcal{B}}_{p+1}} \Exp{\exp\left(- s \bar{y}^{k+1}_{S}\right)}.
  \end{align*}
  Similarly, to the base case, $\bar{\mathcal{B}}_{p+1}$ is a set of nodes with 2 edges. Thus, using \eqref{eq:y_b_p}, \eqref{eq:ZJwoGhKhTZNlf}, \eqref{eq:CtQGxjjECmFTmdsN}, and \eqref{eq:kdULIbZc},
  \begin{align*}
    \Exp{\exp\left(- s \bar{y}^{k+1}_{S}\right)} \leq \frac{3}{2^K} e^{-(k + 1 - K)}.
  \end{align*}
  for all $S \in \bar{\mathcal{B}}_{p+1}.$ For all $S \in \bar{\mathcal{L}}_{p+1},$ using \eqref{eq:y_l_p},
  \begin{align*}
    &\Exp{\exp\left(- s \bar{y}^{k+1}_{S}\right)} \\
    &\leq \Exp{\exp\left(- s \min\left\{\bar{y}^{k} + \sum_{j=1}^{\frac{K}{3}} \min_{i \in S} h_i \eta_{k,i,j}, \bar{y}^{k} + \sum_{j=1}^{\frac{K}{3}} \frac{\mu_{k, S, e, j}}{\bar{w}_{\bar{k}}}, \min_{S \in \cup_{j=1}^p \left(\bar{\mathcal{L}}_j \cup \bar{\mathcal{B}}_j\right)} \bar{y}^{k}_{S}\right\}\right)} \\
    &\leq \Exp{\exp\left(- s \left(\bar{y}^{k} + \sum_{j=1}^{\frac{K}{3}} \min_{i \in S} h_i \eta_{k,i,j}\right)\right)} + \Exp{\exp\left(- s \left(\bar{y}^{k} + \sum_{j=1}^{\frac{K}{3}} \frac{\mu_{k, S, e, j}}{\bar{w}_{\bar{k}}}\right)\right)} \\
    &\quad+ \Exp{\exp\left(- s \min\left(\min_{S \in \cup_{j=1}^p \left(\bar{\mathcal{L}}_j \cup \bar{\mathcal{B}}_j\right)} \bar{y}^{k}_{S}\right)\right)}.
  \end{align*} 
  Using \eqref{eq:ZJwoGhKhTZNlf}, \eqref{eq:CtQGxjjECmFTmdsN}, and \eqref{eq:kdULIbZc},
  \begin{align*}
    \Exp{\exp\left(- s \bar{y}^{k+1}_{S}\right)} \leq \frac{2}{2^K} e^{-(k + 1 - K)} + \Exp{\exp\left(- s \min_{S \in \cup_{j=1}^p \left(\bar{\mathcal{L}}_j \cup \bar{\mathcal{B}}_j\right)}\bar{y}^{k}_{S}\right)},
  \end{align*}
  where we apply \eqref{eq:CtQGxjjECmFTmdsN} with $K \to \frac{K}{3}.$
  Notice that
  \begin{align}
    \label{eq:XGHSoeL}
    &\Exp{\exp\left(- s \min_{S \in \cup_{j=1}^{p} \left(\bar{\mathcal{L}}_j \cup \bar{\mathcal{B}}_j\right)} \bar{y}^{k}_{S}\right)} 
    \leq p \times \frac{3 \cdot 2^{p-1} e^{p-1} n^p p^p}{2^K} e^{-(k - K)} = \frac{3 \cdot 2^{p-1} e^{p-1} n^p p^{p+1}}{2^K} e^{-(k - K)}.
  \end{align}
  due to \eqref{eq:hCvVLgndsboCzRbvZrp}. 
  Using \eqref{eq:XGHSoeL},
  \begin{align*}
    \sum_{S \in \bar{\mathcal{L}}_{p+1}} \Exp{\exp\left(- s \bar{y}^{k+1}_{S}\right)} 
    &\leq \frac{2 \abs{\bar{\mathcal{L}}_{p+1}}}{2^K} e^{-(k + 1 - K)} + \frac{3 \cdot 2^{p-1} e^{p-1} n^p p^{p+1}}{2^K} e^{-(k-K)} \abs{\bar{\mathcal{L}}_{p+1}}  \\
    &\leq \frac{2 \abs{\bar{\mathcal{L}}_{p+1}}}{2^K} e^{-(k + 1 - K)} + \frac{3 \cdot 2^{p-1} e^{p-1} n^{p+1} p^{p+1}}{2^K} e^{-(k-K)}.
  \end{align*}
  since $\abs{\bar{\mathcal{L}}_{p+1}} \leq n.$
  Therefore,
  \begin{align}
    \Exp{\exp\left(- s \min_{S \in \bar{\mathcal{L}}_{p+1} \cup \bar{\mathcal{B}}_{p+1}} \bar{y}^{k+1}_{S}\right)} 
    &\leq \frac{3 \cdot 2^{p-1} e^{p-1} n^{p+1} p^{p+1}}{2^K} e^{-(k-K)} + \frac{2 \abs{\bar{\mathcal{L}}_{p+1}}}{2^K} e^{-(k + 1 - K)} + \frac{3 \abs{\bar{\mathcal{B}}_{p+1}}}{2^K} e^{-(k + 1 - K)} \nonumber \\
    &\leq \frac{3 \cdot 2^{p-1} e^{p} n^{p+1} p^{p+1}}{2^K} e^{-(k + 1 - K)} + \frac{3 n}{2^K} e^{-(k + 1 - K)} \nonumber \\
    &\leq \frac{3 \cdot 2^{p} e^p n^{p+1} (p+1)^{p+1}}{2^K} e^{-(k + 1 - K)}. \label{eq:WAwILgMO}
  \end{align}
  for all $p \geq 1.$ 
  We have proved the next step of \eqref{eq:hCvVLgndsboCzRbvZrp} for $k \to k + 1.$ It left to prove \eqref{eq:kdULIbZc} for $k \rightarrow k + 1.$ Since $\bar{V} = \cup_{j=1}^{\bar{d}} \left(\bar{\mathcal{L}}_j \cup \bar{\mathcal{B}}_j\right),$
  \begin{align}
    \label{eq:XASaQOmshCMPBRbat}
    \Exp{\exp\left(- s \bar{y}^{k+1} \right)} \overset{\eqref{eq:WAwILgMO}}{\leq} \bar{d} \times \frac{3 \cdot 2^{\bar{d}-1} e^{\bar{d}-1} n^{\bar{d}} \bar{d}^{\bar{d}}}{2^K} e^{-(k + 1 - K)} = \frac{3 \cdot 2^{\bar{d}-1} e^{\bar{d}-1} n^{\bar{d}} \bar{d}^{\bar{d} + 1}}{2^K} e^{-(k + 1 - K)}.
  \end{align}
  Notice that
  \begin{align*}
    \log_2\left(3 \cdot 2^{\bar{d}-1} e^{\bar{d}-1} n^{\bar{d}} \bar{d}^{\bar{d} + 1}\right) 
    &= \log_2(3) + (1 + \log_2e)(\bar{d}-1) + \bar{d} \log_2(n) + (\bar{d} + 1) \log_2(\bar{d})
  \end{align*}
  Due to Lemma~\ref{lem:leaf_branch_peeling_depth}, $\bar{d} \leq \log_2(n + 2).$ Thus,
  \begin{align*}
    \log_2\left(3 \cdot 2^{\bar{d}-1} n^{\bar{d}} \bar{d}^{\bar{d} + 1}\right) \leq 8 \log^2_2(n + 2).
  \end{align*}
  It is sufficient to take 
  \begin{align}
    \label{eq:CFwENhAkiqpnohakVd}
    K = 300 \left\lceil\log^2(n + 1)\right\rceil \geq 24 \log^2_2(n + 2)
  \end{align}
  in \eqref{eq:XASaQOmshCMPBRbat} to ensure that \eqref{eq:kdULIbZc} holds for $k \rightarrow k + 1:$
  \begin{align*}
    \Exp{\exp\left(- s \bar{y}^{k + 1} \right)} \leq e^{-(k + 2 - K)}.
  \end{align*}

  \textbf{(Step 7: Endgame).}
  Using Chernoff's method, \eqref{eq:kdULIbZc}, and $y^B \geq \bar{y}^B,$
  \begin{align*}
    \Prob{y^B \leq \bar{t}} 
    &\leq \Prob{\bar{y}^B \leq \bar{t}} \leq \Prob{\exp\left(-s \bar{y}^B\right) \geq \exp\left(-s \bar{t}\right)} \leq \exp\left(s \bar{t}\right) \Exp{\exp\left(- s \bar{y}^B\right)} \\
    &\leq \exp\left(s \bar{t} - (B + 1 - K)\right) = \delta
  \end{align*}
  for all $\delta \in (0, 1]$ and $k \geq 0,$ and fixing
  \begin{align}
    \label{eq:FcmHkQMCPYv}
    \bar{t} = \frac{1}{s} \left(\log(\delta) + (B + 1 - K)\right).
  \end{align}
  Using \eqref{eq:ElKnuvDAUgPNjjaLAn}, for all $t \geq 0,$
  \begin{align*}
  \inf_{x \in G_t} \norm{\nabla f(x)}^2 > 2 \varepsilon \inf_{x \in G_t} \mathbbm{1}\left[\textnormal{prog}^K(x_{[R_T]}) < T\right],
  \end{align*}
  where $G_t$ is the set of points computed by the algorithm $A$ up to time $t.$ By construction, if $y^B > t,$ then $\inf_{x \in G_t} \mathbbm{1}\left[\textnormal{prog}^K(x_{[R_T]}) < T\right] = 1;$ thus,
  \begin{align*}
    \Exp{\inf_{y \in G_t} \norm{\nabla f(y)}^2} > 2 \varepsilon \times \Prob{y^B > t} \geq 2 \varepsilon (1 - \delta).
  \end{align*}
  Choosing $\delta = \frac{1}{2}$ and $t = \bar{t},$
  \begin{align*}
    \Exp{\inf_{y \in G_t} \norm{\nabla f(y)}^2} > \varepsilon.
  \end{align*}
  By construction, notice that $f$ is random. Nevertheless,
  \begin{align*}
    \Exp{\ExpCond{\inf_{y \in G_t} \norm{\nabla f(y)}^2}{f}} > \varepsilon.
  \end{align*}
  Thus, there exists a deterministic $\bar{f}$ such that 
  \begin{align*}
    \Exp{\inf_{y \in G_t(\bar{f})} \norm{\nabla \bar{f}(y)}^2} = \ExpCond{\inf_{y \in G_t} \norm{\nabla f(y)}^2}{f = \bar{f}} > \varepsilon,
  \end{align*}
  where $G_t(\bar{f})$ are outputs of the algorithm given $\bar{f}$ (in the statement of the theorem, we rename $\bar{f}$ to $f$).

  It is left to find the asymptotic of $t = \bar{t}$ using \eqref{eq:FcmHkQMCPYv}:
  \begin{align}
    \label{eq:PDZnKBCYG}
    \bar{t} 
    &\geq \frac{1}{s} \left(B - K\right) \overset{\eqref{eq:HYqffdQkytmukaW}}{=} \frac{1}{2^{10}} \min\left\{\min\limits_{S \in \bar{V}} B_{h}(\nicefrac{1}{p_{\sigma}},S), \frac{d}{K \bar{w}_{\bar{k}}}\right\} \left(\flr{\frac{T}{K}} - K\right).
  \end{align}
  Consider \eqref{eq:WzHRTJlPowGgA}:
  \begin{align*}
    T = \left\lfloor\frac{L \Delta}{2 \Delta^0(K, a) \cdot \ell_1(K, a) \cdot \varepsilon}\right\rfloor.
  \end{align*}
  Using the definitions of $\Delta^0(K, a)$ and $\ell_1(K, a),$
  \begin{align*}
    T = \left\lfloor\frac{L \Delta \log a}{48 \pi e^{3} K^2 a^{2 K} \varepsilon}\right\rfloor.
  \end{align*}
  Choosing $a = 1 + \frac{1}{K},$
  \begin{align*}
    T \geq \left\lfloor\frac{L \Delta}{96 \pi e^{5} K^3 \varepsilon}\right\rfloor
  \end{align*}
  since $\log\left(1 + \frac{1}{K}\right) \geq \frac{1}{2 K}$ and $\left(1 + \frac{1}{K}\right)^{2 K} \leq e^2$ for all $K \geq 1.$ Recall the definition of $K$ in \eqref{eq:CFwENhAkiqpnohakVd}.
  Since in the theorem, we assume that $\frac{L \Delta}{\varepsilon} \geq \bar{c}_1 \log^{10}(n + 1)$ for some universal constant $\bar{c}_1,$
  \begin{align*}
    T \geq \frac{L \Delta}{2 \cdot 96 \pi e^{5} K^3 \varepsilon},
  \end{align*}
  \begin{align*}
    \flr{\frac{T}{K}} \geq \frac{L \Delta}{4 \cdot 96 \pi e^{5} K^4 \varepsilon},
  \end{align*}
  and 
  \begin{align*}
    \flr{\frac{T}{K}} - K \geq \frac{L \Delta}{4 \cdot 96 \pi e^{5} K^4 \varepsilon} - K \geq \frac{L \Delta}{8 \cdot 96 \pi e^{5} K^4 \varepsilon}.
  \end{align*}
  Substituting to \eqref{eq:PDZnKBCYG},
  \begin{align}
    \label{eq:wxcmQVe}
    \bar{t} 
    &\geq \frac{1}{s} \left(B - K\right) \geq \frac{1}{2^{20} \pi e^{5} K^4} \times \min\left\{\min\limits_{S \in \bar{V}} B_{h}\left(\frac{1}{p_{\sigma}},S\right), \frac{d}{K \bar{w}_{\bar{k}}}\right\} \frac{L \Delta}{\varepsilon}.
  \end{align}
  Using the definition of $\gamma_{\infty}(K,a),$
  \begin{align*}
    &\frac{\sigma^2}{2 \varepsilon \gamma_{\infty}^2(K, a)} = \frac{\sigma^2 \log a}{144 \pi e^{3} K^2 a^{2K} \varepsilon} \geq \frac{\sigma^2}{288 \pi e^{5} K^3 \varepsilon}
  \end{align*}
  since $\log\left(1 + \frac{1}{K}\right) \geq \frac{1}{2 K}$ and $\left(1 + \frac{1}{K}\right)^{2 K} \leq e^2$ for all $K \geq 1.$ Due to \eqref{eq:moquqcDcUoBnatBybKQ}, 
  \begin{align*}
    B_{h}(\nicefrac{1}{p_{\sigma}},S) 
    &= B_{h}\left(\max\left\{\frac{\sigma^2}{2 \varepsilon \gamma_{\infty}^2(K, a)}, 1\right\},S\right) \\
    &= \min_{m \in [\abs{S}]} \left[\left(\frac{1}{m} \sum_{i=1}^{m} \frac{1}{h_{\pi_{i}(S)}}\right)^{-1} \left(1 + \max\left\{\frac{\sigma^2}{2 \varepsilon \gamma_{\infty}^2(K, a)}, 1\right\} \times \frac{1}{m}\right)\right] \\
    &\geq \frac{1}{288 \pi e^{5} K^3} \times \min_{m \in [\abs{S}]} \left[\left(\frac{1}{m} \sum_{i=1}^{m} \frac{1}{h_{\pi_{i}(S)}}\right)^{-1} \left(1 + \frac{\sigma^2}{ \varepsilon m}\right)\right] = \frac{1}{288 \pi e^{5} K^3} \times B_{h}\left(\nicefrac{\sigma^2}{\varepsilon},S\right).
  \end{align*}
  Substituting to \eqref{eq:wxcmQVe},
  \begin{align*}
    \bar{t} 
    &\geq \frac{1}{2^{20} \pi e^{5} K^4} \times \min\left\{\frac{1}{288 \pi e^{5} K^3} \times \min\limits_{S \in \bar{V}} B_{h}\left(\nicefrac{\sigma^2}{\varepsilon},S\right), \frac{d}{K \bar{w}_{\bar{k}}}\right\} \frac{L \Delta}{\varepsilon} \\
    &\geq \frac{1}{2^{29} \pi^2 e^{10} K^7} \times \min\left\{\min\limits_{S \in \bar{V}} B_{h}\left(\nicefrac{\sigma^2}{\varepsilon},S\right), \frac{d}{\bar{w}_{\bar{k}}}\right\} \frac{L \Delta}{\varepsilon} \\
    &\geq \frac{1}{2^{43} 3^7 5^{14} \pi^2 e^{10} \cdot \log^{14}(n + 1)} \times \min\left\{\min\limits_{S \in \bar{V}} B_{h}\left(\nicefrac{\sigma^2}{\varepsilon},S\right), \frac{d}{\bar{w}_{\bar{k}}}\right\} \frac{L \Delta}{\varepsilon},
  \end{align*}
  which matches \eqref{eq:KFdeMzpCTPTvM} because $\bar{V} \eqdef (S_{\bar{k} + 1,1}, \dots, S_{\bar{k} + 1,\bar{k} + 1}).$
  \end{proof}

  \subsection{Auxiliary lemmas}
  \label{sec:aux}

  \LEMMAPROB*

  \begin{proof}
  Since the oracle draws $\{\xi\}$ are i.i.d in \eqref{eq:stoch_constr} and the generated $\{\xi\}$ are independent of $\{\mu_{k',S,e,p}\}_{p \geq 1, 1 \leq k' \leq k, S \in \bar{V}, e \in \bar{E}(S)}$ and $\{\nu_{k,S,e,j}\}_{k \geq 1, j \geq 1, S \in \bar{V}, e \in \bar{E}(S)},$ we can conclude that 
  \begin{equation}
  \label{eq:MUNbEyiRMIiLxhGmNuH}
  \begin{aligned}
    &\ProbCond{\eta_{k,i,j} \leq t}{\{\eta_{k,i,j'}\}_{i \in [n], 1 \leq j' < j}, \mathcal{G}_{k-1}} \\
    &= \sum_{p=1}^{\flr{t}} \Prob{\xi_{k,i,j,p} = 1, \xi_{k,i,j,p - 1} = 0, \dots, \xi_{k,i,j,1} = 0} \leq \sum_{p=1}^{\flr{t}} \Prob{\xi_{k,i,j,p} = 1} \leq \flr{t} p_{\sigma}
  \end{aligned}
  \end{equation}
  for all $k \geq 1,$ $i \in [n],$ and $j \geq 1,$
  where $\{\xi_{k,i,j,p}\}_{k,i,j,p \geq 1}$ are i.i.d. Bernoulli random variables.
  
  Let us fix any $\bar{S} \in \bar{V}$ and $e \in \bar{E}(\bar{S}).$ Recall that $R_{T}$ is a uniformly random subset without repetitions; thus,
  \begin{align*}
    \ProbCond{\mu_{k,\bar{S}, e,1} = j}{\mathcal{G}_{k-1}} 
    &= \ProbCond{\nu_{k,\bar{S}, e, j} \in B_k, \nu_{k,\bar{S}, e, j-1} \not\in B_k, \dots, \nu_{k,\bar{S}, e, 1} \not\in B_k}{\mathcal{G}_{k-1}} \\
    &\leq \ProbCond{\nu_{k,\bar{S}, e, j} \in B_k}{\mathcal{G}_{k-1}} = \ExpCond{\mathbf{1}[\nu_{k,\bar{S}, e, j} \in B_k]}{\mathcal{G}_{k-1}}.
  \end{align*}
  Let us define $\mathcal{B}_{k-1}$ as the sigma-algebra generated by $\{\eta_{k',i,j}\}_{1 \leq k' \leq k - 1, i \in [n],j \geq 1},$ $\{B_{k'}\}_{1 \leq k' \leq k - 1},$ and $\{\nu_{k,\bar{S}, e, j}\}_{j \geq 1}.$ Notice that $\mathcal{G}_{k-1} \subseteq \mathcal{B}_{k-1}$ since $\{\mu_{k',S, e,p}\}_{p \geq 1, 1 \leq k' \leq k - 1, S \in \bar{V}, e \in \bar{E}(S)}$ are deterministic, knowing $\mathcal{B}_{k-1}.$ Thus,
  \begin{align*}
    \ProbCond{\mu_{k,\bar{S}, e,1} = j}{\mathcal{G}_{k-1}} 
    &= \ExpCond{\ExpCond{\mathbf{1}[\nu_{k,\bar{S}, e, j} \in B_k]}{\mathcal{B}_{k-1}}}{\mathcal{G}_{k-1}} \\
    &= \ExpCond{\ExpCond{\mathbf{1}[\nu_{k,\bar{S}, e, j} \in B_k]}{\{B_{k'}\}_{1 \leq k' \leq k - 1}, \{\nu_{k,\bar{S}, e, j}\}_{j \geq 1}}}{\mathcal{G}_{k-1}} \\
    &= \ExpCond{\ProbCond{\nu_{k,\bar{S}, e, j} \in B_k}{\{B_{k'}\}_{1 \leq k' \leq k - 1}, \{\nu_{k,\bar{S}, e, j}\}_{j \geq 1}}}{\mathcal{G}_{k-1}} \\
    &\leq \frac{K}{d - (k - 1) K}.
  \end{align*}
  The last inequality follows from the fact that $\{B_{k}\}_{k \in [B]}$ is independent of $\{\nu_{k,\bar{S}, e, j}\}$, and from evaluating the probability that a deterministic value belongs to a random subset $B_k$, given $\{B_{k'}\}_{1 \leq k' \leq k - 1}$. The fact that $\{B_{k}\}_{k \in [B]}$ is independent of $\{\nu_{k,\bar{S}, e, j}\}$ follows from Assumption~\ref{ass:compressors}.

  Since $T \leq \frac{d}{2},$ 
  \begin{align*}
    \ProbCond{\mu_{k,\bar{S}, e,1} = j}{\mathcal{G}_{k-1}} \leq \frac{K}{d - T} \leq \frac{K}{d / 2}.
  \end{align*}
  
  Let us define $u_{m} \eqdef \sum_{p = 1}^{m}\mu_{k,\bar{S}, e,p}.$ Similarly,
  \begin{align*}
    &\ProbCond{\mu_{k,\bar{S}, e,m} = j}{\mu_{k,\bar{S}, e,m - 1}, \dots, \mu_{k,\bar{S}, e,1}, \mathcal{G}_{k-1}} \\
    &=\mathbb{P}\Big(\nu_{k,\bar{S}, e, u_{m-1} + j} \in B_k, \cap_{p=1}^{j-1} \{\nu_{k,\bar{S}, e, {u_{m-1}} + p} \not\in B_k\}, \cap_{p=1}^{m-1} \{\nu_{k,\bar{S}, e, {u_{m-1}} + j} \neq \nu_{k,\bar{S}, e, {u_{p}}}\} \Big\vert \\
    &\qquad\qquad \nu_{k,\bar{S}, e, {u_{m-1}}} \in B_k, \dots, \nu_{k,\bar{S}, e, {u_{m-2}} + 1} \not\in B_k, \nu_{k,\bar{S}, e, {u_{m-2}}} \in B_k, \dots, \nu_{k,\bar{S}, e, {1}} \not\in B_k, \\
    &\qquad\qquad \cap_{p \neq p'=1}^{m-1} \{\nu_{k,\bar{S}, e, {u_{p'}}} \neq \nu_{k,\bar{S}, e, {u_{p}}}\},  \{u_{p}\}_{p=1}^{m-1},  \mathcal{G}_{k-1}\Big) \\
    &\leq\mathbb{P}\Big(\nu_{k,\bar{S}, e, {u_{m-1}} + j} \in B_k \Big\vert \cap_{p=1}^{m-1} \{\nu_{k,\bar{S}, e, {u_{m-1}} + j} \neq \nu_{k,\bar{S}, e, {u_{p}}}\}, \\
    &\qquad\qquad \nu_{k,\bar{S}, e, {u_{m-1}}} \in B_k, \dots, \nu_{k,\bar{S}, e, {u_{m-2}} + 1} \not\in B_k, \nu_{k,\bar{S}, e, {u_{m-2}}} \in B_k, \dots, \nu_{k,\bar{S}, e, {1}} \not\in B_k, \\
    &\qquad\qquad \cap_{p \neq p'=1}^{m-1} \{\nu_{k,\bar{S}, e, {u_{p'}}} \neq \nu_{k,\bar{S}, e, {u_{p}}}\},  \{u_{p}\}_{p=1}^{m-1},  \mathcal{G}_{k-1}\Big) \\
    &\leq\mathbb{P}\Big(\nu_{k,\bar{S}, e, {u_{m-1}} + j} \in B_k \Big\vert \cap_{p=1}^{u_{m-1}} \{\nu_{k,\bar{S}, e, {u_{m-1}} + j} \neq \nu_{k,\bar{S}, e, {p}}\}, \\
    &\qquad\qquad \nu_{k,\bar{S}, e, {u_{m-1}}} \in B_k, \dots, \nu_{k,\bar{S}, e, {u_{m-2}} + 1} \not\in B_k, \nu_{k,\bar{S}, e, {u_{m-2}}} \in B_k, \dots, \nu_{k,\bar{S}, e, {1}} \not\in B_k, \\
    &\qquad\qquad \cap_{p \neq p'=1}^{m-1} \{\nu_{k,\bar{S}, e, {u_{p'}}} \neq \nu_{k,\bar{S}, e, {u_{p}}}\},  \{u_{p}\}_{p=1}^{m-1},  \mathcal{G}_{k-1}\Big)
  \end{align*}
  for all $j \geq 1,$ where use the standard properties of probability and the last inequality comes from the fact that for all $p \geq 1,$ if $\nu_{k,\bar{S}, e, p} \not\in B_k$ and $\nu_{k,\bar{S}, e, p} = \nu_{k,\bar{S}, e, {u_{m-1}} + j},$ \footnotetext{Pedantically, $\ProbCond{A}{B}$ is not defined if $\Prob{B} = 0.$ When we use the inequality $\Prob{A} \leq \ProbCond{A}{B},$ we take into account the standard convention that $\ProbCond{A}{B} = 0$ whenever $\Prob{B} = 0.$}then the probability is zero.
  Now, the last condition 
  \begin{align*}
    &\nu_{k,\bar{S}, e, {u_{m-1}}} \in B_k, \dots, \nu_{k,\bar{S}, e, {u_{m-2}} + 1} \not\in B_k, \nu_{k,\bar{S}, e, {u_{m-2}}} \in B_k, \dots, \nu_{k,\bar{S}, e, {1}} \not\in B_k, \\
    &\cap_{p=1}^{u_{m-1}} \{\nu_{k,\bar{S}, e, {u_{m-1}} + j} \neq \nu_{k,\bar{S}, e, {p}}\}, \cap_{p \neq p'=1}^{m-1} \{\nu_{k,\bar{S}, e, {u_{p'}}} \neq \nu_{k,\bar{S}, e, {u_{p}}}\},  \{u_{p}\}_{p=1}^{m-1}, \mathcal{G}_{k-1}
  \end{align*}
  says us that $B_k$ includes $\nu_{k,\bar{S}, e, {u_{m-1}}} \neq \dots \neq \nu_{k,\bar{S}, e, {u_1}}$ (pairwise distinct) and does not include $\nu_{k,\bar{S}, e, {u_{m-1}} - 1}, \dots, \nu_{k,\bar{S}, e, {u_{m-2}} + 1}, \nu_{k,\bar{S}, e, {u_{m-2}} - 1}, \dots, \nu_{k,\bar{S}, e, {1}}$ such that none of them equals to $\nu_{k,\bar{S}, e, {u_{m-1}} + j},$ and $\mathcal{G}_{k-1}$ also might ``reveal'' the values of $\{B_{k'}\}_{1 \leq k' \leq k - 1}.$ Conditioned on this information, the probability that $\nu_{k,\bar{S}, e, {u_{m-1}} + j} \in B_k$ 
  is less or equal to $\frac{K - (m - 1)}{\max\{d - (k - 1) K -  u_{m-1}, 0\}} \leq \frac{K}{\max\{d - T - u_{m-1}, 0\}} \leq \frac{K}{\max\{d / 2 - u_{m-1}, 0\}}$ since $B_k$ is an uniformly random subset (there are still $K - (m - 1)$ coordinates in $B_k$ with ``unknown positions'' that can uniformly placed in at least $\max\{d - (k - 1) K - u_{m-1}, 0\}$ positions). Therefore, 
  \begin{align*}
    \ProbCond{\mu_{k,\bar{S}, e,m} = j}{\mu_{k,\bar{S}, e,m - 1}, \dots, \mu_{k,\bar{S}, e,1}, \mathcal{G}_{k-1}} \leq \frac{K}{\max\{\frac{d}{2} - \sum_{p = 1}^{m - 1}\mu_{k,\bar{S}, e,p}, 0\}}
  \end{align*}
  and 
  \begin{align}
    \label{eq:sCPEmyFyFqBQAWOsh}
    \ProbCond{\mu_{k,\bar{S}, e,m} \leq t}{\mu_{k,\bar{S}, e,m - 1}, \dots, \mu_{k,\bar{S}, e,1}, \mathcal{G}_{k-1}} \leq \frac{K t}{\max\{\frac{d}{2} - \sum_{p = 1}^{m - 1}\mu_{k,\bar{S}, e,p}, 0\}}.
  \end{align}
  \end{proof}

  \LEMMAPROBTWO*

  \begin{proof}
    For all $t > 0$ and $j \geq 1,$
  \begin{align*}
    &\ExpCond{\exp\left(- s \min_{i \in S} h_i \eta_{k,i,j}\right)}{\{\eta_{k,i,j'}\}_{1 \leq j' \leq j - 1}, \mathcal{G}_{k - 1}} \\
    &\leq e^{- s t} + \ProbCond{\min_{i \in S} h_i \eta_{k,i,j} \leq t}{\{\eta_{k,i,j'}\}_{1 \leq j' \leq j - 1}, \mathcal{G}_{k - 1}} \\
    &\leq e^{- s t} + \sum_{i \in S} \ProbCond{h_i \eta_{k,i,j} \leq t}{\{\eta_{k,i,j'}\}_{1 \leq j' \leq j - 1}, \mathcal{G}_{k - 1}} \\
    &\overset{\eqref{eq:MUNbEyiRMIiLxhGmNuH}}{\leq} e^{- s t} + \sum_{i \in S} p_{\sigma} \flr{\frac{t}{h_i}}.
  \end{align*}
  Taking $t = \frac{1}{32} \times B_{h}(\nicefrac{1}{p_{\sigma}},S),$ we can use Lemma~\ref{lemma:flr} and the bound on $s$ to get
  \begin{align}
    \label{eq:NvgiUNKOYdZPeXP}
    \ExpCond{\exp\left(- s \min_{i \in S} h_i \eta_{k,i,j}\right)}{\{\eta_{k,i,j'}\}_{1 \leq j' \leq j - 1}, \mathcal{G}_{k - 1}} \leq \frac{1}{8}
  \end{align}
  for all $j \geq 1, k \geq 1, S \in \bar{V}.$ 
  Using the tower rule and \eqref{eq:NvgiUNKOYdZPeXP},
  \begin{align*}
    &\ExpCond{\exp\left(- s \sum_{j=1}^K \min_{i \in S} h_i \eta_{k,i,j}\right)}{\mathcal{G}_{k - 1}} \\
    &=\ExpCond{\ExpCond{\exp\left(- s \min_{i \in S} h_i \eta_{k,i,K}\right)}{\{\eta_{k,i,j'}\}_{1 \leq j' \leq K - 1}, \mathcal{G}_{k - 1}} \exp\left(- s \sum_{j=1}^{K-1} \min_{i \in S} h_i \eta_{k,i,j}\right)}{\mathcal{G}_{k - 1}} \\
    &\leq \frac{1}{8} \ExpCond{\exp\left(- s \sum_{j=1}^{K-1} \min_{i \in S} h_i \eta_{k,i,j}\right)}{\mathcal{G}_{k - 1}}.
  \end{align*}
  Then, repeating these derivations $K - 1$ times, we obtain \eqref{eq:ZJwoGhKhTZNlf}.
  \end{proof}

  \LEMMAPROBTHREE*

  \begin{proof}
    Using the tower rule,
    \begin{align*}
      &I^K_2 \eqdef \ExpCond{\exp\left(- s \sum_{j=1}^{K} \frac{\mu_{k, S, e, j}}{w}\right)}{\mathcal{G}_{k - 1}} \\
      &=\ExpCond{\underbrace{\ExpCond{\exp\left(- s \frac{\mu_{k, S, e, K}}{w}\right)}{\{\mu_{k, S, e, j}\}_{1 \leq j \leq K - 1}, \mathcal{G}_{k - 1}} \exp\left(- s \sum_{j=1}^{K - 1} \frac{\mu_{k, S, e, j}}{w}\right)}_{K_2 \eqdef }}{\mathcal{G}_{k - 1}}.
    \end{align*}
    If $\sum_{j=1}^{K - 1} \mu_{k, S, e, j} \geq \frac{d}{4},$ then
  \begin{align*}
    K_2 \leq \exp\left(- \frac{s d}{4 w}\right).
  \end{align*}
  Otherwise, if $\sum_{j=1}^{K - 1} \mu_{k, S, e, j} < \frac{d}{4},$ then
  \begin{align*}
    K_2 
    &\leq \ExpCond{\exp\left(- s \frac{\mu_{k, S, e, K}}{w}\right)}{\{\mu_{k, S, e, j}\}_{1 \leq j \leq K - 1}, \mathcal{G}_{k - 1}} \exp\left(- s \sum_{j=1}^{K - 1} \frac{\mu_{k, S, e, j}}{w}\right) \\
    &\leq \left(\exp\left(- s t\right) + \ProbCond{\frac{\mu_{k, S, e, K}}{w} \leq t}{\{\mu_{k, S, e, j}\}_{1 \leq j \leq K - 1}, \mathcal{G}_{k - 1}}\right) \exp\left(- s \sum_{j=1}^{K - 1} \frac{\mu_{k, S, e, j}}{w}\right)
  \end{align*}
  for all $t > 0.$ Using \eqref{eq:sCPEmyFyFqBQAWOsh}, if $\sum_{j=1}^{K - 1} \mu_{k, S, e, j} < \frac{d}{4},$ then
  \begin{align*}
    K_2 
    &\leq \left(\exp\left(- s t\right) + \frac{K t w}{\max\{\frac{d}{2} - \sum_{p = 1}^{K - 1}\mu_{k,S, e,p}, 0\}}\right) \exp\left(- s \sum_{j=1}^{K - 1} \frac{\mu_{k, S, e, j}}{w}\right) \\
    &\leq \left(\exp\left(- s t\right) + \frac{4 K t w}{d}\right) \exp\left(- s \sum_{j=1}^{K - 1} \frac{\mu_{k, S, e, j}}{w}\right).
  \end{align*}
  Taking $t = \frac{d}{128 K w}$ and using $s \geq \frac{\log(32)}{t},$ 
  \begin{align*}
    K_2 \leq \left(\frac{1}{32} + \frac{1}{32}\right) \exp\left(- s \sum_{j=1}^{K - 1} \frac{\mu_{k, S, e, j}}{w}\right) = \frac{1}{16} \exp\left(- s \sum_{j=1}^{K - 1} \frac{\mu_{k, S, e, j}}{w}\right).
  \end{align*}
  Combining both cases,
  \begin{align*}
    I_2^{K} 
    &\leq \ExpCond{\max\left\{\frac{1}{16} \exp\left(- s \sum_{j=1}^{K - 1} \frac{\mu_{k, S, e, j}}{w}\right), \exp\left(- \frac{s d}{4 w}\right)\right\}}{\mathcal{G}_{k - 1}} \\
    &\leq \frac{1}{16} \underbrace{\ExpCond{\exp\left(- s \sum_{j=1}^{K - 1} \frac{\mu_{k, S, e, j}}{w}\right)}{\mathcal{G}_{k - 1}}}_{I_2^{K-1}} + \exp\left(- \frac{s d}{4 w}\right).
  \end{align*}
  Repeating the same steps $K - 1$ more times,
  \begin{align*}
    I_2^{K} 
    &\leq \frac{1}{16^{K}} + 2 \exp\left(- \frac{s d}{4 w}\right) \leq \frac{1}{8^{K}},
  \end{align*}
  where we use that $s \geq \frac{256 K w}{d}$ and $\sum_{j=0}^{\infty} \frac{1}{16^j} \leq 2.$
  \end{proof}

  \begin{lemma}
    \label{lemma:flr}
    For all $h_1, \dots, h_n \geq 0,$ $p_{\sigma} \in (0, 1],$ and $S \subseteq [n],$ we have 
    \begin{align*}
      p_{\sigma} \sum_{i \in S} \flr{\frac{t}{h_i}} \leq \frac{1}{16}
    \end{align*}
    for
    \begin{align*}
      t = \frac{1}{32} \times B_{h}(\nicefrac{1}{p_{\sigma}},S),
    \end{align*}
    where 
    \begin{align*}
      B_{h}(\nicefrac{1}{p_{\sigma}},S) \eqdef \min_{m \in [\abs{S}]} \left[\left(\frac{1}{m} \sum_{i=1}^{m} \frac{1}{h_{\pi_{i}(S)}}\right)^{-1} \left(1 + \frac{1}{p_{\sigma} m}\right)\right],
    \end{align*}
    where $\pi(S)$ is a permutation that sorts $\{h_i\}_{i \in S}:$ $h_{\pi_1(S)} \leq \dots \leq h_{\pi_{\abs{S}}(S)}.$
  \end{lemma}
  \begin{proof}
  If $\flr{t/h_i}=0$ for all $i\in S$, then the lemma is true. Otherwise, let
  \begin{align}
    \label{eq:sussFrtXWJkdv}
    m \eqdef \max\Bigl\{r\in[\abs{S}]:\; h_{\pi_r(S)}\le t\Bigr\}.
  \end{align}
  Then $\flr{t/h_{\pi_i(S)}}\ge 1$ for $i\le m$ and $\flr{t/h_{\pi_i(S)}}=0$ for $i>m$. Hence,
  \[
  \sum_{i\in S}\flr{\frac{t}{h_i}}
  =\sum_{i=1}^m \flr{\frac{t}{h_{\pi_i(S)}}}
  \le \sum_{i=1}^m \frac{t}{h_{\pi_i(S)}}
  = t \sum_{i=1}^m \frac{1}{h_{\pi_i(S)}}.
  \]
  Let $A_m \eqdef \frac{1}{m}\sum_{i=1}^m \frac{1}{h_{\pi_i(S)}}$. By the definition of
  $B_h(\nicefrac{1}{p_\sigma},S)$ and $t=\frac{1}{32}B_h(\nicefrac{1}{p_\sigma},S)$,
  \[
  t \le \frac{1}{32}\,A_m^{-1} \left(1+\frac{1}{p_\sigma m}\right).
  \]
  Therefore,
  \begin{align*}
  p_\sigma \sum_{i\in S}\flr{\frac{t}{h_i}}
  &\le p_\sigma t \sum_{i=1}^m \frac{1}{h_{\pi_i(S)}}
  = p_\sigma t m A_m \le \frac{1}{32}\,p_\sigma \Bigl(A_m^{-1}\Bigl(1+\frac{1}{p_\sigma m}\Bigr)\Bigr) m A_m
  = \frac{1}{32}\bigl(p_\sigma m + 1\bigr).
  \end{align*}
  We now show that $p_\sigma m < 1$. If $p_\sigma m\ge 1$, then
  $1+\frac{1}{p_\sigma m}\le 2.$ Moreover, $A_m \ge \frac{1}{h_{\pi_m(S)}}$, hence
  $A_m^{-1}\le h_{\pi_m}$. Consequently,
  \[
  t \le \frac{1}{32} A_m^{-1} \times 2 \le \frac{1}{16}h_{\pi_m(S)}<h_{\pi_m(S)},
  \]
  contradicting $h_{\pi_m(S)}\le t$ in \eqref{eq:sussFrtXWJkdv}. Thus $p_\sigma m<1$, and
  \[
  p_\sigma \sum_{i\in S}\flr{\frac{t}{h_i}}
  \le \frac{1}{32}(1+p_\sigma m)<\frac{1}{16}.
  \]
  \end{proof}

\subsection{Proof of Lemma~\ref{lem:leaf_branch_peeling_depth}}
\label{proof:leaf}

\LEMMABOUNDD*

\begin{proof}
For any step $p\in[\bar d]$, consider
\[
\bar T_p \eqdef \bar T \setminus \bigcup_{i < p}\left(\bar{\mathcal{L}}_i \cup \bar{\mathcal{B}}_i\right),
\qquad\text{so that }\bar T_1\equiv \bar T.
\]
Let $\widetilde T_p$ be the tree obtained from $\bar T_p$ by suppressing all degree-$2$ nodes
(recursively removing any degree-$2$ node and merging its two incident edges into one).
Denote $\widetilde m_p \eqdef |V(\widetilde T_p)|$.

In $\widetilde T_p$, every node has degree $1$ or at least $3$. Let $\widetilde{\mathcal{L}}_p \equiv \bar{\mathcal{L}}_p$ be the set
of leaves of $\widetilde T_p$ and $\widetilde{\mathcal{I}}_p \eqdef V(\widetilde T_p)\setminus \widetilde{\mathcal{L}}_p$. Let $\deg_{\widetilde T_p}(v)$ be the number of edges of node $v$ in tree $\widetilde T_p.$ Using the degree-sum identity for trees,
\[
\sum_{v\in V(\widetilde T_p)} \deg_{\widetilde T_p}(v) = 2(\widetilde m_p-1)
\;\ge\; |\widetilde{\mathcal{L}}_p| + 3|\widetilde{\mathcal{I}}_p|,
\]
which implies $|\widetilde{\mathcal{L}}_p|\ge |\widetilde{\mathcal{I}}_p|+2$ due to $\widetilde m_p = |\widetilde{\mathcal{L}}_p| + |\widetilde{\mathcal{I}}_p|$ and hence
\[
|\widetilde{\mathcal{I}}_p|\le \frac{\widetilde m_p-2}{2}.
\]
By construction of the leaf-branch peeling step, removing
$\bar{\mathcal{L}}_p \cup \bar{\mathcal{B}}_p$ from $\bar T_p$ yields $\bar T_{p+1}$.
Note that $\widetilde{\mathcal{I}}_p$ is a set of nodes belonging to $\bar T_{p+1}$ and not suppressed by the degree-2 suppressing procedure in $\bar T_p$ (all nodes have degree $\geq 3$). Consider the set of degree-2 nodes $\widetilde{\mathcal{K}} = V(\bar T_{p+1}) \setminus \widetilde{\mathcal{I}}_p$ such that $\deg_{\bar T_p}(v) = 2$ for all $v \in \widetilde{\mathcal{K}}.$ For any $v \in \widetilde{\mathcal{K}},$ the degree of $v$ does not change when we construct $\bar T_{p+1}$, because if it changes, then $v$ would belong to $\bar{\mathcal{B}}_p$ (contradiction). Thus, $\deg_{\bar T_{p+1}}(v) = 2$ for all $v \in \widetilde{\mathcal{K}}.$ It means that when construct $\bar T_{p+1},$ vertices $\widetilde{\mathcal{K}}$ still have degree 2. Thus, the degree-2 suppressing procedure on $\bar T_{p+1}$ will remove at least $\widetilde{\mathcal{K}}$ and we get
\begin{align*}
  \widetilde m_{p+1} \leq |\widetilde{\mathcal{I}}_p|.
\end{align*}
Therefore,
\[
\widetilde m_{p+1} \le |\widetilde{\mathcal{I}}_p|\le \frac{\widetilde m_p-2}{2}.
\]
Unrolling the recursion, 
\begin{align*}
  \widetilde m_{p}+2 \le \frac{\widetilde m_1+2}{2^{p-1}} \le \frac{n + 2}{2^{p-1}}
\end{align*}
and 
\begin{align*}
  2^{\bar{d}} \le n + 2.
\end{align*}
since $m_{\bar{d}} \geq 0.$ Thus
\begin{align*}
  \bar{d} \le \flr{\log_2(n + 2)}.
\end{align*}
\end{proof}



\subsection{Proof of Theorem~\ref{thm:main_heter}}



\MAINTHEOREMHETER*

\begin{proof}
  The lower bound
  \begin{align}
    \Omega\left(\max\left\{\max_{i \in [n]} h_i, \frac{\sigma^2}{n \varepsilon}\left(\frac{1}{n} \sum_{i=1}^n h_i\right)\right\} \frac{L \Delta}{\varepsilon}\right)
  \end{align}
  follows from Theorem A.2 by \citet{tyurin2023optimal}, where the authors do not take into account communication times. With communication times, the lower bound can only increase. We now prove the first term in the $\max.$

  Unlike the homogeneous setup, it is sufficient to use the deterministic construction from \citep{carmon2020lower} and apply the standard trick of placing the blocks of this function on different nodes. For any $T \in \N,$ \citet{carmon2020lower} define $F_T \,:\,\R^T \to \R$ such that
\begin{align}
    \label{eq:worst_case}
    F_T(x) \eqdef -\Psi(1) \Phi(x_1) + \sum_{i=2}^T \left[\Psi(-x_{i-1})\Phi(-x_i) - \Psi(x_{i-1})\Phi(x_i)\right],
\end{align}
where
\begin{align*}
    \Psi(x) = \begin{cases}
        0, & x \leq 1/2, \\
        \exp\left(1 - \frac{1}{(2x - 1)^2}\right), & x \geq 1/2,
    \end{cases}
    \quad\textnormal{and}\quad
    \Phi(x) = \sqrt{e} \int_{-\infty}^{x}e^{-\frac{1}{2}t^2}dt.
\end{align*}
\begin{lemma}[\cite{carmon2020lower}]
    \label{lemma:worst_function}
    The function $F_T$ satisfies:
    \begin{enumerate}
        \item $F_T(0) - \inf_{x \in \R^T} F_T(x) \leq \Delta^0 T,$ where $\Delta^0 \eqdef 12.$
        \item The function $F_T$ is $l_1$--smooth, where $l_1 \eqdef 152.$
        \item For all $x \in \R^T,$ $\textnormal{prog}(\nabla F_T(x)) \leq \textnormal{prog}(x) + 1.$
        \item For all $x \in \R^T,$ if $\textnormal{prog}(x) < T,$ then $\norm{\nabla F_T(x)} > 1.$
    \end{enumerate}
\end{lemma}

  Find undirected version of $G:$ $\bar{G}=(V,\bar E,b)$ where $\{i,j\}\in\bar E$ with weight $b_{ij}$ iff $(i,j)\in E$ with weight $b_{ij}.$ For the graph $\bar{G}$, we construct a Gomory-Hu tree $T = (V, F, w).$ Let us take any pair $\{\bar{i},\bar{j}\} \in F$ of workers such that $\min_{\{i,j\} \in F} w_{ij} = w_{\bar{i}\bar{j}}.$ We now split the blocks of the function $F_T(x)$ from \eqref{eq:worst_case} and define two new functions. 
  
  First, sample a uniformly random subset $R_T = [R_{T,1}, \dots, R_{T,T}]$ without repetitions from set $[d]$ ($R_{T,1} < \dots < R_{T,T}$). Let us fix any $\lambda > 0,$ then we take $F_{T,1} \,:\,\R^d \to \R$ and $F_{T,2} \,:\,\R^d \to \R$ such that
  \begin{align}
    \label{eq:CXOAktJs}
    F_{T,1}(x) \eqdef -\Psi(1) \Phi([x]_{R_{T,1}}) + \sum_{i \in \{2, \dots, T\}, i\,|\,2 = 1} \left[\Psi(-[x]_{R_{T,i-1}})\Phi(-[x]_{R_{T,i}}) - \Psi([x]_{R_{T,i-1}})\Phi([x]_{R_{T,i}})\right],
  \end{align}
  and 
  \begin{align*}
    F_{T,2}(x) \eqdef \sum_{i \in \{2, \dots, T\}, i\,|\,2 = 0} \left[\Psi(-[x]_{R_{T,i-1}})\Phi(-[x]_{R_{T,i}}) - \Psi([x]_{R_{T,i-1}})\Phi([x]_{R_{T,i}})\right],
  \end{align*}
  where $[x]_j$ is the $j$\textsuperscript{th} coordinate of $x.$ The idea is that we apply a random permutation of the coordinates. Notice that the $d - T$ coordinates are artificial. We consider the following functions $f_i:$
    \begin{align*}
        f_i(x) \eqdef 
        \begin{cases} 
          \frac{n L \lambda^2}{l_1} F_{T,1}\left(\frac{x}{\lambda}\right), & i = \bar{i}, \\
          \frac{n L \lambda^2}{l_1} F_{T,2}\left(\frac{x}{\lambda}\right), & i = \bar{j}, \\
          0 , & i \neq \bar{i} \textnormal{ and } i \neq \bar{j}.
        \end{cases}
    \end{align*}
    Then, we get
    \begin{align*}
      f(x) = \frac{1}{n} \sum_{i=1}^n f_i(x) = \frac{1}{n} \left(\frac{n L \lambda^2}{l_1} F_{T,1} \left(\frac{x}{\lambda}\right) + \frac{n L \lambda^2}{l_1} F_{T,2}\left(\frac{x}{\lambda}\right) \right) = \frac{L \lambda^2}{l_1} F_{T}\left(\frac{x_{[R_T]}}{\lambda}\right),
    \end{align*}
    where $x_{[R_T]} \in \R^T$ is the subvector of size $T$ of vector $x \in \R^d$ such that $[x_{[R_T]}]_{i} = [x]_{R_{T,i}}$ for all $i \in [T].$
    The function $f$ is $L$-smooth since
    \begin{align*}
        \norm{\nabla f(x) - \nabla f(y)} &= \frac{L \lambda}{l_1} \norm{\nabla F_{T}\left(\frac{x_{[R_T]}}{\lambda}\right) - \nabla F_{T}\left(\frac{y_{[R_T]}}{\lambda}\right)} \leq L \norm{x - y}.
    \end{align*}
    Let us take $$T = \left\lfloor\frac{\Delta l_1}{L \lambda^2 \Delta^0}\right\rfloor,$$ then
    \begin{align*}
        f(0) - \inf_{x \in \R^d} f(x) = \frac{L \lambda^2}{l_1} (F_{T}\left(0\right) - \inf_{x \in \R^{T}} F_{T}(x)) \leq \frac{L \lambda^2 \Delta^0 T}{l_1} \leq \Delta.
    \end{align*}
    We take 
    $$\lambda = \frac{l_1 \sqrt{\varepsilon}}{L}$$
    to ensure that 
    \begin{align*}
        \norm{\nabla f(x)}^2 = \frac{L^2 \lambda^2}{l_1^2} \norm{\nabla F_{T}\left(\frac{x_{[R_T]}}{\lambda}\right)}^2 > \frac{L^2 \lambda^2}{l_1^2} = \varepsilon
    \end{align*}
    for all $x \in \R^T$ such that $\textnormal{prog}(x) < T.$ In the last inequality, we use Lemma~\ref{lemma:worst_function}.

    We assume that the workers have access to non-stochastic mappings $\nabla f_i(x)$ that are unbiased and $0$-variance-bounded, which is sufficient to derive the lower bound.

    Substituting the choice of the parameters,
    $$T = \left\lfloor\frac{\Delta L}{l_1 \varepsilon \Delta^0}\right\rfloor.$$

    Only workers $\bar{i}$ and $\bar{j}$ possess information about the function $f$. The function $f$ is constructed as a zero-chain function, and its components are distributed between workers $\bar{i}$ and $\bar{j}$. Because of this partitioning, these two workers must communicate in order to identify the next non-zero coordinate. Initially, only worker $\bar{i}$ can obtain a non-zero value in the first ``useful'' coordinate $R_{T,1}$ via the gradient of $F_{T,1}$. Subsequently, however, this worker cannot obtain a non-zero value in the second ``useful'' coordinate $R_{T,2}$ due to the construction in \eqref{eq:CXOAktJs}. Thus, it has to pass the first ``useful'' coordinate to worker $\bar{j}$ because only this worker can discover a non-zero value in the second ``useful'' coordinate. Then, once the second worker discovers the second ``useful'' coordinate and has to pass it to the first worker, and so forth.
    
    Let $\{\nu_{1,j}\}_{j \geq 1}$ be the sequence that the second worker receives from the first worker during the optimization process. Due to our construction, the function is random with randomly permuted coordinates. We define $\mu_{1}$ as the number of received coordinates $\{\nu_{1,j}\}_{j \geq 1}$ by the second worker from the first worker until the moment when a received coordinate is $R_{T,1}.$ Thus,
    \begin{align*}
      \Prob{\mu_{1} = j} 
      &= \Prob{\nu_{1, j} = R_{T,1}, \nu_{1, j-1} \neq R_{T,1}, \dots, \nu_{1, 1} \neq R_{T,1}} \leq \Prob{\nu_{1, j} = R_{T,1}} \\
      &= \Exp{\ProbCond{\nu_{1, j} = R_{T,1}}{\{\nu_{1, j}\}_{j \geq 1}}}.
    \end{align*}
    Due to Assumption~\ref{ass:compressors}, $\{\nu_{1, j}\}_{j \geq 1}$ are independent of $R_{T,1}.$ Thus,
    \begin{align*}
      \Prob{\mu_{1} = j} \leq \frac{1}{d}
    \end{align*}
    and 
    \begin{align*}
      \Prob{\mu_{1} \leq t} \leq \frac{t}{d}
    \end{align*}
    for all $t \geq 0,$ since $R_{T,1}$ is a uniformly random index from $[d]$, given $\{\nu_{1, j}\}_{j \geq 1}.$
    Let us define $y^k$ as the first moment of time when any of the workers can discover the $k$\textsuperscript{th} coordinate $R_{T,k}.$ 
    Similarly, we define $\mu_{2}$ as the number of received coordinates $\{\nu_{2,j}\}_{j \geq 1}$ by the first worker from the second worker, after time $y^2,$ until the moment when a received coordinate is $R_{T,2},$ we define $\mu_{3}$ as the number of received coordinates $\{\nu_{3,j}\}_{j \geq 1}$ by the second worker from the first worker, after time $y^3,$ until the moment when a received coordinate is $R_{T,3},$ where $\{\nu_{k,j}\}_{j \geq 1}$ are the coordinate sent after time $y^k.$ Thus,
    \begin{align*}
      &\ProbCond{\mu_{k} = j}{\{\mu_{k'}\}_{1 \leq k' < k}} \leq \ProbCond{\nu_{k, j} = R_{T,k}}{\{\mu_{k'}\}_{1 \leq k' < k}}.
    \end{align*}
    Let us define $\mathcal{B}_{k-1}$ as the sigma-algebra generated by $\{R_{T,k'}\}_{1 \leq k' < k}$ and $\{\nu_{k,j}\}_{j \geq 1, k \geq 1}.$ Since $\{\mu_{k'}\}_{1 \leq k' < k}$ are deterministic knowing $\mathcal{B}_{k-1}.$ Using the tower rule,
    \begin{align*}
      &\ProbCond{\mu_{k} = j}{\{\mu_{k'}\}_{1 \leq k' < k}} \leq \ProbCond{\nu_{k, j} = R_{T,k}}{\{\mu_{k'}\}_{1 \leq k' < k}} \\
      &= \ExpCond{\ProbCond{\nu_{k, j} = R_{T,k}}{\{R_{T,k'}\}_{1 \leq k' < k}, \{\nu_{k,j}\}_{j \geq 1, k \geq 1}}}{\{\mu_{k'}\}_{1 \leq k' < k}}.
    \end{align*}
    Conditioned on $\{R_{T,k'}\}_{1 \leq k' < k}$, there are at least $d - (k - 1)$ positions where $R_{T,k}$ can be placed uniformly. Therefore,
    \begin{align*}
      &\ProbCond{\mu_{k} = j}{\{\mu_{k'}\}_{1 \leq k' < k}} \leq \ProbCond{\nu_{k, j} = R_{T,k}}{\{\mu_{k'}\}_{1 \leq k' < k}} \leq \frac{1}{d - (k - 1)} \leq \frac{2}{d}
    \end{align*}
    and 
    \begin{align*}
      &\ProbCond{\mu_{k} \leq t}{\{\mu_{k'}\}_{1 \leq k' < k}} \leq \frac{2 t}{d}
    \end{align*}
    for all $t \geq 0$ and $1 \leq k \leq T - 1$ since $T \leq d / 2.$

    By our construction, the two workers are separated by the edge with weight $\min_{\{i,j\} \in F} w_{ij}.$ Hence, the maximum flow, or in terms of our problem, the maximum number of coordinates per second that nodes $\bar{i}$ and $\bar{j}$ can send to each other is bounded by $\min_{\{i,j\} \in F} w_{ij}$. 
    
    Combining all together, the time discover the last $T$\textsuperscript{th} ``useful'' is lower bounded by the sum
    \begin{align*}
      \sum_{j=1}^{T - 1} \frac{\mu_{j}}{\min\limits_{\{i,j\} \in F} w_{ij}}.
    \end{align*}
    for an algorithm to find a vector $x \in \R^T$ such that $\textnormal{prog}(x) = T.$ We define $w_{\min} \eqdef \min\limits_{\{i,j\} \in F} w_{ij}.$

    Hence, for any $\bar t>0$ and $s>0$, by Chernoff's method,
\begin{align*}
  \Prob{\sum_{j=1}^{T-1}\frac{\mu_j}{w_{\min}}\le \bar t}
  &=\Prob{\exp\!\left(-s\sum_{j=1}^{T-1}\frac{\mu_j}{w_{\min}}\right)\ge e^{-s\bar t}}\\
  &\le e^{s\bar t}\Exp{\exp\!\left(-s\sum_{j=1}^{T-1}\frac{\mu_j}{w_{\min}}\right)}.
\end{align*}
Using the tower rule and the bound
$\ProbCond{\mu_k\le t}{\{\mu_{k'}\}_{k'<k}}\le \frac{2t}{d}$, for any $t\ge 0$,
\begin{align*}
  \ExpCond{\exp\!\left(-s\frac{\mu_k}{w_{\min}}\right)}{\{\mu_{k'}\}_{k'<k}}
  &\le \exp\!\left(-s t\right)
    + \ProbCond{\mu_k\le w_{\min} t}{\{\mu_{k'}\}_{k'<k}}\\
  &\le \exp\!\left(-s t\right)+\frac{2 w_{\min} t}{d}.
\end{align*}
Using the tower rule,
\begin{align*}
  \Exp{\exp\!\left(-s\sum_{j=1}^{T-1}\frac{\mu_j}{w_{\min}}\right)}
  \le \left(\exp\!\left(-s t \right)+\frac{2 w_{\min} t}{d}\right)^{T-1},
\end{align*}
and
\begin{align*}
  \Prob{\sum_{j=1}^{T-1}\frac{\mu_j}{w_{\min}}\le \bar t}
  \le \exp(s\bar t)\left(\exp\!\left(-s t\right)+\frac{2 w_{\min} t}{d}\right)^{T-1}.
\end{align*}
Choosing $t= \frac{1}{s},$ gives
\begin{align*}
  \Prob{\sum_{j=1}^{T-1}\frac{\mu_j}{w_{\min}}\le \bar t}
  \le \exp(s\bar t)\left(e^{-1}+\frac{2 w_{\min} }{sd}\right)^{T-1}.
\end{align*}
Now we choose
\begin{align*}
  s = \frac{2 e w_{\min}}{d}
\end{align*}
to get
\begin{align*}
  \Prob{\sum_{j=1}^{T-1}\frac{\mu_j}{w_{\min}}\le \bar t}
  \le \exp\!\left(\frac{2 e w_{\min}}{d}\,\bar t - \frac{1}{4}(T-1)\right)
\end{align*}
For any $\delta\in(0,1]$, if
\begin{align*}
  \bar t \;\le\; \frac{d}{2 e w_{\min}}
  \left(\frac{1}{4} (T-1)-\log\!\left(\frac{1}{\delta}\right)\right),
\end{align*}
then $\Prob{\sum_{j=1}^{T-1}\frac{\mu_j}{w_{\min}}\le \bar t}\le \delta.$

For $\delta = \frac{1}{2},$
\begin{align*}
  &\frac{d}{2 e w_{\min}} \left(\frac{1}{4} (T-1)-\log\!\left(\frac{1}{\delta}\right)\right) \geq \frac{1}{2 e} \times \frac{d}{w_{\min}} \left(\frac{1}{4} \left\lfloor\frac{\Delta L}{l_1 \varepsilon \Delta^0}\right\rfloor - 2\right).
\end{align*}
In the theorem, we assume that $\frac{L \Delta}{\varepsilon} \geq \bar{c}_1$ for some universal constant $\bar{c}_1$ (we take $\bar{c}_1$ large enough to ensure that $\frac{1}{4} \left\lfloor\frac{\Delta L}{l_1 \varepsilon \Delta^0}\right\rfloor - 2 \geq \frac{\Delta L}{8 l_1 \varepsilon \Delta^0}$). Therefore, 
\begin{align*}
  &\frac{d}{2 e w_{\min}} \left(\frac{1}{4} (T-1)-\log\!\left(\frac{1}{\delta}\right)\right) \geq \frac{1}{16 e l_1 \Delta^0} \times \frac{d}{w_{\min}} \cdot \frac{\Delta L}{\varepsilon}.
\end{align*}
      
Thus, we get
\begin{align*}
  \Exp{\inf_{x \in G_t} \norm{\nabla f(x)}^2} > 2 \varepsilon \Prob{\sum_{j=1}^{T-1}\frac{\mu_j}{w_{\min}} > t} \geq \varepsilon
\end{align*}
for $$t = \bar{c}_1 \times \frac{d}{w_{\min}} \cdot \frac{\Delta L}{\varepsilon},$$
where $\bar{c}_1$ is a universal constant. 

It is left to recall that $f$ is random. Nevertheless, since 
\begin{align*}
  \Exp{\ExpCond{\inf_{x \in G_t} \norm{\nabla f(x)}^2}{f}} > \varepsilon,
\end{align*}
there exists a deterministic function $\bar{f}$ such that 
\begin{align*}
  \Exp{\inf_{x \in G_t(\bar{f})} \norm{\nabla \bar{f}(x)}^2} > \varepsilon,
\end{align*}
where $G_{t}(\bar{f})$ are outputs of the algorithm given $\bar{f}$ (in the statement of the theorem, we rename $\bar{f}$ to $f$).
\end{proof}

\section{One Step in Synchronous SGD}
\label{sec:sync}
In this section, we provide more details on how to implement Synchronous SGD and obtain \eqref{eq:mFiTJ}, as discussed in Section~\ref{sec:known_methods}. One way to do it is to fix any pivot worker (server) $v$ that aggregates the stochastic gradients. Then, consider any other worker $i.$ Choose any path $p$ between workers $i$ and $v$ in $G.$ Worker $i$ sends $\nabla f(x^k;\xi^k_i)$ along this path $p$, which takes at most $\nicefrac{d}{\min_{(i,j) \in p} b_{ij}} \leq \nicefrac{d}{\min_{(i,j) \in E} b_{ij}}$ seconds. There is an important caveat: all workers send their vectors to worker $v$ in parallel, potentially occupying the same edges and causing congestion (e.g., if two workers send their vectors through the same edge, it leads to a $2\times$ slowdown because the edge would process $2 \times d$ coordinates). One known way to fix this is to perform \emph{online in-network aggregation}: each worker streams its vector coordinate-by-coordinate, while intermediate nodes wait for the first coordinate from their children, aggregate these values, and immediately forward the result further. Thus, every edge carries only a single aggregated stream instead of multiple separate vectors, avoiding additional congestion factors, and at most $d$ coordinates pass through each edge. The broadcast operation can be performed in a similar way: every node, upon receiving the new vector $x^{k+1}$, immediately broadcasts its first coordinates further, before receiving the last ones.

\newpage
\section{Examples of Graphs}
In this section, we present examples of graphs that represent the communication topologies of the workers.

\subsection{Star Graph}
\label{sec:star_graph}
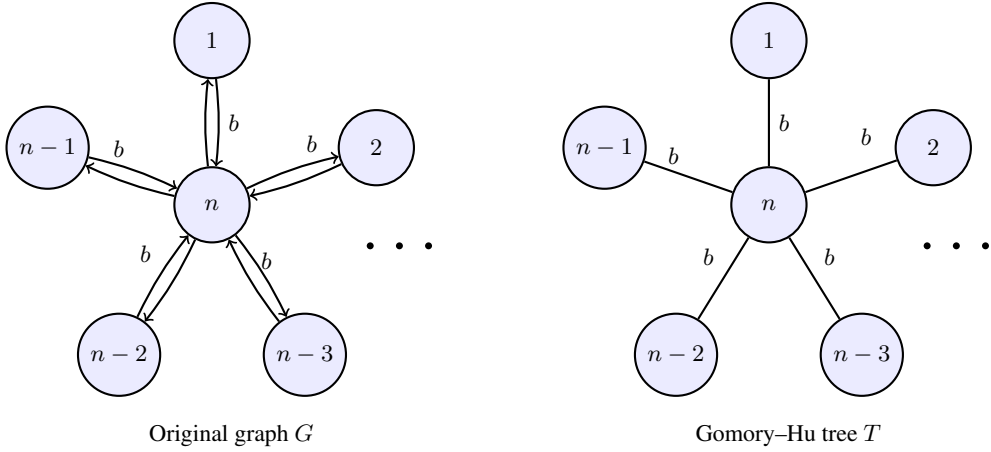
\begin{figure}[h]
\centering
\begin{minipage}{0.47\textwidth}
\centering
\begin{tikzpicture}[
    scale=0.95,
    every node/.style={font=\small},
    worker/.style={
        circle,
        draw=black,
        thick,
        minimum size=1.0cm,
        fill=blue!8
    },
    link/.style={
        ->,
        thick,
        bend left=6
    }
]

\node[worker] (n) at (0,0) {$n$};

\node[worker] (1) at (0,2.3) {$1$};
\node[worker] (2) at (2.3,0.8) {$2$};
\node at (2.7,-0.6) {\Huge $\cdots$};
\node[worker] (k) at (1.3,-2.1) {$n-3$};
\node[worker] (4) at (-1.3,-2.1) {$n-2$};
\node[worker] (s) at (-2.3,0.8) {$n-1$};

\draw[link] (1) to node[midway, right] {$b$} (n);
\draw[link] (n) to (1);

\draw[link] (2) to node[midway, above right, yshift=8pt] {$b$} (n);
\draw[link] (n) to (2);

\draw[link] (k) to node[midway, right, yshift=8pt] {$b$} (n);
\draw[link] (n) to (k);

\draw[link] (4) to node[midway, left, yshift=8pt] {$b$} (n);
\draw[link] (n) to (4);

\draw[link] (s) to node[midway, left, yshift=8pt] {$b$} (n);
\draw[link] (n) to (s);

\end{tikzpicture}

\vspace{2mm}
{\small Original graph \(G\)}
\end{minipage}
\hfill
\begin{minipage}{0.47\textwidth}
\centering
\begin{tikzpicture}[
    scale=0.95,
    every node/.style={font=\small},
    worker/.style={
        circle,
        draw=black,
        thick,
        minimum size=1.0cm,
        fill=blue!8
    },
    treelink/.style={
        draw,
        thick
    }
]

\node[worker] (n) at (0,0) {$n$};

\node[worker] (1) at (0,2.3) {$1$};
\node[worker] (2) at (2.3,0.8) {$2$};
\node at (2.7,-0.6) {\Huge $\cdots$};
\node[worker] (k) at (1.3,-2.1) {$n-3$};
\node[worker] (4) at (-1.3,-2.1) {$n-2$};
\node[worker] (s) at (-2.3,0.8) {$n-1$};

\draw[treelink] (1) -- node[midway, right] {$b$} (n);
\draw[treelink] (2) -- node[midway, above right, yshift=8pt] {$b$} (n);
\draw[treelink] (k) -- node[midway, right, yshift=8pt] {$b$} (n);
\draw[treelink] (4) -- node[midway, left, yshift=8pt] {$b$} (n);
\draw[treelink] (s) -- node[midway, left, yshift=8pt] {$b$} (n);

\end{tikzpicture}

\vspace{2mm}
{\small Gomory--Hu tree \(T\)}
\end{minipage}

\caption{
Star graph $G$ in the centralized setting and a Gomory--Hu tree of the corresponding undirected graph $\bar{G}$.
}
\label{fig:star_graph_gomory_hu}
\end{figure}

Consider Figure~\ref{fig:star_graph_gomory_hu} where workers communicate through a server (another worker). One can show that a Gomory-Hu tree $T$ of $G$ is shown in Figure~\ref{fig:star_graph_gomory_hu}. In the tree, sorting the weights, we get $\bar{w}_1 = b,$ \dots, $\bar{w}_{n-1} = b,$ $\bar{w}_n = \infty.$ For $k = 1$ in Algorithm~\ref{alg:preprocess}, $S_{1,1} = [n]$ and $\min_{S \in \mathcal{S}_1} t_1(S) = \left(\nicefrac{d}{b} + \nicefrac{h \sigma^2}{n \varepsilon}\right) + h.$ For $k = 2,$ $S_{2,1} = \{1\}$ and $S_{2,2} = \{2, \dots, n\}$ because we remove the edge corresponding to the first worker (w.l.o.g., we could have chosen any other worker), and $\min_{S \in \mathcal{S}_2} t_2(S) = \left(\nicefrac{d}{b} + \nicefrac{h \sigma^2}{{\color{red} (n-1)} \varepsilon}\right) + h.$ Notice that $\min_{S \in \mathcal{S}_2} t_2(S) \leq \min_{S \in \mathcal{S}_1} t_1(S).$ Repeating the same procedure and removing the edges, we get $\min_{S \in \mathcal{S}_1} t_1(S) \leq \min_{S \in \mathcal{S}_k} t_k(S)$ for all $1 \leq k \leq n - 1.$ However, for $k = n,$ we get $\bar{w}_n = \infty$ and $S_{n,1} = \{1\}, S_{n,2} = \{2\}, \dots, S_{n,n} = \{n\},$ and $\min_{S \in \mathcal{S}_n} t_n(S) = \nicefrac{h \sigma^2}{\varepsilon} + h,$ which can be smaller than $\min_{S \in \mathcal{S}_1} t_1(S).$ Thus, the total time complexity is
\begin{align*}
  \textstyle \cO\left(\min\left\{\left(\frac{d}{b} + \frac{h \sigma^2}{n \varepsilon}\right) \frac{L \Delta}{\varepsilon} + \frac{h  L \Delta}{\varepsilon}, \frac{h \sigma^2 L \Delta}{\varepsilon^2} + \frac{h  L \Delta}{\varepsilon}\right\}\right).
\end{align*}

\newpage
\subsection{$K$ clusters}

\begin{figure}[h]
\centering
\begin{tikzpicture}[
    scale=1,
    every node/.style={font=\small},
    cluster/.style={
        circle,
        draw=black,
        thick,
        minimum size=1.2cm,
        fill=blue!8
    },
    link/.style={
        ->,
        thick,
        bend left=6
    }
]

\node[cluster] (c1) at (0,3) {$\frac{n}{K}$ GPUs};
\node[cluster] (c2) at (2.8,1.4) {$\frac{n}{K}$ GPUs};
\node (cdots) at (3.7,0) {\Large $\cdots$};
\node[cluster] (c3) at (2.8,-1.4) {$\frac{n}{K}$ GPUs};
\node[cluster] (c4) at (0,-3) {$\frac{n}{K}$ GPUs};
\node[cluster] (c5) at (-2.8,-1.4) {$\frac{n}{K}$ GPUs};
\node[cluster] (c6) at (-2.8,1.4) {$\frac{n}{K}$ GPUs};


\draw[link] (c1) to node[midway,right, yshift=10pt] {$b_{\textnormal{slow}}$} (c2);
\draw[link] (c2) to (c1);

\draw[link] (c2) to node[midway,right] {$b_{\textnormal{slow}}$} (cdots);
\draw[link] (cdots) to (c2);

\draw[link] (cdots) to node[midway,right] {$b_{\textnormal{slow}}$} (c3);
\draw[link] (c3) to (cdots);

\draw[link] (c3) to node[midway,right, yshift=-10pt] {$b_{\textnormal{slow}}$} (c4);
\draw[link] (c4) to (c3);

\draw[link] (c4) to node[midway,left, yshift=-10pt] {$b_{\textnormal{slow}}$} (c5);
\draw[link] (c5) to (c4);

\draw[link] (c5) to node[midway,left] {$b_{\textnormal{slow}}$} (c6);
\draw[link] (c6) to (c5);

\draw[link] (c6) to node[midway,left, yshift=10pt] {$b_{\textnormal{slow}}$} (c1);
\draw[link] (c1) to (c6);

\end{tikzpicture}

\caption{
Example of $K$ clusters arranged in a ring. Each cluster contains $n/K$ workers with fast intra-cluster communication (all-to-all with bandwidth $\infty$). Communication between neighboring clusters occurs with bandwidth $b_{\textnormal{slow}}$.
}
\label{fig:k_clusters_ring}
\end{figure}
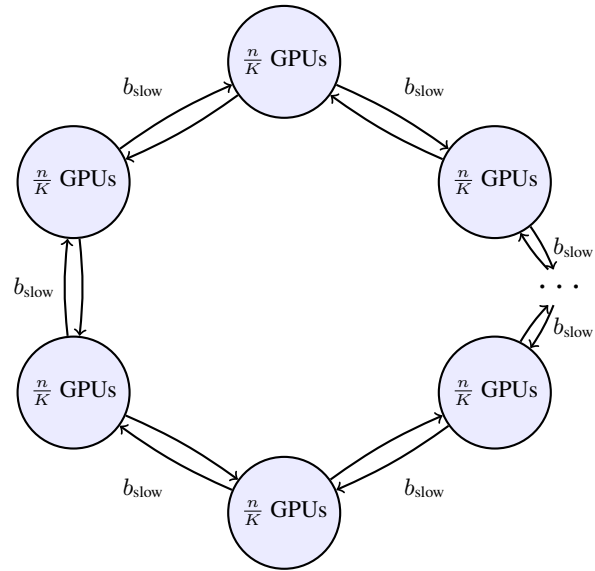

\newpage
\subsection{$p$-Torus}
Let $p \ge 1$ and $k \ge 2$. The $p$-Torus is the directed graph 
$G=(V,E,b)$ defined as follows.
\[
V = [k]^p = \{(v_1,\dots,v_p) : v_i \in [k]\}.
\] For every vertex $v=(v_1,\dots,v_p)$ and every coordinate $i\in[p]$,
there are directed edges
\[
v \to (v_1,\dots,v_i+1 \!\!\!\!\pmod{k},\dots,v_p)
\]
and
\[
v \to (v_1,\dots,v_i-1 \!\!\!\!\pmod{k},\dots,v_p).
\] 
Each edge $(u,v)\in E$ has weight $b_{uv}=b>0$ (for simplicity).

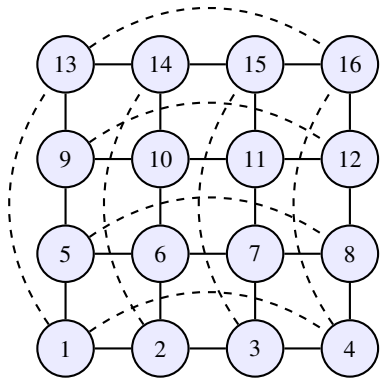
\begin{figure}[h]
\centering
\begin{tikzpicture}[
    scale=0.7,
    every node/.style={font=\small},
    worker/.style={
        circle,
        draw=black,
        thick,
        minimum size=0.75cm,
        fill=blue!8
    },
    link/.style={
        thick
    },
    wrap/.style={
        thick,
        dashed
    }
]

\foreach \x in {0,1,2,3} {
    \foreach \y in {0,1,2,3} {
        \pgfmathtruncatemacro{\lab}{\x+4*\y+1}
        \node[worker] (v\x\y) at (1.8*\x,1.8*\y) {\lab};
    }
}

\foreach \y in {0,1,2,3} {
    \draw[link] (v0\y) -- (v1\y);
    \draw[link] (v1\y) -- (v2\y);
    \draw[link] (v2\y) -- (v3\y);
}

\foreach \x in {0,1,2,3} {
    \draw[link] (v\x0) -- (v\x1);
    \draw[link] (v\x1) -- (v\x2);
    \draw[link] (v\x2) -- (v\x3);
}

\foreach \y in {0,1,2,3} {
    \draw[wrap, bend left=35] (v0\y) to (v3\y);
}
\foreach \x in {0,1,2,3} {
    \draw[wrap, bend left=35] (v\x0) to (v\x3);
}

\end{tikzpicture}
\caption{
Example of a $2$-Torus with $k = 4$. Each node communicates with its four neighbors, and the dashed edges indicate the wrap-around connections. All links have bandwidth $b$. In the visualization, two arcs merged and visualized with one undirected edge.
}
\label{fig:torus}
\end{figure}

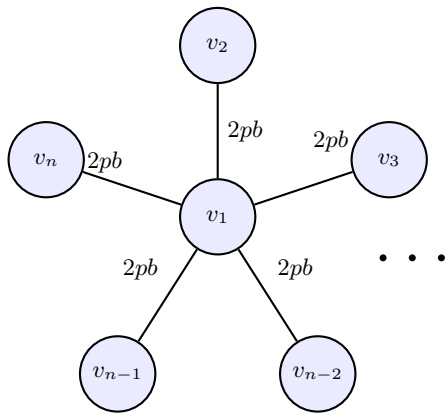
\begin{figure}[h]
\centering
\begin{tikzpicture}[
    scale=0.95,
    every node/.style={font=\small},
    worker/.style={
        circle,
        draw=black,
        thick,
        minimum size=1.0cm,
        fill=blue!8
    },
    treelink/.style={
        draw,
        thick
    }
]

\node[worker] (1) at (0,0) {$v_1$};

\node[worker] (2) at (0,2.4) {$v_2$};
\node[worker] (3) at (2.4,0.8) {$v_3$};
\node at (2.8,-0.6) {\Huge $\cdots$};
\node[worker] (k1) at (1.4,-2.2) {$v_{n-2}$};
\node[worker] (k2) at (-1.4,-2.2) {$v_{n-1}$};
\node[worker] (n) at (-2.4,0.8) {$v_n$};

\draw[treelink] (2) -- node[midway, right] {$2pb$} (1);
\draw[treelink] (3) -- node[midway, above right, yshift=10pt] {$2pb$} (1);
\draw[treelink] (k1) -- node[midway, right, yshift=10pt] {$2pb$} (1);
\draw[treelink] (k2) -- node[midway, left, yshift=10pt] {$2pb$} (1);
\draw[treelink] (n) -- node[midway, left, yshift=10pt] {$2pb$} (1);

\end{tikzpicture}

\caption{
A Gomory--Hu tree $T$ of the undirected graph $\bar{G}$ corresponding to a $p$-Torus with bandwidths $b$. The fact that $T$ is a tree with weights $2pb$ follows from the fact that every value of a min-cut in $\bar{G}$ is $2pb.$
}
\label{fig:torus_gomory_hu}
\end{figure}

\clearpage
\newpage
\section{Practical Guidelines and Numerical Experiments}
\label{sec:exp}
In this section, we discuss practical guidelines and conduct numerical experiments of the new algorithm, Grace SGD. As we explain in Section~\ref{sec:grace}, Grace SGD is simple. It is just a stochastic gradient method with preprocessing (Algorithm~\ref{alg:preprocess}) and the optimal-bandwidth AllReduce (Algorithm~\ref{alg:allreduce}). Implementing Algorithm~\ref{alg:preprocess} is straightforward. Finding a Gomory-Hu tree is a standard graph problem, and there are many open-source libraries with implementations (for instance, use \texttt{gomory\_hu\_tree} from \texttt{NetworkX} \citep{hagberg2007exploring}). Besides that, Algorithm~\ref{alg:preprocess} is a standard one loop function where we find all steps can be implemented in Python.

At the same time, in general, implementing the optimal-bandwidth AllReduce algorithm described in Algorithm~\ref{alg:allreduce} is arguably trickier. The first non-trivial step is to solve the \emph{Steiner Tree Packing} problem. While it is a classical problem in computer science, \citet{lau2004approximate} only recently proposed a polynomial-time algorithm that finds at least $\flr{\alpha_{\bar{G}}(S) / 26}$ edge-disjoint trees that contain every node in $S \subseteq V,$ where $\alpha_{\bar{G}}(S)$ is the minimum value of an $S$-cut in $\bar{G}$ (Definition~\ref{def:min_cut} and Theorem~\ref{thm:allreduce_time}). Once the Steiner Tree Packing is solved (up to a constant factor), the reduce and broadcast steps can be implemented in practice using the standard communication libraries.

The description of the polynomial-time algorithm in \citet{lau2004approximate} is somewhat non-trivial. Fortunately, for standard graph structures, implementing an optimal-bandwidth AllReduce is more straightforward. In particular, consider the $2$-Torus $G$ from Figure~\ref{fig:torus}. Without loss of generality, assume that $b = 1$ and $n = (2k + 1)^2$ for some $k \geq 1.$ We fix the node with coordinate $(k, k)$ and refer to it as the pivot node (see Figure~\ref{fig:graph_torus}). In this graph, every node has $4$ outgoing edges. The idea of this optimal-bandwidth AllReduce is that every node $i$ splits its local vector into four blocks $(a_{i,1}, a_{i,2}, a_{i,3}, a_{i,4})$ and sends them in different directions. The goal is now to choose the directions in such a way that each block eventually reaches the pivot worker without cycles, while intermediate nodes use online in-network aggregation to combine blocks with the same index and avoid congestion. 

Formally, for all $(i,j)$ with $i \neq k$ and $j \neq k$, the node sends the first block to $(i+1 \bmod (2k+1),\, j)$, the second block to $(i,\, j+1 \bmod (2k+1))$, the third block to $(i-1 \bmod (2k+1),\, j)$, and the fourth block to $(i,\, j-1 \bmod (2k+1))$. For nodes with $i = k$ or $j = k$, the node sends the first block to $(i,\, j+1 \bmod (2k+1))$, the second block to $(i+1 \bmod (2k+1),\, j)$, the third block to $(i,\, j-1 \bmod (2k+1))$, and the fourth block to $(i-1 \bmod (2k+1),\, j)$.

See visualization for $k = 2$ in Figure~\ref{fig:graph_torus}. For instance, node $(3,0)$ sends the first block to $(4,0)$, which aggregates its own block with it, and sends the sum to $(0,0),$ which also adds its first block. And it happens until the total sum arrives to $(2,2).$ Notice that $(2,1)$ aggregates the sum from nodes $(2,0)$ and $(1,1).$

Compared to the naive strategy, where every worker sends the full vectors to the pivot worker, using this routing algorithm and online in-network aggregation, we can speed up the reduce operation by $4\times$, since each block arrives at the central node via different and independent paths. The improvement by $4\times$ is expected since the value of a min-cut in the graph is $4.$ The broadcast can be implemented similarly by reversing the edges. We could have obtained a $2$--$4\times$ speedup using Algorithm~\ref{alg:allreduce}, but the strategy specialized for the 2-Torus is arguably simpler. For the ring graph, the algorithm is similar and one get a $2\times$ communication speed up. In general, for a $p$-Torus, the speedup should be $\Theta(p)$ (at least with Algorithm~\ref{alg:allreduce} due to Theorem~\ref{thm:allreduce_time}).
\begin{figure}[h]
\centering
\includegraphics[width=0.7\textwidth]{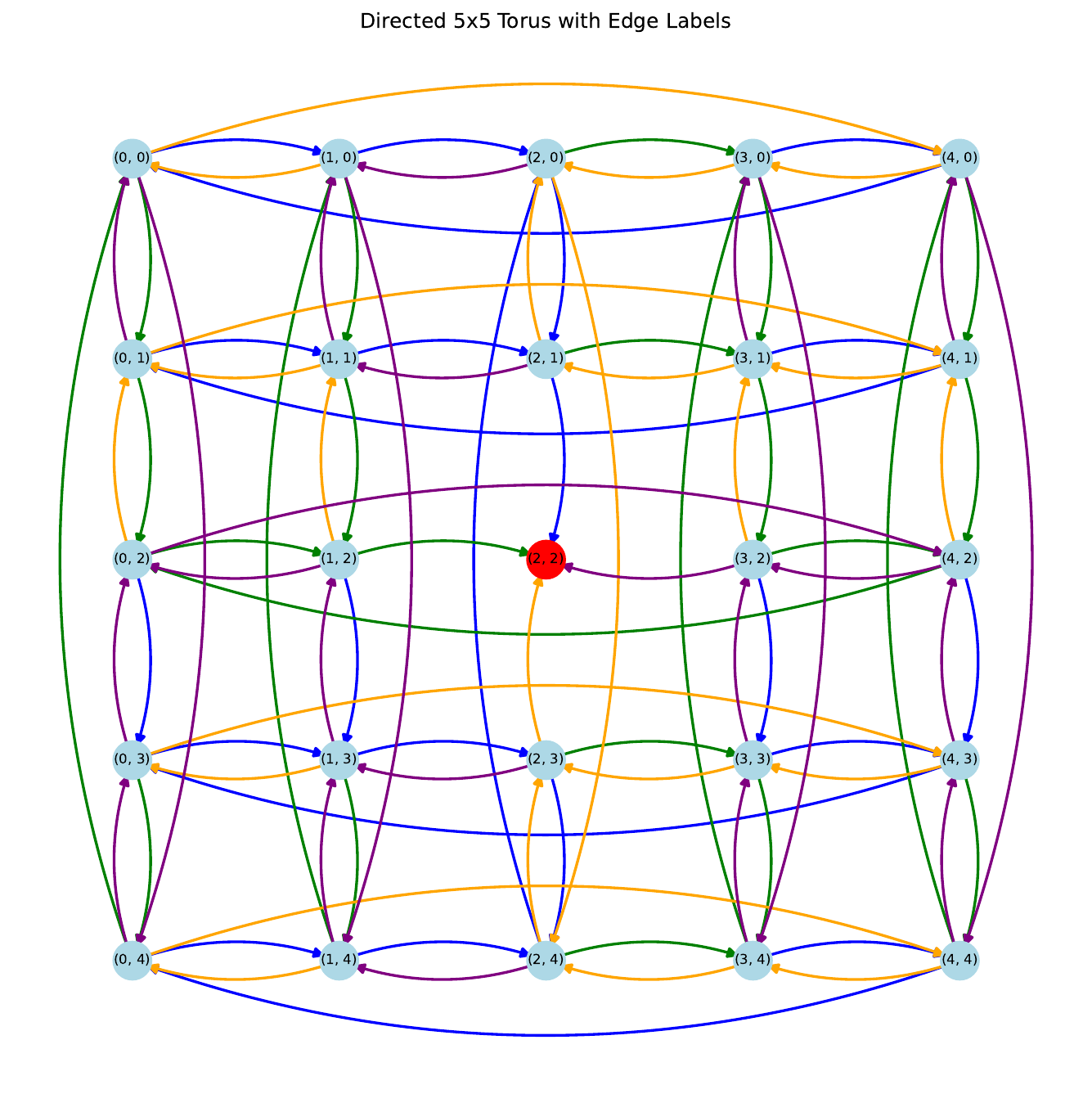}
\caption{2-Torus and Reduce routing. Each color corresponds to one of the four blocks.}
\label{fig:graph_torus}
\end{figure}
\subsection{Numerical experiments}
\label{sec:experiments}

We now consider Grace SGD and compare it to Synchronous SGD. The environment was emulated in Python 3 with one Intel(R) Xeon(R) Platinum 8168 CPU @ 2.70GHz. We assume that $i$\textsuperscript{th} worker requires $h_i = 1$ second to calculate one stochastic gradient. We assume that the graph $G$ is 2-Torus (see Figure~\ref{fig:torus}) with bandwidth $b = 0.1$ for all the edges. We consider the setup with $n = 100.$ In all methods, we tune the step sizes from the set $\{2^i\,|\,i \in [-20, 20]\}$. In \eqref{eq:main_problem}, we consider the standard logistic regression problem with \textit{MNIST} dataset \citep{lecun2010mnist}, where every worker samples one sample from the dataset and calculates a stochastic gradient.

\textbf{Naive synchronization vs. optimal-bandwidth AllReduce.} In this part, we compare Synchronous SGD with naive synchronization, where full vectors are sent to neighbors (a similar problem in gossip methods), and Grace SGD with optimal-bandwidth AllReduce. In Figure~\ref{fig:one}, we set $\nicefrac{\sigma^2}{\varepsilon} = 100$ for Algorithm~\ref{alg:main} to ensure that, in both Grace SGD and Synchronous SGD, all workers compute one stochastic gradient per iteration, providing a fair comparison of the communication strategies. In practice, we observe that Grace SGD converges faster.

\begin{figure}[h]
\centering
\includegraphics[width=0.49\textwidth]{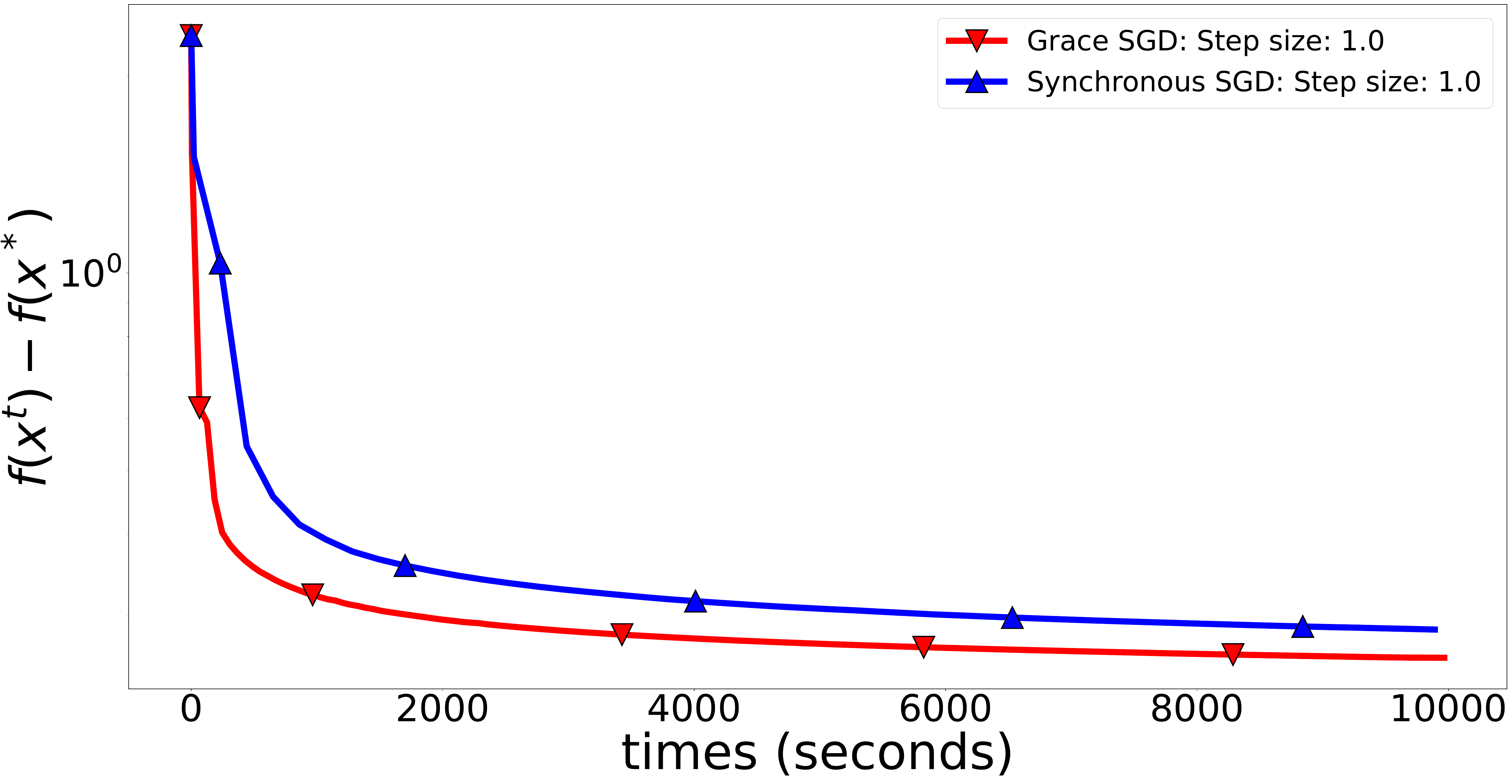}
\hfill
\includegraphics[width=0.49\textwidth]{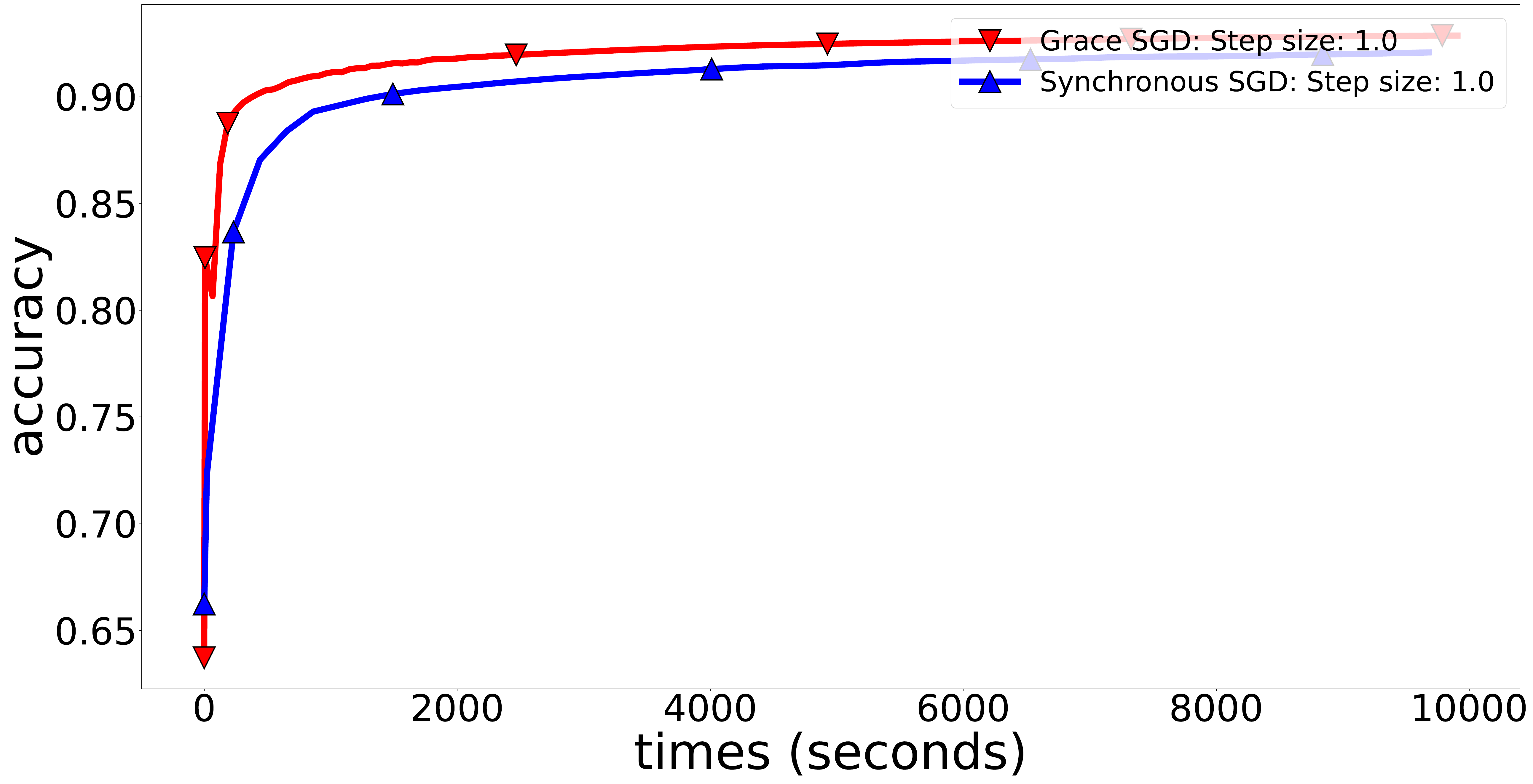}
\caption{Experiments with Synchronous SGD and Grace SGD}
\label{fig:one}
\end{figure}

\textbf{$K$ clusters.} We now consider the practical setup from Figure~\ref{fig:k_clusters_ring} to test the effectiveness of Algorithm~\ref{alg:preprocess} and the usefulness of local training within a single cluster when $b_{\textnormal{slow}}$ is small. We take $n = 100$ and $K = 10.$ For this setup, Algorithm~\ref{alg:preprocess} can either return $S^* = [n]$ (all workers) or $S^* = [n / K]$ (one cluster), depending on the value of $b_{\textnormal{slow}}$ (see the discussion in Section~\ref{sec:examples}). In Figure~\ref{fig:two}, we compare the performance of these two options and observe that $S^* = [n]$ has faster convergence, which concurs with the discussion in Section~\ref{sec:examples}. However, as we start decreasing the bandwidth in Figures~\ref{fig:three} and~\ref{fig:four}, we observe that local training in one cluster is faster in Figure~\ref{fig:four}, which supports the fact that $S^* = [n / K]$, as discussed in Section~\ref{sec:examples}.

\begin{figure}[h]
\centering
\includegraphics[width=0.49\textwidth]{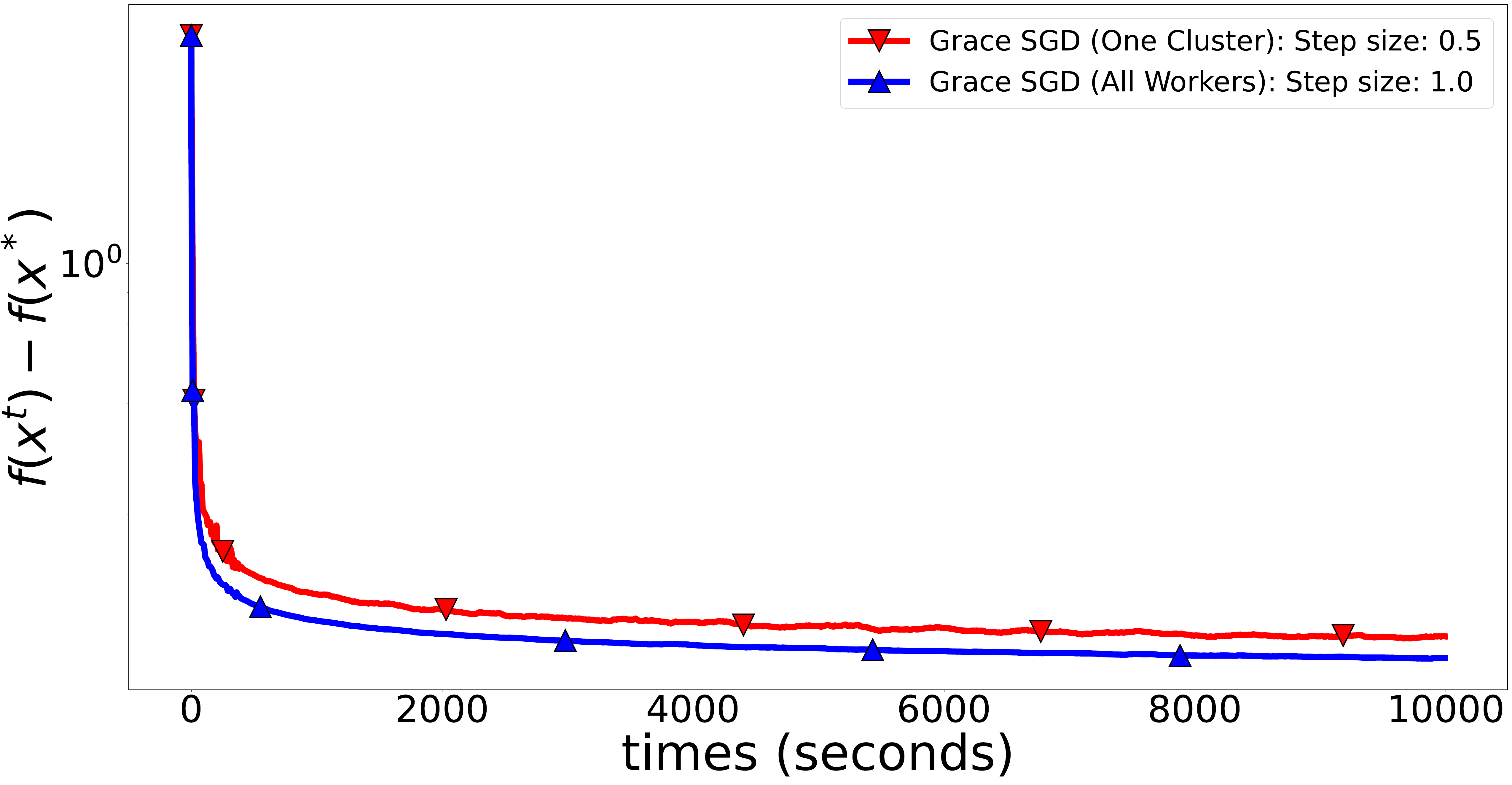}
\hfill
\includegraphics[width=0.49\textwidth]{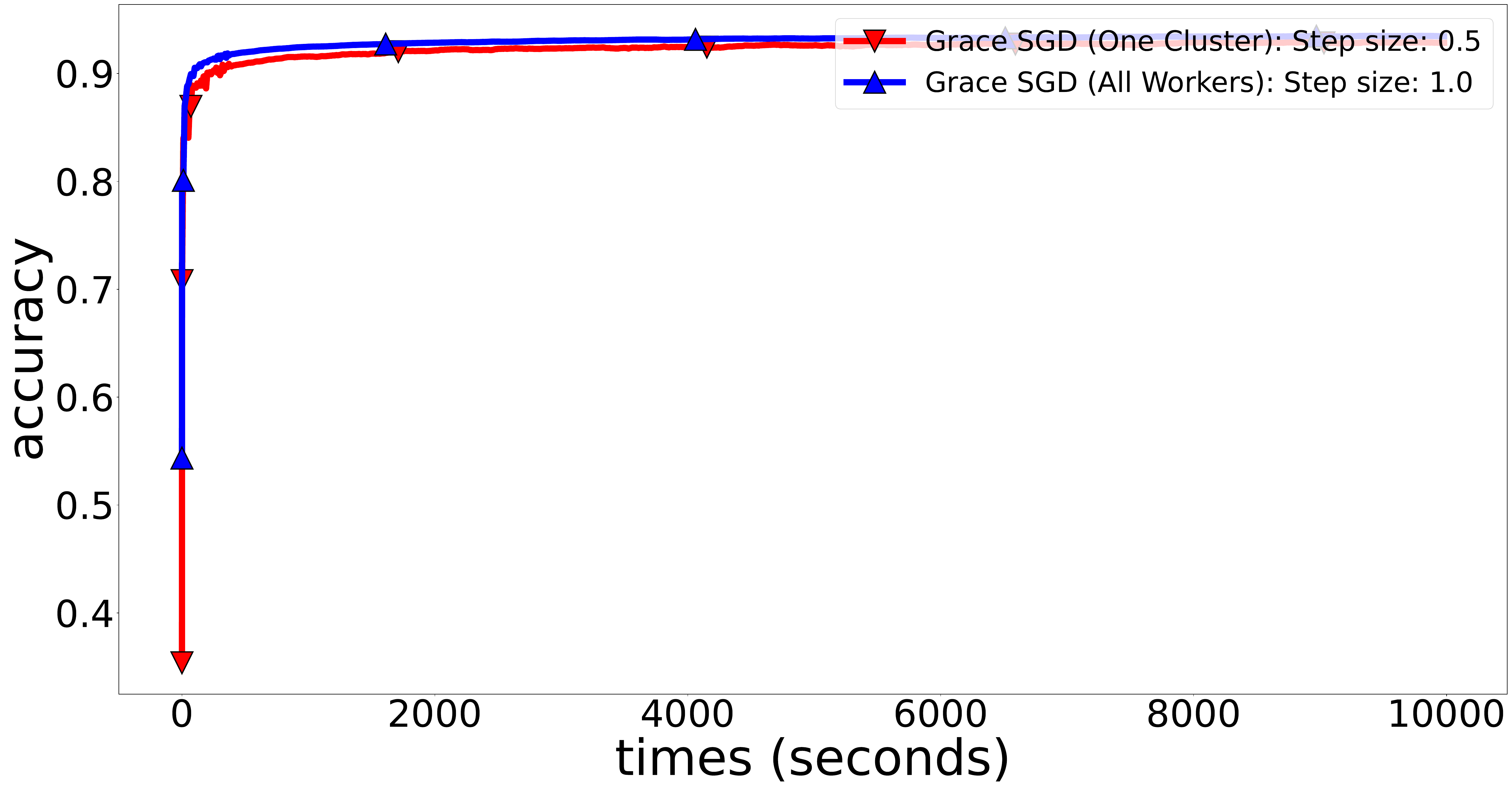}
\caption{Experiments with Grace SGD and $b_{\textnormal{slow}} = \infty$}
\label{fig:two}
\end{figure}

\begin{figure}[h]
\centering
\includegraphics[width=0.49\textwidth]{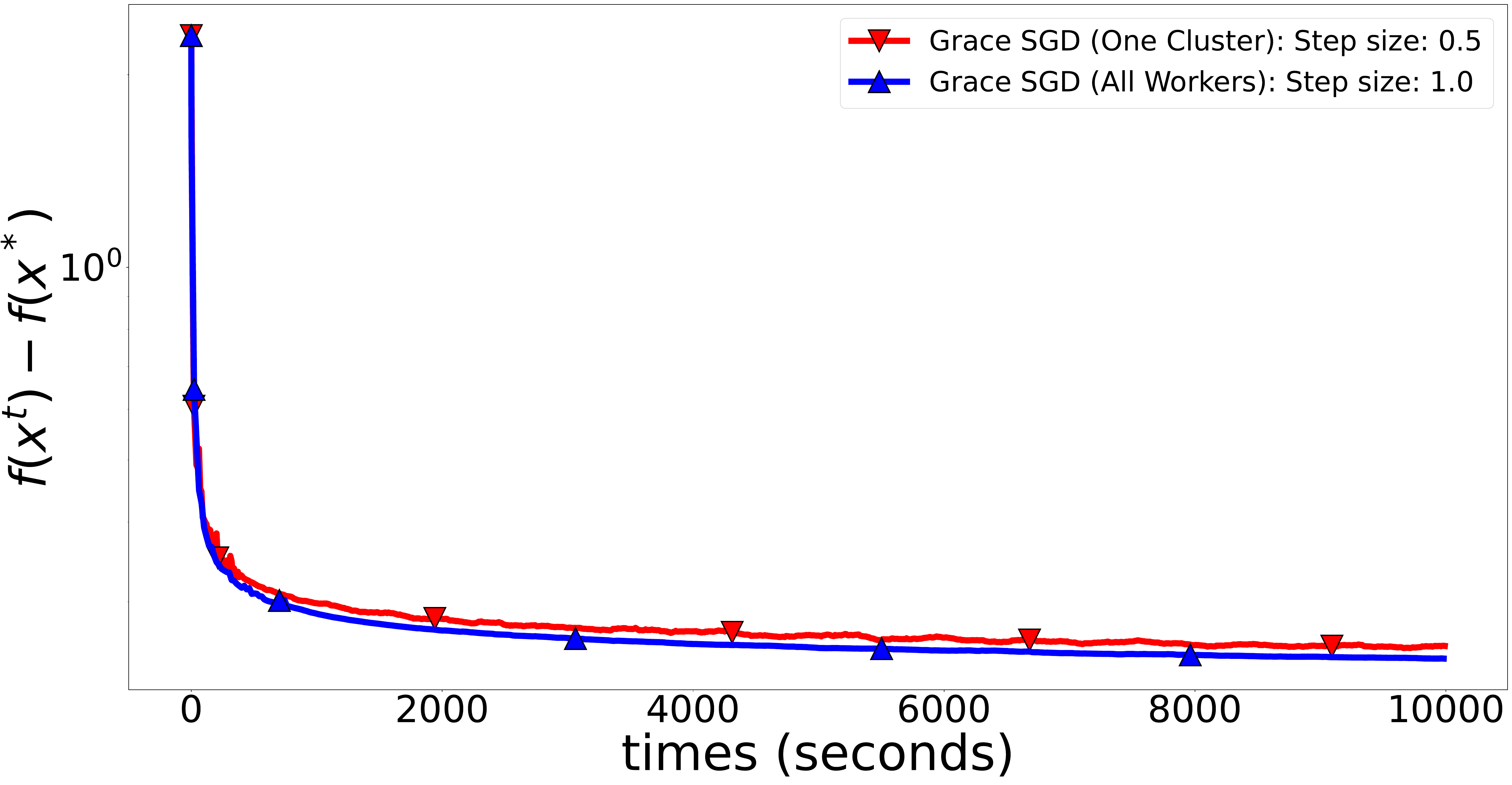}
\hfill
\includegraphics[width=0.49\textwidth]{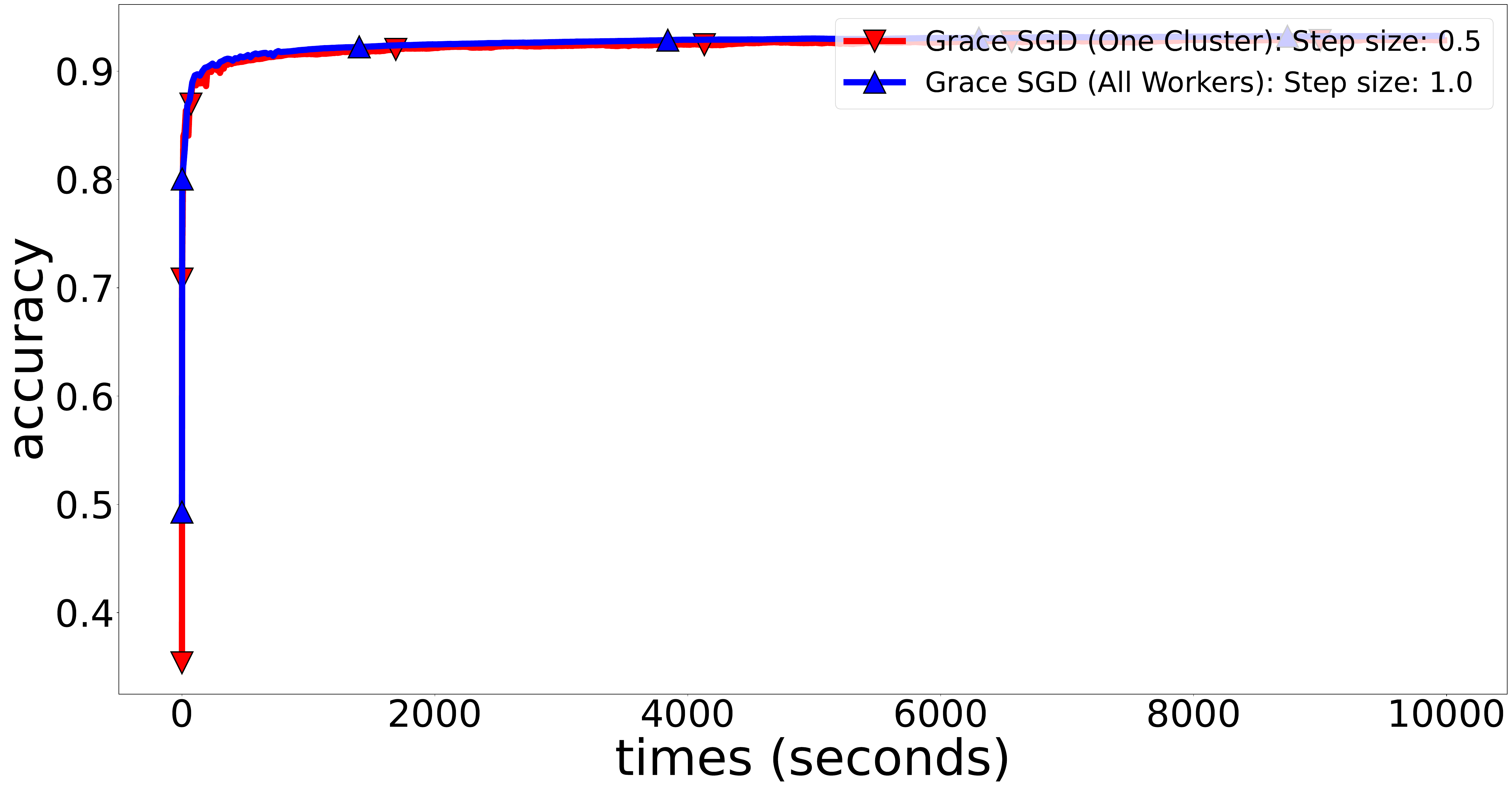}
\caption{Experiments with Grace SGD and $b_{\textnormal{slow}} = 1$}
\label{fig:three}
\end{figure}

\begin{figure}[h]
\centering
\includegraphics[width=0.49\textwidth]{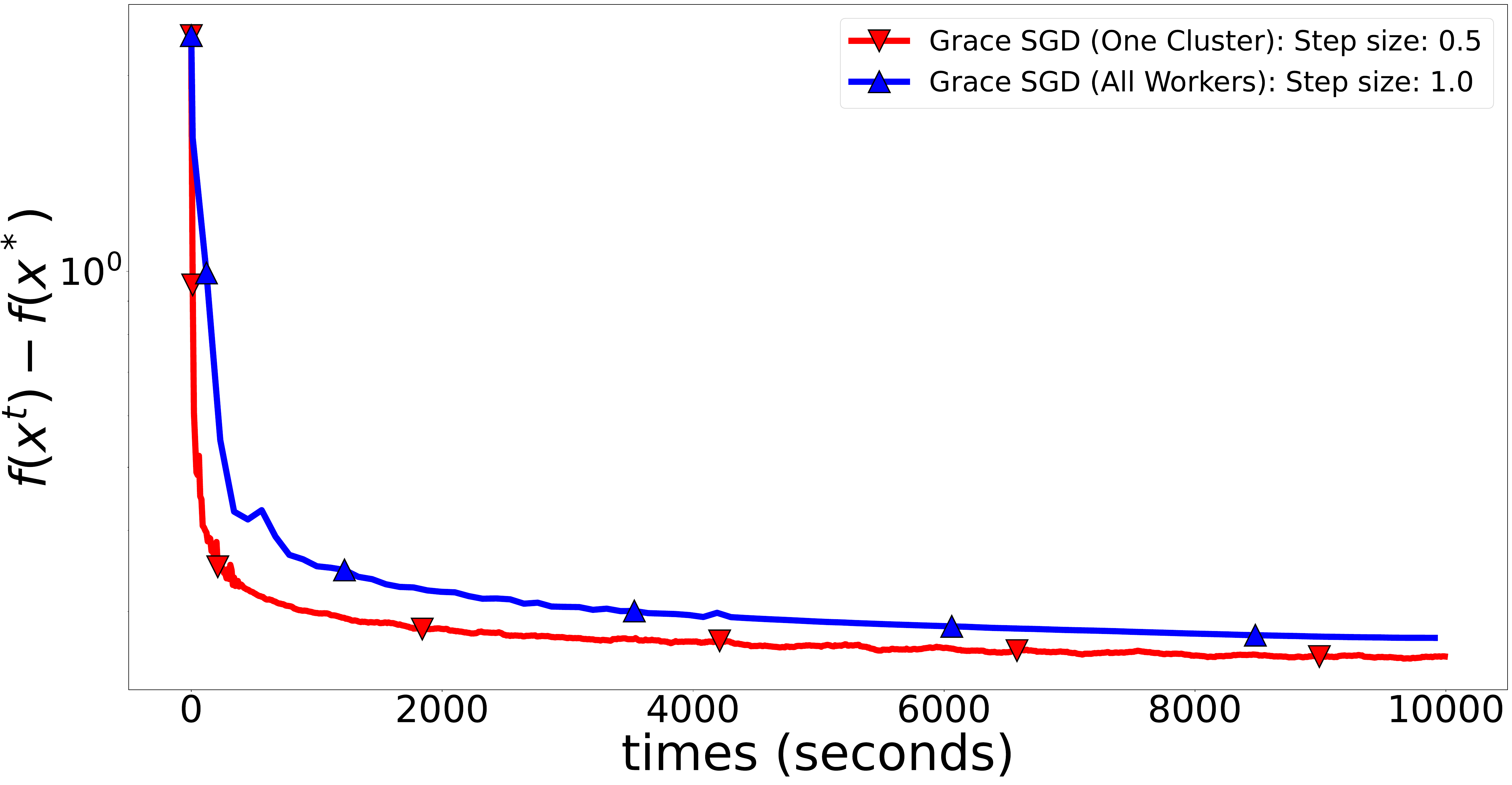}
\hfill
\includegraphics[width=0.49\textwidth]{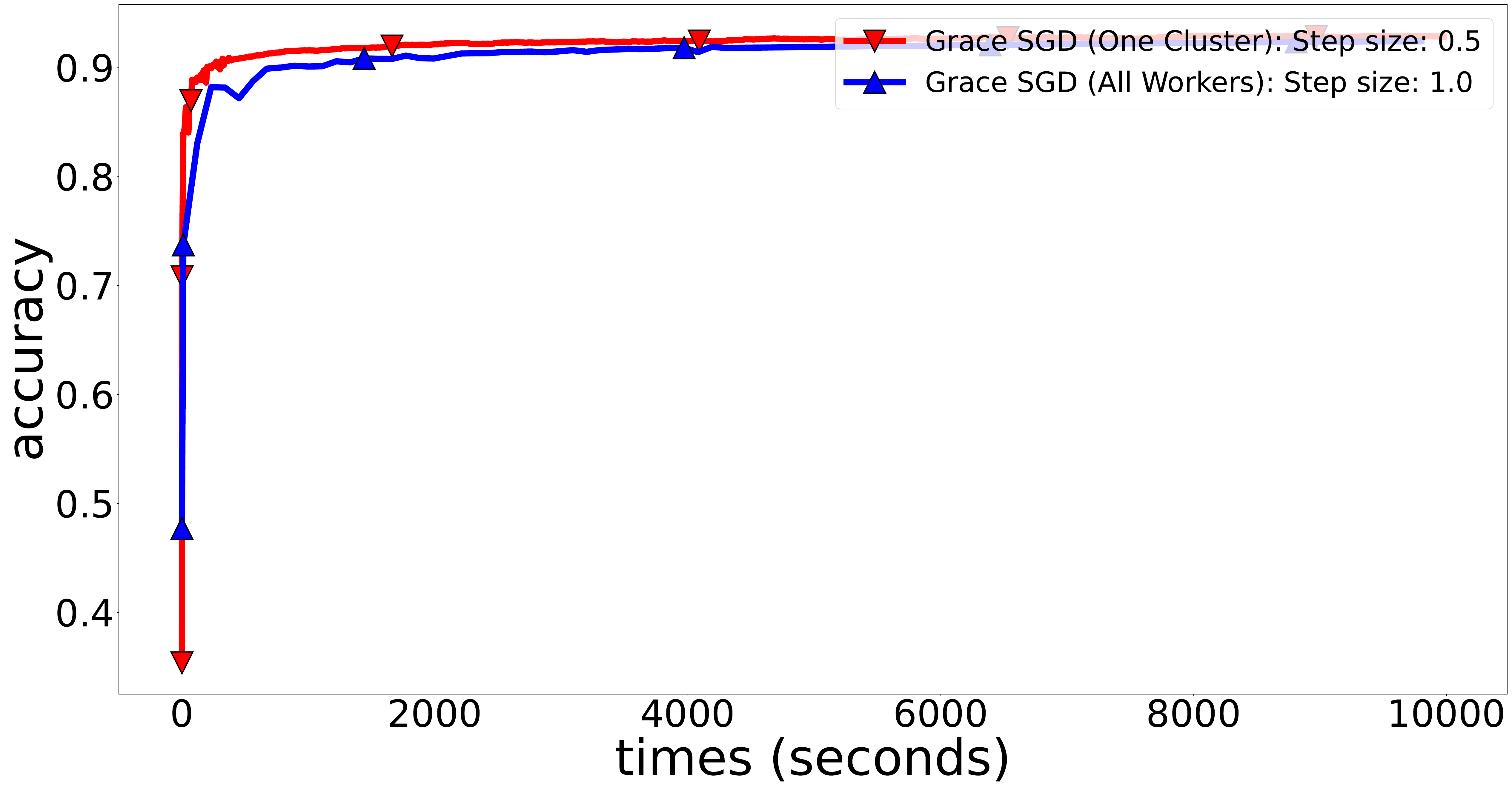}
\caption{Experiments with Grace SGD and $b_{\textnormal{slow}} = 0.1$}
\label{fig:four}
\end{figure}

\clearpage
\section{Graph-Bandwidth Communication Model with Latencies}
\label{sec:latency}
We can extend \hyperref[box:comm_model]{\textcolor{gray}{Graph-Bandwidth Communication Model}} by assuming that each edge has a latency $\ell_{ij} \geq 0,$ and use the standard $\alpha$--$\beta$ model, where transferring $s$ coordinates through edge $(i,j)$ takes $\ell_{ij} + \nicefrac{s}{b_{ij}}$ seconds. Notice that the first term does not depend on the amount of transmitted information $s,$ and since workers typically send large volumes of data, the latency term does not dominate.

Formally, under this model, instead of the main upper bound results in Theorems~\ref{thm:sgd_homog} and \ref{thm:sgd_heter}, we obtain $\eqref{eq:GvDIFIzv} + \cO\left(\nicefrac{\ell_{\max} L \Delta}{\varepsilon}\right)$ and $\eqref{eq:GvDIFIzvheter} + \cO\left(\nicefrac{\ell_{\max} L \Delta}{\varepsilon}\right)$, respectively, where $\ell_{\max} = \max_{(i,j) \in E} \ell_{ij}.$ The only result that needs to be adjusted is Theorem~\ref{thm:allreduce_time}, where the reduce and broadcast operations take $\cO\left(\ell_{\max} + d / p\right)$ seconds instead of $\cO\left(d / p\right)$ (the asymptotic rate remains $\cO\left(d / p\right)$ if $d$ is large). For instance, if $h_i = h$ for all $i \in [n],$ then the time complexity of Grace SGD in the homogeneous case is
\begin{align*}
  \cO\left(\min\left\{\min\limits_{k \in [n - 1]} \left(\ell_{\max} + \frac{d}{\bar{w}_{k}} + \min\limits_{p \in [k]} \frac{h \sigma^2}{\abs{S_{k,p}} \varepsilon} + h\right) \frac{L \Delta}{\varepsilon}\right\}, \frac{h \sigma^2 L \Delta}{\varepsilon^2} + \frac{h L \Delta}{\varepsilon}\right).
\end{align*}
Thus, if $d$ is large, which is the case in modern training, the term involving $\ell_{\max}$ does not dominate, and all our results and conclusions remain valid. For clarity, and because $d$ is large in practice, we ignore latencies in the main part. Nevertheless, deriving tight lower bounds under the model with large latencies is an important direction for future work.

\end{document}